\documentclass[pdflatex,sn-mathphys-num]{sn-jnl}

\usepackage{graphicx}
\usepackage{multirow}
\usepackage{amsmath,amssymb,amsfonts}
\usepackage{amsthm}
\usepackage{mathrsfs}
\usepackage[title]{appendix}
\usepackage{xcolor}
\usepackage{textcomp}
\usepackage{manyfoot}
\usepackage{booktabs}
\usepackage{algorithm}
\usepackage{algorithmicx}
\usepackage{algpseudocode}
\usepackage{listings}

\usepackage{stmaryrd}
\usepackage{bbm} 
\usepackage{esint}
\usepackage{mathtools} 

\theoremstyle{thmstyleone}
\newtheorem{theorem}{Theorem}[section] 
\newtheorem{proposition}[theorem]{Proposition}
\newtheorem{lemma}[theorem]{Lemma} 
\newtheorem{corollary}[theorem]{Corollary} 

\theoremstyle{thmstyletwo}

\newtheorem{remark}{Remark}

\theoremstyle{thmstylethree}
\newtheorem{definition}{Definition} 

\begin{document}

\title[Gradient estimates for $p(\cdot)$-harmonic forms]{Gradient estimates for $p\left(\cdot\right)$-harmonic differential forms}

\author*[1]{\fnm{Anna} \sur{Balci}}\email{anngremlin@gmail.com}
\author[2]{\fnm{Swarnendu} \sur{Sil}}\email{swarnendusil@iisc.ac.in}
\author[3]{\fnm{Mikhail} \sur{Surnachev}}\email{peitsche@yandex.ru}

\affil[1]{\orgname{University Bielefeld}, \orgaddress{\city{Bielefeld}, \postcode{33615}, \country{Germany}}}

\affil[2]{\orgdiv{Department of Mathematics}, \orgname{Indian Institute of Science}, \orgaddress{\city{Bangalore}, \postcode{560012}, \country{India}}}

\affil[3]{\orgname{Keldysh Institute of Applied Mathematics}, \orgaddress{\city{Moscow}, \country{Russia}}}

\abstract{In this paper, we establish gradient bounds for $p(\cdot)$-harmonic differential forms subject to a Coulomb-type gauge condition. For variable exponents satisfying the log-H\"older continuity assumption, we derive higher integrability estimates of Meyers type, ensuring improved regularity beyond the natural energy space. Furthermore, under the stronger assumption of H\"older continuity of the exponent function, we prove that the gradient of solutions exhibits H\"older continuity. These results extend classical regularity theory for constant-exponent $p$-harmonic systems to the variable-exponent setting, which is essential for modeling nonhomogeneous and anisotropic media.}



\maketitle

\section{Introduction}

Let $M$ be an  $n$-dimensional Riemannian manifold $M$ (with boundary). In this paper we are concerned with the interior and boundary regularity of solutions to the following system of non-linear partial differential equations:
\begin{equation}\label{eq1}
d^* \bigl(a(x) \bigl(\mu^2+|\omega|^2)^\frac{p(x)-2}{2} \omega\bigr) = d^* F, \quad d\omega=0,
\end{equation}
where $\mu\geq 0$, on $M$. Here $d$ is the exterior derivative, $d^*$ is the Hodge codifferential, $\omega$ is a vector-valued differential form, the weight $a$ is a measurable bounded nonnegative function separated from $0$, and the variable exponent $p(\cdot)$ is separated from $1$ and $\infty$ and satisfies the log-H\"older condition:
\begin{gather*}
1<p^{-}\leq p(x)\leq p^{+}<\infty \quad \text{for all}\quad x\in M,\\
|p(x)-p(y)|\leq \frac{C}{\log (e + (\mathrm{dist}\,(x,y))^{-1})}.
\end{gather*}
Since we are interested in local properties of solutions, the manifold is assumed to be orientable.

Our results extend the classical work of Zhikov~\cite{Zhikov_higherintegrability,Zhikov_higherintegrability1}, Alkhutov~\cite{Alkhutov_Harnack1997} and Acerbi and Mingione~\cite{Acerbi_Mingione_ARMA_2001}, \cite{Acerbi_Mingione_Crelle_2005}, on the one hand, and Uhlenbeck~\cite{uhlenbecknonlinearelliptic} and Hamburger~\cite{hamburgerregularity}, on the other hand, to the setting of equations for differential forms with variable exponents of nonlinearity. Further $\Lambda^{k}=\Lambda^k(TM,\mathbb{R}^N)$, $k=0,1,\ldots,n$ denotes the bundle of $\mathbb{R}^N$ valued differential forms on $M$, see the notation in Section~\ref{EAS}. As usual, we denote the conjugate exponent to $p$ by $p'$, that is $p'(x)=p(x)/(p(x)-1)$, and $\mathring{M}$ stands for the interior of $M$. The first group of results deals with \textit{local} properties of solutions.

\begin{enumerate}
    \item The solution $\omega$ to \eqref{eq1} enjoys higher integrability estimates of Meyers type: if $F \in L^{q p'(\cdot)}_{\mathrm{loc}}(\mathring{M};\Lambda^{k+1})$, $q>1$, then $|\omega|^{p(x)(1+\sigma)}$ is locally integrable for some $\sigma > 0$.  
    \item If, in addition, the weight function $a(\cdot)$ is continuous, the variable exponent $p(\cdot)$ is vanishing log-H\"older,  and $F$ has the bounded mean oscillation property, then $\omega$ is in the Morrey space 
    $L^{p^{-},\lambda}_{\mathrm{loc}}(\mathring{M},\Lambda^{k+1})$ for any $\lambda < n$.
    \item If, furthermore, $a(\cdot)$, $p(\cdot)$, and $F$ are H\"older continuous in $\mathring{M}$, then so is $\omega$.
\end{enumerate}

The precise statements of the results are given in Theorem~\ref{higher integrability} (higher integrability), Theorem~\ref{MorreyBound} (Morrey bounds), and Theorem~\ref{main theorem} (H\"older continuity of $\omega$). Our proofs follow the approach of~\cite{Acerbi_Mingione_ARMA_2001}, with necessary modifications for the gauge-fixing procedure.

Since locally any closed form is exact, we can write $\omega =du$, so the first-order system \eqref{eq1} can be rewritten as the second-order system 
\begin{equation}\label{eq2}
d^* \bigl(a(x) (\mu^2+|du|^2)^\frac{p(x)-2}{2} du\bigr) = d^* F,
\end{equation}
which is the standard $p$-Laplacian if $u$ is a scalar function, $\mu=0$ and $a, p$ are constant. This explains ``$p(\cdot)$-harmonic'' in the title. If we consider the equation in form \eqref{eq2}, then for $u$ we can infer no regularity results without an additional gauge condition (if $u$  is a solution then the sum of $u$ and any closed form is also a solution to the same system). We impose the Coulomb gauge $d^* u = 0$ and then for $u$ we get the following.
\begin{enumerate}
\item If $F \in L^{q p'(\cdot)}_{\mathrm{loc}}(\mathring{M};\Lambda^{k+1})$, for some $q>1$, then for some $\sigma>0$ we have $u\in W_{\mathrm{loc}}^{1,(1+\sigma)p(\cdot)}(\mathring{M},\Lambda^k)$. 

\item If, in addition, the weight function $a(\cdot)$ is continuous, the variable exponent $p(\cdot)$ is vanishing log-H\"older, and  $F$ has the bounded mean oscillation property, then $\nabla u$ is (locally) in the Morrey space $L^{p^{-},\lambda}$ for any $\lambda < n$, and as a corollary $u$  is H\"older continuous with any exponent less than $1$.

\item  If, further, $a(\cdot)$, $p(\cdot)$, and $F$ are H\"older continuous, then so is $\nabla u$.
\end{enumerate}

The precise statements are given in Theorems~\ref{T1u}, \ref{T2u}, and \ref{T3u}.

We also study solutions satisfying the homogeneous Dirichlet or Neumann condition, namely $t\omega =0$ and $n \bigl(a(x) \bigl(\mu^2+|\omega|^2)^\frac{p(x)-2}{2} \omega\bigr)=0$. Under these conditions our results are valid up to the boundary:
\begin{enumerate}
    \item If $F \in L^{q p'(\cdot)}(M;\Lambda^{k+1})$, $q>1$, then $|\omega|^{p(x)(1+\sigma)}\in L^1(M)$ for some $\sigma > 0$. 
    
    \item If, in addition, the weight function $a(\cdot)$ is continuous, the variable exponent $p(\cdot)$ is vanishing log-H\"older,  and $F$ has the bounded mean oscillation property on $M$ (up to the boundary), then $\omega$ is in the Morrey space 
    $L^{p^{-},\lambda}(M;\Lambda^{k+1})$ for any $\lambda < n$.
    
    \item If, furthermore, $a(\cdot)$, $p(\cdot)$, and $F$ are H\"older continuous on $M$, then so is $\omega$.
\end{enumerate}
The precise statements are given in Theorems~\ref{th:HI_B}, \ref{th:MorreyBoundsDN}, \ref{main theorem global}. The corresponding results for the second-order system \eqref{eq2} are stated in Theorems~\ref{T1uB}, \ref{T2uB}, \ref{T3u_global}.

Classical $p$-energies and corresponding Lebesgue spaces $L^p$ can be generalized by allowing the exponent $p$ to vary with the spatial variable, that is, $p = p(x)$. Interest in these \emph{variable exponent Lebesgue spaces} has grown significantly, driven in part by their role in the modeling of composites \cite{Zhikov1994} and electro-rheological fluids~\cite{Ruzicka_ERF_2000}. These are ``smart'' materials whose viscosity changes in response to an external electric field, a phenomenon that can be described mathematically by making the viscosity depend on a position-dependent exponent. Such materials are used, for example, in the engineering of clutches and shock absorbers. For the introduction to the theory of variable exponent Lebesgue and Sobolev spaces see the excellent treatises \cite{Diening_et_al_variable_exponent}  and \cite{CruFio13}. 

After the classical paper \cite{uhlenbecknonlinearelliptic} and the work \cite{hamburgerregularity}, where the results were transferred to the whole range $1<p=\mathrm{const}<\infty$, for the model system, there was certain gap in systematic results in this direction, for more than two decades. In recent years a growing interest to the analysis of integral functionals, variational relations and systems of partial differential equations with differential forms has reappeared. The authors of \cite{IwaScoStr99}  developed nonlinear Hodge theory on manifolds with boundary, with particular emphasis on the analysis of appropriate Sobolev spaces.  In the last decade one of the authors of the present paper has rekindled the systematic study of regularity for systems with differential forms, see papers of Sil \cite{silthesis,Sil_semicontinuity,Sil_nonlinearStein,Sil_Sengupta_MorreyLorentz,Sil_linearregularity}.  Gaudin~\cite{Gaudin2024Hodge} established Hodge decompositions and maximal $L^{p}$-regularity results for Hodge–Laplacians on the half-space in homogeneous function spaces, and more recently Breit and Gaudin developed optimal regularity theory for the Stokes–Dirichlet problem in mixed Sobolev–Besov settings~\cite{BreitGaudin2025Stokes}.

  We also refer the reader to \cite{Balci_Surnachev_Lavrentiev} for discussions on the Lavrentiev phenomenon in the context of differential forms and to \cite{Balci_Sil_Surnachev_arxiv2025} for existence and regularity results on classical linear boundary problems with differential forms in variable exponent spaces.  

For applications in 3D, of interest (see  \cite{Yin2001,Laforest2018,WanLaforest2020, ChoiLaforest2018, LawLaforest2019} for an interpretation for high-temperature semiconductors) are properties of minimizers of ``partial'' energies of the type 
\begin{equation}\label{funct_3D}
\int\limits_\Omega \biggl[\frac{a(x)}{p(x)} \bigl(\mu^2 + |\mathrm{curl}\, \mathbf{u}|^2\bigr)^\frac{p(x)}{2} - \mathbf{F}\cdot \mathrm{curl}\, \mathbf{u}\biggr]\, dV,
\end{equation}
where $\Omega$ is a domain in $\mathbb{R}^n$ and $dV=dx$, which yields the Euler-Lagrange equation of the form \eqref{eq2}, the so-called $p$-curl (or $p$-curl\,curl) system
\begin{equation}\label{p_curl}
\mathrm{curl}\, \bigl(a(x)(\mu^2+|\mathrm{curl}\,\mathbf{u}|^2)^\frac{p(x)-2}{2} \mathrm{curl}\, \mathbf{u}\bigr) =  \mathrm{curl}\,\mathbf{F}.
\end{equation}
This equation can be interpreted as a stationary equation for the eddy-current problem arising from the Maxwell system under the assumption of the electrical resistivity $\rho = a(x)|\mathrm{curl}\,\mathbf{u}|^{p(x)-2}$ given by the power-law with variable exponent, here $\mathbf{u}$ will be the magnetic field, $\mathbf{J}=\mathrm{curl}\, \mathbf{u}$ will be the current density and $\mathbf{F}$ will represent the effects of the applied external field/current. For equation~\eqref{p_curl}, our main result says that if $a(\cdot)$, $p(\cdot)$, and $F(\cdot)$ are H\"older continuous, then for critical points of the functional \eqref{funct_3D} (for solutions of \eqref{p_curl}), the field $\mathbf{J}=\mathrm{curl}\, \mathbf{u}$ is H\"older continuous and modulo a potential field (or under the divergence-free condition), the vector field $\mathbf{u}$ itself has H\"older continuous derivatives. If no boundary value is imposed this regularity is local, and under the Dirichlet condition $\mathbf{n}\times\mathbf{u} =0$ or the Neumann condition $\mathbf{n}\cdot\mathbf{u}=0$ on $\partial \Omega$ these results are valid up to the boundary. 


The paper is organized as follows. Section~\ref{Preliminaries} introduces algebraic notation and spaces of differential forms.  In Section~\ref{p_Laplacian_for_forms} we define solutions and recall classical Uhlenbeck estimates. Section~\ref{Regularity} contains the proofs of our main results. In Section~\ref{sec:2nd} we give interpretations of these results for the second-order system~\eqref{eq2}. Appendix (Section~\ref{sec:appendix}) contains several facts we use in the proofs.

	\section{Preliminaries}\label{Preliminaries}
    
    \subsection{Notations}
	We record the notations we shall use for the exterior algebra and differential forms. For further details we refer to \cite{CsatoDacKneuss} and \cite{silthesis}.	Let $n \geq 2$, $N \geq 1$ and $0 \leq k \leq n$ be integers.  The integers $n \geq 2$ --- the dimension of the underlying (or base) space, and $N \geq 1$ --- the dimension of the ``target'' space, remain fixed but arbitrary for the rest. Let $M$ be an $n$-dimensional Riemannian manifold (with boundary). Further without loss we assume it to be orientable since we are interested only in local properties of solutions.

We shall work with vector-valued forms. For any integer $0 \leq k \leq n$ --- we denote the bundle of $\mathbb{R}^N$-valued differential forms of degree $k$ by 
		\begin{align*}
			\varLambda^{k}:= \Lambda^{k}(TM,\mathbb{R}^{N}).
		\end{align*} 
		Clearly, $\operatorname{rank}\left(\varLambda^{k} \right)  ={\binom{{n}}{{k}}}\times N$. We use the standard notation (see for instance \cite{CsatoDacKneuss,silthesis})  for the exterior algebra and exterior bundle, for the reader's convenience it is summarized in Section~\ref{EAS}. An $\mathbb{R}^{N}$-valued differential $k$-form $u$ on $M$ is a section of $\Lambda^k$. We assume that its coefficients in any coordinate system are measurable functions. In the framework of this paper a vector-valued differential form can be understood as an $N$-tuple of scalar-valued differential forms. We use the notation $\Lambda^k TM$ to denote the standard bundle of (real-valued) differential forms.  


    We use the following notations for sets. 
	\begin{itemize}
	\item For any Lebesgue measurable subset $A \subset \mathbb{R}^{n}, $ we denote its $n$-dimensional Lebesgue measure  by $\left\lvert A \right\rvert$. 
	\item For any $z \in \mathbb{R}^{n}$ and any $r>0$, the open ball with center $z$ and radius $r$ is denoted by $B_{r}(z) := \left\lbrace x \in \mathbb{R}^{n}: \left\lvert x - z \right\rvert < r\right\rbrace$.  
    \item For $x_0\in \mathbb{R}^n$ and $r>0$ by $Q_r(x_0)$ we denote the open cube in $\mathbb{R}^n$ with side length $2r$ and center $x_0$, with edges parallel to coordinate axes.
    \item When we need not specify the center of a ball or cube we write just $B_r$ or $Q_r$.
    \item By $B^g_r(p)$ we denote the geodesic ball on $M$ of radius $r$ and center $p$.

	\end{itemize}

For the rest, $M$ will always be a $C^{1,1}$ Riemannian manifold with boundary $bM$, so that in particular its metric tensor $g_{ij}$ is Lipschitz. More regularity will be explicitly specified. 

By $dV$ we denote the standard volume form, in coordinates 
$$
dV =\sqrt{g} dx^1\ldots dx^n, \quad g=\mathrm{det}\{g_{ij}\}.
$$
By $d\sigma$ we shall denote the $(n-1)$-dimensional ``surface'' measure on the boundary $b M$.  By $\nu=(\nu_1,\ldots,\nu_n)$ we denote the ``outward unit normal'' 1-form (in the Euclidean case this is identified with the usual outward unit normal). In a boundary patch, if $bM$ is locally given by $x^n=0$ and the image of $M$ corresponds to $x^n>0$, then $\nu = -(g^{nn})^{-1/2} dx^n$. 

By $\mathrm{Lip}(M;\Lambda^k)$ we define the set of Lipschitz-continuous $k$-forms on $M$. The notation $\mathrm{Lip}_c(M;\ldots)$ will be used to denote Lipschitz (functions or forms or vector-valued forms) with support separated from $bM$ (that is, with compact support in the topology of $\mathring{M}$). By $\mathrm{Lip}_T(M;\Lambda^k)$ (resp. $\mathrm{Lip}_N(M;\Lambda^k)$) we denote the set of $k$-forms on $M$ with Lipschitz-continuous coefficients, with vanishing tangential (resp. normal) part on $bM$. That is,  $\omega\in \mathrm{Lip}_T(M;\Lambda^k)$ satisfies $\nu \wedge \omega=0$ on $bM$ and $\omega \in \mathrm{Lip}_N(M;\Lambda^k)$ satisfies $\nu \lrcorner \omega =0$ on $bM$.

\paragraph{Admissible coordinate systems}
(\textit{Admissible coordinate systems}) In an admissible coordinate system the metric tensor on $x^n=0$ (where the image of the boundary is located) satisfies $g_{nj}=0$, $j<n$, and $g_{nn}=1$, that is 
$$
ds^2 =\sum_{1<i,j<n}g_{ij}dx^i dx^j + (dx^n)^2 \quad \text{on}\quad x^n=0.
$$
By \cite[Chapter 7]{Morrey1966}, on any $C^{s,1}$ manifold ($s \ge 1$) with boundary there exist $C^{s,1}$ admissible coordinate systems. In an admissible coordinate system, on the boundary ($x^n=0$) we have
$$
\nu = -dx^n, \quad d\sigma = \sqrt{g}dx^1\ldots dx^{n-1}.
$$

(\textit{Gau\ss' formula})For a Lipschitz vector field $X^j$ the Gau\ss{} theorem holds
$$
\int\limits_M \nabla_j X^j \, dV = \int\limits_{bM} \nu(X)d\sigma,
$$
where $\nabla_j$ are the covariant (Levi-Civita) derivatives. 

(\textit{Sobolev spaces}) For constant $p\in [1,\infty)$ Sobolev spaces $W^{1,p}(M;\Lambda^k)$ are defined in the standard way, by requiring that in any coordinate chart $(U,\varphi)$ all the coordinates of a form belong to $W^{1,p}(\varphi(U))$, with the norm obtained by summing the corresponding norm over all charts of the atlas.

(\textit{Tangential and normal parts}) The forms $\nu \wedge \omega$ and $\nu \lrcorner \omega$ represent the tangential and normal part of $\omega$ on the boundary: $\omega$ splits on the boundary as
    $$
    \omega = t\omega+n\omega, \quad t\omega= \nu \lrcorner (\nu \wedge \omega), \quad n\omega= \nu \wedge (\nu \lrcorner \omega).
    $$
In an admissible local coordinate system, on $x^n=0$ the tangential part $t\omega$ corresponds to the terms without $dx^n$, and the normal part contains the terms with $dx^n$. The tangential part $t\omega=0$ iff $\nu\wedge \omega=0$, and the normal part $n\omega=0$ iff $\nu\lrcorner \omega=0$. 

(\textit{Sobolev spaces with vanishing tangential/normal part}) The spaces $W_{T}^{1,p}\left(M;\varLambda^{k}\right)  $ and $W_{N}^{1,p}\left(  M;\varLambda^{k}\right)  $ are defined as (see for instance \cite{CsatoDacKneuss})
	\begin{align*}
		W_{T}^{1,p}\left(  M;\varLambda^{k}\right)  &=\left\{  \omega\in
		W^{1,p}\left(  \Omega;\varLambda^{k}\right)  :\nu\wedge\omega=0\text{ on
		}\partial\Omega\right\}, \\
		W_{N}^{1,p}\left(  M;\varLambda^{k}\right)  &=\left\{  \omega\in
		W^{1,p}\left(  M;\varLambda^{k}\right)  :\nu\lrcorner\omega=0\text{ on
		}\partial\Omega\right\}.
	\end{align*}
The subspaces $W_{d^{\ast}, T}^{1,p}(M; \varLambda^{k})$ and $W_{d, N}^{1,p}(M; \varLambda^{k})$  are defined as  
	\begin{align*}
		W_{d^{\ast}, T}^{1,p}(M; \varLambda^{k}) &= \left\lbrace \omega \in W_{T}^{1,p}(M; \varLambda^{k}) : d^{\ast}\omega = 0 \text{ in }
		M \right\rbrace, \\
		W_{d, N}^{1,p}(M; \varLambda^{k}) &= \left\lbrace \omega \in W_{N}^{1,p}(M; \varLambda^{k}) : d\omega = 0 \text{ in }
		M\right\rbrace.		
	\end{align*}


(\textit{Weak exterior derivative}) We say that $\varphi\in L^{1}_{\mathrm{loc}}(\mathring{M};\varLambda^{k+1})$ is the weak exterior derivative of $u\in L^{1}_{\mathrm{loc}}\left(\mathring{M};\varLambda^{k}\right)$, denoted by $du$,  if
		$$
		\int\limits_{M} \eta\wedge\varphi=(-1)^{n-k}\int\limits_{M} d\eta\wedge u,
		$$
		for all $\eta\in \mathrm{Lip}_c(\mathring{M},\Lambda^{n-k-1}TM)$.  
        
        (\textit{Weak codifferential}) The Hodge codifferential of an $\mathbb{R}^N$-valued $k$-form is the $\mathbb{R}^N$-valued $(k-1)$-form defines as 
		$$
		d^{\ast}u:=(-1)^{n(k-1)+1} \ast d \ast \omega.
		$$ 
        We say that $u\in L^{1}_{\mathrm{loc}}\left(\mathring{M};\varLambda^{k}\right)$ has weak codifferential $d^* u = \varphi\in L^{1}_{\mathrm{loc}}(\mathring{M};\varLambda^{k-1})$ if and only if 
		$$
		\int\limits_{M} \eta\wedge\varphi=(-1)^{n-k+1}\int\limits_{M} d^*\eta\wedge u,
		$$
        for all $\eta\in \mathrm{Lip}_c(\mathring{M},\Lambda^{n-k+1}TM)$.
        
        (\textit{$d^*$ as adjoint to $d$}) For $u\in L^{1}_{\mathrm{loc}}\left(\mathring{M};\varLambda^{k}\right)$ there holds $du = \varphi$ in the weak sense if and only if 
        \begin{equation}\label{def:du}
        \int\limits_M \langle \varphi,\eta\rangle \, dV =  \int\limits_M \langle u,d^*\eta\rangle \, dV
        \end{equation}
        for all $\eta\in \mathrm{Lip}_c(M;\Lambda^{k+1}TM)$, and $d^* u =\varphi$ in the weak sense if and only if 
        \begin{equation}\label{def:delta u}
        \int\limits_M \langle \varphi,\eta\rangle \, dV = \int\limits_M \langle u,d\eta\rangle \, dV
        \end{equation}
        for all $\eta\in \mathrm{Lip}_c(M;\Lambda^{k-1}TM)$. See \cite{csatothesis, CsatoDacKneuss,silthesis} for the properties and the integration by parts formula regarding these operators. Note that the sign of $d^*$ is different from the one used in \cite{CsatoDacKneuss} for the codifferential and chosen so that $d$ and $d^*$ are formally adjoint.

 (\textit{Boundary data in the weak sense}) If $\varphi \in L^1(M;\Lambda^{k+1})$ is the weak exterior derivative of $u\in L^1(M;\Lambda^k)$ we say that $u$ has zero tangential part on $\partial \Omega$, or $\nu\wedge u=0$, if \eqref{def:du} holds for any $\eta \in \mathrm{Lip}(M; \Lambda^{k+1})$. If  $\varphi \in L^1(M;\Lambda^{k-1})$ is the (weak) codifferential of $u\in L^1(M;\Lambda^k)$ we say that $u$ has zero normal part on $bM$, or $\nu\lrcorner u=0$, if \eqref{def:delta u} holds for any $\eta \in \mathrm{Lip}(M; \Lambda^{k-1})$.



 (\textit{Morrey spaces}) We say that $\omega \in \mathrm{L}^{p,\lambda} (M;\Lambda^k)$ iff 
 $$ 
 \lVert \omega \rVert_{\mathrm{L}^{p,\lambda}\left(M;\varLambda^{k}\right)}^{p} := \sup_{\substack{ x_{0} \in M,\\ \rho >0 }} 
		\rho^{-\lambda} \int\limits_{B^g_{\rho}(x_{0})} \lvert \omega \rvert^{p}\, dV < \infty. 
$$
This is equivalent to saying that in any coordinate chart $(U, \varphi)$ each component of $\omega$ belongs to the Morrey space $\mathrm{L}^{p,\lambda} (\varphi(U))$. We say that  $\omega \in \mathrm{L}^{p,\lambda}_{\mathrm{loc}} (\mathring{M},\Lambda^k)$ if for any indicator function $\chi_K$ of a compact subset $K$ of $\mathring{M}$, $\omega\chi_K \in \mathrm{L}^{p,\lambda} (M;\Lambda^k)$. We recall the definitions and basic properties of Morrey and Campanato spaces on Euclidean domains in Section~\ref{ssec:MC}.

(\textit{Harmonic fields}) The space of tangential and normal harmonic $k$-fields are defined as 
	\begin{align*}
		\mathcal{H}_{T}\left(  M;\varLambda^{k}\right)  &=\left\{  \omega\in
		W_{T}^{1,2}\left(  M;\varLambda^{k}\right)  :d\omega=0\text{ and }%
		d^{\ast}\omega=0\text{ in }M\right\},\\
		\mathcal{H}_{N}\left(  M;\varLambda^{k}\right)  &=\left\{  \omega\in
		W_{N}^{1,2}\left(  M;\varLambda^{k}\right)  :d\omega=0\text{ and }%
		d^{\ast}\omega=0\text{ in }M\right\}.
	\end{align*}

    In this paper we work with variable exponent Lebesgue and Sobolev spaces. 
    
	\subsection{Variable exponent Lebesgue and Sobolev spaces}\label{variable exponent spaces}
	

	 We recall the following definitions (cf. \cite{Diening_et_al_variable_exponent}) concerning our exponent functions. Let  $\mathcal{P}(M)$ denote the set of all Lebesgue measurable functions $p:M \rightarrow [1, \infty)$. For any $p \in \mathcal{P}(M)$, we set 
	\begin{align*}
		p^{-}_{M}:= \operatorname*{ess\, inf}\limits_{y \in M}\  p\left(y\right) \qquad \text{ and } \qquad 
		p^{+}_{M}:= \operatorname*{ess\, sup}\limits_{y \in M}\ p\left(y\right) . 
	\end{align*}
	\begin{definition} \label{def:1}
		A function $p \in \mathcal{P}(M)$ is log-H\"{o}lder continuous in $M$ if there exists a constant $c_{1}>0$, further denoted by $c_{\mathrm{log}}(p)$, such that 
			\begin{align*}
				\left\lvert p(x) - p(y) \right\rvert \leq \frac{c_{1}}{\log \left( e + \frac{1}{\mathrm{dist}\,(x,y)}\right)} \qquad \text{ for all } x, y \in M. 
			\end{align*}
            	\end{definition}

	Now we set 
	\begin{align*}
		\mathcal{P}^{\log} (M) := \left\lbrace p \in \mathcal{P}(M): \frac{1}{p} \text{ is log-H\"{o}lder continuous in }M \right\rbrace.  
	\end{align*}
	If $p^{+}_{M} < \infty$, then $p \in \mathcal{P}^{\log} (M) $ if and only if $p \in \mathcal{P}(M)$ is log-H\"{o}lder continuous in $M$. We shall always assume 
	\begin{align}\label{definition log holder}
		p \in \mathcal{P}^{\log} (M) \text{ and } 1 < p^{-}_M \leq p^{+}_M < \infty. 
	\end{align}
Clearly, \eqref{definition log holder} holds if and only if in any coordinate chart $(U,\varphi)$ the function $p_\varphi:=p\circ \varphi^{-1}$ belongs to $\mathcal{P}^{\log}(\varphi(U))$ and satisfies $1<\inf_{\varphi(U)}p_\varphi \leq \sup_{\varphi(U)}p_\varphi<\infty$. As usual, we denote $p'(x) = \frac{p(x)}{p(x)-1}$.

By $L^{p(\cdot)}(M;\Lambda^k)$ we denote the space of all $\omega\in \Lambda^k$ with measurable components in any coordinate system and the finite Luxemburg norm
$$
\|u\|_{L^{p(\cdot)}(M;\Lambda^k)}= \inf\{\lambda>0\,:\, \rho(u\lambda^{-1})\leq 1\}, \quad \rho(u) = \int\limits_M |u|^{p(x)}\, dV.
$$
If $k=0$ and $N=1$, that is if we work with scalar functions, $\Lambda^k$ is dropped from the notation and we write just $L^{p(\cdot)}(M)$.

By $W^{1,p(\cdot)}(M;\Lambda^k)$ we denote the subspace of $u\in L^{p(\cdot)}(M;\Lambda^k)$ which have in any coordinate chart $(U,\varphi)$ weak (generalized in the sense of S.L.~Sobolev) derivatives $\partial_j u \in L^{p(\cdot)}(\varphi(U); \Lambda^k(\mathbb{R}^n,\mathbb{R}^N))$, endowed with the norm
$$
\|u\|_{W^{1,p(\cdot)}(M;\Lambda^k)} = \|u\|_{L^{p(\cdot)}(M;\Lambda^k)} + \|\nabla u \|_{L^{p(\cdot)}(M;T^*M\otimes\Lambda^k)}.
$$
By the log-H\"older condition on the exponent (see \cite{Zhikov1995} and \cite{Balci_Sil_Surnachev_arxiv2025}) and since the boundary is good enough, $W^{1,p(\cdot)}(M;\Lambda^k)$ coincides with the closure of $\mathrm{Lip}(M;\Lambda^k)$ in this space .

The local versions $L^{p(\cdot)}_{\mathrm{loc}}(\mathring{M};\Lambda^k)$ and $W^{1,p(\cdot)}_{\mathrm{loc}}(\mathring{M};\Lambda^k)$ are defined in the obvious way. 

In $W^{1,p(\cdot)}(M;\Lambda^k)$ we introduce the subspace $W_0^{1,p(\cdot)}(M;\Lambda^k)$  of elements with zero trace on the boundary $W_0^{1,p(\cdot)}(M;\Lambda^k)$ as the closure of $\mathrm{Lip}_c(M;\Lambda^k)$ in $W^{1,p(\cdot)}(M;\Lambda^k)$. Thanks to the log-H\"older condition on the exponent, working in local charts we can easily see that this definition coincides with the one obtained by requiring the vanishing trace on $bM$.


The spaces $W^{1,p(\cdot)}_T(M;\Lambda^k)$, $W^{1,p(\cdot)}_N(M;\Lambda^k)$, $W^{1,p(\cdot)}_{d^*,T}(M;\Lambda^k)$ and $W^{1,p(\cdot)}_{d,N}(M;\Lambda^k)$, are defined similar to the above, by replacing $p$ with $p(\cdot)$ in the definitions.

We introduce the partial Sobolev spaces 
\begin{gather*}
W^{d,p(\cdot)}(M;\Lambda^k)= \{u\in L^{p(\cdot)}(M;\Lambda^k)\,:\, du \in L^{p(\cdot)}(M;\Lambda^k)\},\\
W^{d^*,p(\cdot)}(M;\Lambda^k)= \{u\in L^{p(\cdot)}(M;\Lambda^k)\,:\, d^* u \in L^{p(\cdot)}(M;\Lambda^k)\},
\end{gather*}
and the partial Sobolev spaces with zero tangential/normal part on the boundary:
\begin{gather*}
 W^{d,p(\cdot)}_T(M;\Lambda^k)= \{u\in W^{d,p(\cdot)}(M;\Lambda^k)\,:\  \nu \wedge u=0\},\\
W^{d^*,p(\cdot)}_N(M;\Lambda^k)= \{u\in W^{d,p(\cdot)}(M;\Lambda^k)\,:\ \nu \lrcorner u=0\}.
\end{gather*}
By the log-H\"older property of the variable exponent $p(\cdot)$ and since $d$ is locally a differential operators with constant coefficients, $\mathrm{Lip}(M;\Lambda^k)$ is dense in $W^{d,p(\cdot)}(M;\Lambda^k)$ and $\mathrm{Lip}_c (M;\Lambda_k)$ is dense in $W^{d,p(\cdot)}_T(M;\Lambda^k)$. By the Hodge duality, the same density results are valid for the spaces $W^{d^*,p(\cdot)}(M;\Lambda^k)$ and $W^{d^*,p(\cdot)}_N(M;\Lambda^k)$. See \cite{Balci_Sil_Surnachev_arxiv2025} for details.


 Local versions of the partial Sobolev spaces defined above, namely $W^{d,p(\cdot)}_{\mathrm{loc}}(M;\Lambda^k)$ and $W^{d^*,p(\cdot)}_{\mathrm{loc}}(M;\Lambda^k)$ are defined in the obvious way.

\section{\texorpdfstring{$p$}{p}-Laplacian for differential forms} \label{p_Laplacian_for_forms}
%

In this paper we are concerned with properties of solutions to the quasilinear first order system of partial differential equations 
\begin{equation}\label{main system}
	d^{\ast}\bigl( a(x)  (\mu^2+\lvert \omega \rvert^2)^\frac{p(x)-2}{2}\omega\bigr) = d^{\ast}F, \quad d\omega=0 \quad\text{ in } M,
\end{equation}
where 
\begin{equation}\label{acond}
a\in L^\infty(M), \quad a:M \to [a^-_\Omega,a^+_\Omega], \quad 0<a^-_M\leq a^+_M<\infty,
\end{equation}
and the exponent $p$ satisfies \eqref{definition log holder}. We assume that $0\leq \mu \leq \mu^+$.  Let $k\in \{0,\ldots,n-1\}$.


\begin{definition}\label{def:2}
We say that $\omega \in L^{p(\cdot)}_{\mathrm{loc}}\left(\mathring{M}; \varLambda^{k+1}\right)$ is a \textbf{local weak solution} of \eqref{main system} in $M$ if 
	\begin{align}\label{weak formulation in WTdpx}
		\int\limits_{M} \bigl\langle a(x) (\mu^2+\lvert \omega \rvert^2)^\frac{p(x)-2}{2} \omega, d\phi \bigr\rangle \, dV &= \int\limits_{M}\langle F, d\phi \rangle\, dV \\
        \text{ and }\qquad  \int\limits_M \langle \omega, d^{\ast} \psi\rangle\, dV&=0  \label{weak2}
	\end{align}
    for every  $\phi \in \mathrm{Lip}_c(\mathring{M};\Lambda^k)$ and $\psi \in \mathrm{Lip}_c(\mathring{M};\Lambda^{k+2})$. We say that $\omega$ is a \textbf{weak solution} of \eqref{main system} in $\Omega$ if moreover $\omega\in L^{p(\cdot)}\left(M; \varLambda^{k+1}\right)$.
\end{definition}
 



Since ${\mathrm{Lip}}_c(M;\Lambda^k)$ is dense in $W^{d,p(\cdot)}_T(M; \varLambda^{k})$,  for weak solutions the integral identity \eqref{weak formulation in WTdpx} in Definition~\ref{def:2} is valid for all test forms $\phi \in W_{T}^{d,p(\cdot)}( \Omega; \varLambda^{k})$. For the same reason one can take $\psi \in W^{d^*, p'(\cdot)} _N (\Omega; \Lambda^{k+2})$ in \eqref{weak2}.

We shall also treat the two boundary-value problems associated with the system \eqref{main system}. 
\begin{definition}\label{omega:defD}
  We say that   $\omega \in L^{p(\cdot)}\bigl(M; \varLambda^{k+1}\bigr)$ is a weak solution of \eqref{main system} in $M$ with the Dirichlet boundary condition $\nu \wedge \omega =\nu \wedge \eta$ for a form $\eta \in W^{d,p(\cdot)}\left(M; \varLambda^{k+1}\right)$ if \eqref{weak formulation in WTdpx} holds for any $\phi \in \mathrm{Lip}_T(M;\Lambda^k)$, $d\omega=0$ and $t(\omega-\eta)=0$ in the weak sense. 
\end{definition}
\begin{definition}\label{omega:defN}
  We say that $\omega \in L^{p(\cdot)}\bigl(M; \varLambda^{k+1}\bigr)$ is a weak solution of \eqref{main system} in $M$ with the Neumann boundary condition $\nu\lrcorner \bigl( (\mu^2 +|\omega|^2)^\frac{p(x)-2}{2}\omega\bigr)=0$ if \eqref{weak formulation in WTdpx} holds for any $\phi \in \mathrm{Lip}(M;\Lambda^k)$ and $d\omega=0$ in the weak sense. 
\end{definition}

We have the following fundamental estimate due to Uhlenbeck \cite{uhlenbecknonlinearelliptic} (for $p>2$ and $g_{ij}=\delta_{ij}$) which was extended by Hamburger \cite{hamburgerregularity}(for any $1<p<\infty$ and an arbitrary metric $g_{ij}$ with Lipschitz coefficients). We state it in the form used in this paper. 



\begin{theorem}\label{Uhlenbeck estimate}
Let $\Omega \subset \mathbb{R}^{n}$ be open and let the metric tensor $g_{ij}$ have Lipschitz coefficients on $\Omega$. Let $a>0$, $\mu \geq 0$, and $1<p^{-} \leq p\leq p^{+}< \infty$. Let $\omega \in L^{p}_{\mathrm{loc}}\left( \Omega; \varLambda^{k+1}\right)$ be a local weak solution to 
\begin{equation}\label{eq:UH}
	d^{\ast}\bigl( a (\mu^2 + |\omega|^2)^\frac{p-2}{2}\omega\bigr) =0, \quad d\omega=0 \quad \text{in}\quad \Omega.
\end{equation}
Then $\omega$ is locally H\"{o}lder continuous in $\Omega$ and for any ball $B_{R} \subset \Omega$ we have the estimates 
\begin{align}\label{sup_est_hom}
	\sup\limits_{B_{R/2}} \left\lvert \omega \right\rvert \leq c_{1} \biggl(~ \fint\limits_{B_{R}} (\mu^2+\left\lvert \omega \right\rvert^{2})^\frac{p}{2}\, dV\biggr)^{\frac{1}{p}}
\end{align}
and 
\begin{align}\label{hamburger_osc_use}
     \fint\limits_{B_{\rho}} \big\lvert  \omega  - \left( \omega \right)_{B_{\rho}}\big\rvert^{p}\, dV  \leq c_2  \left(\frac{\rho}{R}\right)^{p\beta} \fint\limits_{B_{R}} (\mu^2+\left\lvert \omega \right\rvert^{2})^\frac{p}{2}\, dV. 
\end{align}
for any $0 < \rho < R/2$, for some constant $c_{1}, c_{2}>0$ and some $0 < \beta < 1$. The constants $c_1$, $c_2$, $\beta$ depend only on $n$, $k$, $p^{-}$, $p^{+}$, $N$.
\end{theorem}
The more customary way is to state the H\"older property in the form
\begin{align}\label{osc_est_hom}
	\sup\limits_{x, y \in B_{\rho}} \left\lvert \omega(x) - \omega(y) \right\rvert \leq c_{2}\left(\frac{\rho}{R}\right)^{\beta} \biggl(~ \fint\limits_{B_{R}} (\mu^2+\left\lvert \omega \right\rvert^{2})^\frac{p}{2}\, dV\biggr)^{\frac{1}{p}},
\end{align}
but the weaker form \eqref{hamburger_osc_use} (equivalent to \eqref{osc_est_hom} by the Campanato characterization) is sufficient. See further comments in Appendix (Section~\ref{append_uhl}).

\begin{remark}
In the theorems of \cite{uhlenbecknonlinearelliptic} and \cite{hamburgerregularity} the exponent $p$ is constant and the constants depend on $n$, $N$, $k$ and $p$, but going through the proofs one can see that the dependence on $p$ is quantitatively controllable and one can take a uniform constant for all $1<p^{-} \leq p \leq p^{+} < \infty$, which depends only on $n$, $N$, $k$, $p^-$, $p^+$ and on the Riemannian manifold $M$.     
\end{remark}

We shall also discuss properties of solutions of the second-order quasilinear system 
\begin{equation}\label{main2}
d^{\ast}\bigl( a(x)  (\mu^2+\lvert du \rvert^2)^\frac{p(x)-2}{2}du\bigr) = d^{\ast}F\quad\text{in}\quad M,
\end{equation}
which is closely related to \eqref{main system}.

\begin{definition}
We say that $u\in W^{d,p(\cdot)}_{\mathrm{loc}}(\mathring{M}; \Lambda^k)$ is a local (weak) solution to the system \eqref{main2} if $\omega=du$ is a local weak solution to \eqref{main system}.
\end{definition}

\begin{definition}
We say that $u\in W^{d,p(\cdot)}(M; \Lambda^k)$ is a (weak) solution to the system \eqref{main2} with the Dirichlet boundary condition $\nu\wedge u=\nu \wedge u_0$ for $u_0\in W^{d,p(\cdot)}(M; \Lambda^k)$  if $\omega=du$ is a weak solution to \eqref{main system} and $\nu\wedge (u-u_0)=0$ in the weak sense.
\end{definition}

\begin{definition}
We say that $u\in W^{d,p(\cdot)}(M; \Lambda^k)$ is a (weak) solution to the system \eqref{main2} with the Neumann boundary condition $\nu \lrcorner \bigl( (\mu^2 +|du|^2)^\frac{p(x)-2}{2}du\bigr)=0$ if $\omega=du$ is a weak solution to \eqref{main system} with the Neumann boundary condition in the sense of Definition~\ref{omega:defN}.
\end{definition}

We shall also need the result of C.~Hamburger~\cite{hamburgerregularity} for the case of the Dirichlet boundary condition. In the following statement, $B_R^+=B_R \cap \{x^n>0\}$ where $B_R$ is a ball of radius $R$ centered on the hyperplane $\{x^n=0\}$.  
\begin{theorem}\label{Hamburger estimate}
Let $\Omega \subset \mathbb{R}^{n}$ be open, $\Omega\subset \{x^n>0\}$, $\Gamma=\partial \Omega \cap \{x^n=0\}$ be a non-empty open subset of $\mathbb{R}^{n-1}$, and let the metric tensor $g_{ij}$ have Lipschitz coefficients in $\Omega\cup \Gamma$. Let $a>0$, $\mu \geq 0$, and $1<p^{-} \leq p\leq p^{+}< \infty$. Let $\omega \in L^{p}_{\mathrm{loc}}\left( \Omega\cup \Gamma; \varLambda^{k+1}\right)$ be a local weak solution to \eqref{eq:UH} satisfying $t\omega=0$ on $\Gamma$. Then $\omega$ is locally H\"{o}lder continuous in $\Omega\cup \Gamma$ and for any ball $B_{R}^+ \subset \Omega$ we have the estimates 
\begin{align}\label{sup_est_homD}
	\sup\limits_{B_{R/2}^+} \left\lvert \omega \right\rvert \leq c_{1} \biggl(~ \fint\limits_{B_{R}^+} (\mu^2+\left\lvert \omega \right\rvert^{2})^\frac{p}{2}\, dV\biggr)^{\frac{1}{p}}
\end{align}
and 
\begin{align}\label{hamburger_osc_useD}
     \fint\limits_{B_{\rho}^+} \big\lvert  \omega  - \left( \omega \right)_{B_{\rho}^+}\big\rvert^{p}\, dV  \leq c_2  \left(\frac{\rho}{R}\right)^{p\beta} \fint\limits_{B_{R}^+} (\mu^2+\left\lvert \omega \right\rvert^{2})^\frac{p}{2}\, dV. 
\end{align}
for any $0 < \rho < R/2$, for some constant $c_{1}, c_{2}>0$ and some $0 < \beta < 1$. The constants $c_1$, $c_2$, $\beta$ depend only on $n$, $k$, $p^{-}$, $p^{+}$, $N$.
\end{theorem}

We state and prove a simple existence result following \cite{silthesis,Sil_semicontinuity} (also \cite{Balci_Surnachev_Lavrentiev,BandDacSil} for the case $N=1$). Let $0\leq k \leq n-1$, $F\in L^{p'(\cdot)}(M;\Lambda^{k+1})$ and introduce the integral functional 
$$
I[u]:=\int\limits_M \left[ \frac{a( x)}{p(x)} (\mu^2+\left\lvert du \right\rvert^2)^{p(x)/2}- \left\langle F, du \right\rangle\right]\, dV
$$
acting on $k$-forms with $du \in L^{p(\cdot)}(M;\Lambda^k)$.

\begin{proposition}\label{existence of minimizers}
	Let $M$ be a compact Riemannian $C^{1,1}$ manifold with boundary. Let the function $a=a(x)$ satisfy \eqref{acond} and the exponent $p=p(x)$ satisfy \eqref{definition log holder}, $F\in L^{p'(\cdot)}\left(M; \varLambda^{k+1}\right)$ and $u_{0} \in W^{1,p(\cdot)}\left(M; \varLambda^{k}\right)$. Then the minimization problem 
\begin{align}\label{minimization}
	m:= \inf \left\lbrace I[u]\, :\ u \in u_{0} + W^{1,p(\cdot)}_{d^{\ast}, T}\left( M; \varLambda^{k}\right)\right\rbrace 
\end{align}
admits a minimizer $\bar{u} \in W^{1,p(\cdot)}\left( M; \varLambda^{k}\right)$, which is a weak solution to the system 
$$
d^{\ast}\left( a(x)(\mu^2+ |d\bar{u}|^2)^\frac{p(x)-2}{2}d\bar{u}\right) = d^{\ast}F \quad \text{and}\quad d^{\ast}\bar{u} = d^{\ast}u_{0}  \quad\text{in}\quad M, 
$$
satisfying the boundary condition $\nu\wedge \bar{u} =\nu\wedge u_{0}$ on $bM$.
This solution (a minimizer of the original variational problem) is unique modulo a Dirichlet harmonic field. The Euler-Lagrange equation 
\begin{equation}\label{eq:EulerLagrange}
\int\limits_M \bigl\langle  a(x)(\mu^2+ |d\bar{u}|^2)^\frac{p(x)-2}{2}d\bar{u}-F, d\xi\bigr \rangle =0
\end{equation}
holds for any $\xi \in W^{1,p(\cdot)}_{T}(M;\Lambda^k)$.
\end{proposition} 
\begin{proof}
The proof is by the direct method in the calculus of variations. Let $u_0+v_j$, $v_j \in  W^{1,p(\cdot)}_{d^{\ast}, T}\left( M; \varLambda^{k}\right)$ be a minimizing sequence for the problem \eqref{minimization}, that is $I[u_0+v_j] \to m$ as $j\to \infty$. Clearly, $m\leq I[u_0]$, and so we can assume without loss that $I[u_0+v_j]\leq I[u_0]$. Then using the Young inequality we easily establish that 
$$
\int\limits_{M} |d v_j|^{p(x)}\, dV \leq C
$$
with $C$ independent of $j$, and thus the sequence $dv_j$ is bounded in $L^{p(\cdot)}(M;\Lambda^{k+1})$. Moreover, by  subtracting the harmonic part (but keeping the same notation $v_j$) we can assume that $(v_j,h_T)=0$ for any $h_T \in \mathcal{H}_T(M)$.

By the Gaffney inequality for variable exponent spaces obtained in \cite{Balci_Sil_Surnachev_arxiv2025} (see Theorem~\ref{divcurl system} in Appendix),
$$
\|v_j\|_{W^{1,p(\cdot)}(M;\Lambda^k)} \leq C \|dv_j\|_{L^{p(\cdot)}(M;\Lambda^k)} \leq C.
$$
Extract from $\{v_j\}$ a weakly convergent subsequence in $W^{1,p(\cdot)}(M;\Lambda^k)$ keeping the same notation $\{v_j\}$. Denote the limit by $v$. Clearly, $d^* v=0$ and $\nu\wedge v=0$, so $v\in W^{1,p(\cdot)}_{d^*,T}(M;\Lambda^k)$. For $x\in \Omega$ and $t\geq 0$ denote
$$
\varphi(x,t) = \frac{a(x)}{p(x)} \left[(\mu^2+t^2)^\frac{p(x)}{2} - \mu^{p(x)}\right].
$$ 
The function $\varphi$ is a generalized uniformly convex $N$-function (see Lemma~\ref{uniform_convexity1} in Appendix). 


On $L^{p(\cdot)}(M;\Lambda^{k+1})$ define the modular
$$
\rho(f) = \int\limits_M \varphi(x,|f(x)|)\,dV
$$
induced by $\varphi$. By \cite[Theorem 2.2.8]{Diening_et_al_variable_exponent} the modular $\rho$ is weakly (sequentially) lower semicontinuous. In this notation,
$$
I[u] = \rho(du) +c_* +  \int\limits_M \langle F,du \rangle\, dV, \quad c_* = \int\limits_M \frac{a(x)}{p(x)}\mu^{p(x)}\, dx,
$$
where $c_*$ is constant. Since $I[u_0+v_j]\to m$, 
$$
I[u_0+v] \leq \liminf_{j\to \infty} I[u_0+v_j] \leq m,
$$ 
and thus $\bar{u} = u_0+v$ is the required minimizer.

By \cite[Theorem 2.4.11]{Diening_et_al_variable_exponent}, the modular $\rho$ is uniformly convex and so by \cite[Lemma 2.4.17]{Diening_et_al_variable_exponent}, we get $\rho((v_j-v))\to 0$. This means that the sequence $v_j$ in fact converges strongly in $W^{1,p(\cdot)}(M;\Lambda^k)$.

 Let us show that indeed $\bar{u}=u_0+v$ is a weak solution of the required system. From the Euler-Lagrange equation we obtain the Euler-Lagrange equation \eqref{eq:EulerLagrange} for any $\xi\in W^{1,p(\cdot)}_{d^*,T}(M;\Lambda^k)$. By Theorem~\ref{divcurl system} in Appendix, for any $\xi \in W^{1,p(\cdot)}_T(M;\Lambda^k)$ there exists $\widetilde{\xi} \in W^{1,p(\cdot)}_{d^*,T}(M;\Lambda^k)$ such that $d\widetilde{\xi}=d\xi$. Therefore, \eqref{eq:EulerLagrange} holds for any $\xi \in W^{1,p(\cdot)}_T(M;\Lambda^k)$.

Now, if $\bar{u}$ and $\bar{v}$ are two solutions, then for the difference $\bar{u}-\bar{v}$ from the Euler-Lagrange equation~\eqref{eq:EulerLagrange} and monotonicity we obtain $d(\bar{u}-\bar{v})=0$. Since also $d^*(\bar{u}-\bar{v})=0$ and $\nu \wedge (\bar{u}-\bar{v})=0$, $\bar{u}-\bar{v}$ is a Dirichlet harmonic field.

\end{proof}

By the standard convexity based arguments, any $\bar{u}\in u_0 + W^{1,p(\cdot)}_{d^*,T}(M;\Lambda^k)$, which satisfies \eqref{eq:EulerLagrange} for any $\xi \in W^{1,p(\cdot)}_T(M;\Lambda^k)$, is a minimizer of the variational problem \eqref{minimization}.

 For the Neumann data a similar statement takes the following form
\begin{proposition}\label{existence of minimizersN}
	Let $M$ be a compact Riemannian $C^{1,1}$ manifold with boundary. Let the function $a=a(x)$ satisfy \eqref{acond} and the exponent $p=p(x)$ satisfy \eqref{definition log holder}, $F\in L^{p'(\cdot)}\left(M; \varLambda^{k+1}\right)$ and $u_{0} \in W^{1,p(\cdot)}\left(M; \varLambda^{k}\right)$. Then the minimization problem 
\begin{align}\label{minimizationN}
	m:= \inf \left\lbrace I[u]\, :\ u \in u_{0} + W^{1,p(\cdot)}_{d^{\ast}, N}\left( M; \varLambda^{k}\right)\right\rbrace 
\end{align}
admits a minimizer $\bar{u} \in W^{1,p(\cdot)}\left( M; \varLambda^{k}\right)$, which is a weak solution to the system 
$$
d^{\ast}\biggl( a(x)(\mu^2+ |d\bar{u}|^2)^\frac{p(x)-2}{2}d\bar{u}\biggr) = d^{\ast}F \quad \text{and}\quad d^{\ast}\bar{u} = d^{\ast}u_{0}  \quad\text{in}\quad M,
$$
satisfying the boundary conditions
$$
\nu\lrcorner \bar{u}=\nu\lrcorner u_0\quad\text{and}\quad \nu\lrcorner \bigl((\mu^2+ |d\bar{u}|^2)^\frac{p(x)-2}{2}d\bar{u} \bigr)=0 \quad\text{on}\quad bM.
$$
This solution (a minimizer to the original variational problem) is unique modulo a Dirichlet harmonic field. The Euler-Lagrange equation \eqref{eq:EulerLagrange} holds for any $\xi \in W^{1,p(\cdot)}(M;\Lambda^k)$. 
\end{proposition} 
\begin{proof}
The existence proof repeats the previous case. Now, in the Euler-Lagrange equation for this minimization problem we take the test forms from the class $W^{1,p(\cdot)}_{d^*,N}(M;\Lambda^k)$. But for any  $\xi \in W^{1,p(\cdot)}(M;\Lambda^k)$ by Theorem~\ref{divcurl system} there exists $\widetilde{\xi} \in W^{1,p(\cdot)}_{d^*,N}(M;\Lambda^k)$ such that $d\widetilde{\xi} = d\xi$. This completes the proof.
\end{proof}

By the standard convexity based argument, any $\bar{u}\in u_0 + W^{1,p(\cdot)}_{d^*,N}(M;\Lambda^k)$, which satisfies \eqref{eq:EulerLagrange} for any $\xi \in W^{1,p(\cdot)}(M;\Lambda^k)$, is a minimizer of the variational problem \eqref{minimizationN}.

\section{Regularity for \texorpdfstring{$p\left(x\right)$}{p(x)}-Laplacian for forms}\label{Regularity}

This is the main section of this paper. In this section we study solutions to the first-order system \eqref{main system}. First, we prove the higher integrability of solutions using the Gehring-type lemma by Giaquinta and Modica. Second, under vanishing log-H\"older condition we prove Morrey bounds for $\omega$ with the help of the iteration lemma by Giaquinta and Giusti. Using these bounds, for H\"older continuous variable exponents we obtain the H\"older continuity of solutions and H\"older property of $\omega$. For the reader's convenience, in Appendix we provide statements of the necessary technical results.   

    In the rest of this section, $n \geq 2$, $N \geq 1$ and $0 \leq k \leq n-1$  are integers, $\Omega$ is a bounded Lipschitz domain in $\mathbb{R}^n$, the weight $a(\cdot)$ satisfies \eqref{acond} and the variable exponent $p=p(\cdot)$ satisfies the log-H\"older condition   \eqref{definition log holder}. The log-H\"older continuity of the exponent is sufficient to establish an analogue of the Meyers property (the higher integrability for the solution $\omega$ of \eqref{eq1}).

    Let $\Theta_p(\cdot)$ denote the modulus of continuity of the function $p$. To prove the Morrey bounds, we shall assume the stronger \textit{vanishing log-H\"older} condition 
	\begin{align}\label{log holder vanishing}
		\lim\limits_{R \rightarrow 0} \Theta_p(R)\log \left(\frac{1}{R}\right) = 0 . 
	\end{align}	
	
    For our main result on the H\"older property for solutions to \eqref{eq1}, we shall assume that $p(\cdot)$ is H\"{o}lder continuous with exponent $\alpha_1\in (0,1)$, i.e. 
	\begin{align}\label{holder exponent}
		\Theta_p(R) \leq C_{H} R^{\alpha_1}. 
	\end{align}

With a slight abuse of notation, for $x\in M$ introduce the functions 
\begin{gather*}
\mathcal{E}: [1,\infty)\times \Lambda^{k+1}(T_x M; \mathbb{R}^N)\to [0,\infty)\\
\mathcal{A}: [1,\infty)\times \Lambda^{k+1}(T_x M; \mathbb{R}^N)\to \Lambda^{k+1}(T_x M; \mathbb{R}^N)
\end{gather*}
as
\begin{equation}\label{Aedef}
\mathcal{E}(q,\eta)= \frac{1}{q} (\mu^2 + |\eta|^2)^{q/2}, \quad \mathcal{A}(q,\eta) = (\mu^2 +|\eta|^2)^\frac{q-2}{2}\eta.
\end{equation}
Formally, these functions depend also on $x\in M$, but we drop this dependance from notation, it is mentioned in $T_x M$. These functions could be also globally defined using the language of sections on bundles. For instance, $\mathcal{E}$ takes a function from $\mathcal{P}(M)$ and a section of $\Lambda^{k+1}$ as its argument and produces a nonnegative function on $M$, while $\mathcal{A}$ takes a function from $\mathcal{P}(M)$ and a section of $\Lambda^{k+1}$ as its argument and produces a section of $\Lambda^{k+1}$. In the context of this paper this would only encumber the notation.

Clearly, $\mathcal{A}(q,\eta) = D_\eta \mathcal{E}(q,\eta)$, and in this notation the system \eqref{main system} becomes
$$
d^* \mathcal{A}(p,\omega)=d^* F, \quad d\omega=0,
$$
the integral identity \eqref{weak formulation in WTdpx} in the definition of solution to \eqref{main system} looks as 
$$
\int\limits_\Omega a(x)\langle\mathcal{A}(p(x),\omega) - F, d\varphi \rangle\, \mathrm{d}V=0, 
$$
for all $\varphi\in \mathrm{Lip}_c(M;\Lambda^{k+1})$ (and so for all $\varphi \in W^{d,p(\cdot)}_T(M;\Lambda^{k+1})$ by closure).

The Uhlenbeck estimates \eqref{sup_est_hom} and \eqref{hamburger_osc_use} for the constant exponent homogeneous case $a=\mathrm{const}$, $p=\mathrm{const}$, $g_{ij}=\mathrm{const}$, $F=0$, take the form
\begin{equation}\label{sup_est_hom1}
	\sup\limits_{B_{R/2}} \left\lvert \omega \right\rvert \leq c_{1} \biggl(~ \fint\limits_{B_{R}} \mathcal{E}(p,\omega)\, dV\biggr)^{\frac{1}{p}},
\end{equation}
which for $\rho \leq R/2$ yields 
\begin{equation}\label{sup_est_hom2}
\int\limits_{B_\rho} \mathcal{E}(p,\omega)\, dV \leq c_1' \left(\frac{\rho}{R} \right)^n \int\limits_{B_R} \mathcal{E}(p,\omega)\, dV,
\end{equation}
and (also for $\rho \leq R/2$)
\begin{equation}\label{hamburger_osc_use1}
     \fint\limits_{B_{\rho}} \big\lvert  \omega  - \left( \omega \right)_{B_{\rho}}\big\rvert^{p}\, dV  \leq c_2  \left(\frac{\rho}{R}\right)^{p\beta} \fint\limits_{B_{R}} \mathcal{E}(p,\omega)\, dV. 
\end{equation}

We shall use the following simple estimate: by convexity of the function $\mathcal{E}(p,\eta)$ in $\eta$, for $p\in [p^{-}_\Omega, p^{+}_{\Omega}]$,
\begin{equation}\label{convex}
\mathcal{E}(p,\eta) \leq  \langle D_\eta \mathcal{E}(p,\eta), \eta \rangle + \mathcal{E}(p,0) \leq \langle \mathcal{A}(p,\eta),\eta \rangle + 1 + \mu^{p^+_\Omega}.
\end{equation}

Further all the proofs will be done in local coordinates.

\subsection{Higher integrability estimates}

We now begin with a crucial higher integrability result, which generalizes the result of Zhikov~\cite{Zhikov_higherintegrability}. 

\begin{theorem}[Higher integrability]\label{th:HIM}
 Let $ F \in L^{q p'(\cdot) }_{\mathrm{loc}} \left(\mathring{M}; \varLambda^{k+1}\right)$, $q>1$. Let $\omega \in L^{p(\cdot)} _{\mathrm{loc}}(\mathring{M}; \varLambda^{k+1})$ be a local weak solution to the system \eqref{main system}. Then there exists $\sigma >0$ such that  $\omega\in  L^{(1+\sigma)p(\cdot)}_{\mathrm{loc}}(\mathring{M};\Lambda^{k+1})$.
\end{theorem}
Since this result is local in nature, it is sufficient to prove it in one coordinate chart. Therefore, we can assume that the $(k+1)$-form $\omega$ is defined in a bounded domain $\Omega\subset \mathbb{R}^n$, with boundary at least Lipschitz, but probably with a non-constant metric tensor $g_{ij}$, and $\omega$ is a local weak solution of \eqref{main system} in $\Omega$. We set 
	\begin{align}\label{energy bound}
		K_{0}[\Omega]:= \int\limits_{\Omega} \left\lvert \omega \right\rvert^{p\left(x\right)}\ \mathrm{d}V +1.
	\end{align}
Let $0\leq \mu \leq \mu^+$. Introduce the set of parameters
$$
data = \{ n, N, k,  p^{-}_\Omega, p^{+}_\Omega, c_{\mathrm{log}}(p), a^-_\Omega, a^+_\Omega\}.
$$

\begin{theorem}[Higher integrability]\label{higher integrability}
    Let $ F \in L^{q p'(\cdot) } \left(\Omega; \varLambda^{k+1}\right)$, $q>1$. Let $\omega \in L^{p\left(\cdot\right)} (\Omega; \varLambda^{k+1})$ be a weak solution to the system \eqref{main system} in $\Omega$.
     Then there exist constants $c = c (\textit{data},\mu^+)>0$ and $\sigma_{0} = \sigma_{0}\bigl(\textit{data}, K_{0}[\Omega], q-1 \bigr) \in (0,1)$ and a radius $R_{0} = R_{0}\bigl( n,K_0[\Omega], c_{\mathrm{log}}(p) \bigr) >0$ such that $\omega\in L^{(1+\sigma)p(\cdot)}_{\mathrm{loc}}(\Omega;\Lambda^{k+1})$ and for any cube $Q_{R}\subset \Omega$ with $0 < R \leq R_{0}$, any Lipschitz co-closed form $\xi$ and any $\sigma \leq \sigma_{0}$, we have 
	\begin{multline}\label{higher integrability estimate}
		\biggl(~\fint\limits_{Q_{R/2}}\left\lvert \omega\right\rvert^{p(x)(1 + \sigma)}\ \mathrm{d}V\biggr)^{\frac{1}{1+\sigma}}
		 \leq c \fint\limits_{Q_{R}}\left\lvert \omega \right\rvert^{p(x)}\ \mathrm{d}V \  + c \biggl( 1+  \fint\limits_{Q_{R}}\left\lvert F - \xi \right\rvert^{(1+\sigma)p'(x) }\ \mathrm{d}V \biggr)^{\frac{1}{1+\sigma}}. 
	\end{multline} 
\end{theorem}
\begin{proof}
    We shall abbreviate $K_0 = K_0[\Omega]$. We assume without loss that $p^{-}_{\Omega} \leq 2n/(2n-1)$ and set 
    $$
    s: = \frac{p^{-}_{\Omega}+1}{2}.
    $$ 
    Choose a number $R_{0} \in (0,1/K_{0})$ small enough such that 
	$$
    \Theta_p \left(2R_{0}\right) \leq \frac{s}{2n}.
    $$
    Let $x_0 \in \Omega$ be such that $Q_{2R}(x_0) \subset \Omega$.
	Note that for 
    $$
    p_{2}:= p^{+}_{B_{R_{0}}\left(x_{0}\right) }\quad \text{and}\quad p_{1}:= p^{-}_{B_{R_{0}}\left(x_{0}\right) }
    $$
    this implies 
	\begin{align*}
		\frac{p_{1}}{p_{2}} \geq 1 - \frac{s}{2n},
	\end{align*}
	We set 
    \begin{align*}
        \left( \frac{p_{1}}{s}\right)^{\ast} := \left\lbrace \begin{aligned}
            &\frac{np_{1}}{ns-p_{1}} &&\text{ if } p_{1} < ns, \\
            &p_{2}+1 &&\text{ if } p_{1} \geq ns.  
        \end{aligned}\right. 
    \end{align*} 
    Then $\left( \frac{p_{1}}{s}\right)^{\ast} \geq p_{2}.$ Indeed, in the case $p_{1} \leq ns,$ this follows from 
    \begin{align*}
         \left( \frac{p_{1}}{s}\right)^{\ast} \geq p_{2}\frac{n\left( 1- \frac{p^{-}_\Omega}{2n}\right)}{np^{-}_\Omega-p^{-}_\Omega} \geq p_{2},
    \end{align*}
    where we have used the bounds $p^{-}_\Omega \leq p_1$ and $s\leq p^{-}_{\Omega} \leq 2n/(2n-1)$. 
    
    From now on, all balls (and cubes) will be centered at $x_{0}$ and we omit writing the center of the balls. For any $0 < R \leq  R_{0}$, $\omega \in L^{p\left(\cdot\right)}\left( B_{R}; \varLambda^{k+1}\right)$ and so by the results of \cite{Balci_Sil_Surnachev_arxiv2025} (see Theorem~\ref{T:dcb} in Appendix), we can find $v \in W^{1, p\left(\cdot\right)}\left( B_{R}; \varLambda^{k}\right)$ such that 
	\begin{align}\label{gauge fixing 1}
		\left\lbrace \begin{aligned}
			dv &=\omega &&\text{ in } B_{R}, \\
			d_0^{\ast}v &= 0 &&\text{ in } B_{R}, \\
			\nu \lrcorner_0 v &= 0 &&\text{ on } \partial B_{R}
		\end{aligned}\right.
	\end{align}
and 
\begin{equation}\label{est0}
\begin{aligned}
\|\nabla v\|_{L^{p_1/s}(B_R;\Lambda^k)} + R^{-1} \|v\|_{L^{p_1/s}(B_R;\Lambda^k)} 
&\leq C \|\omega\|_{L^{p_1/s}(B_R;\Lambda^{k+1})},\\
\|\nabla v\|_{L^{p(\cdot)}(B_R;\Lambda^k)} + R^{-1} \|v\|_{L^{p(\cdot)}(B_R;\Lambda^k)} 
&\leq C \|\omega\|_{L^{p(\cdot)}(B_R;\Lambda^{k+1})}.
\end{aligned}
\end{equation}
Here $C$ depends only on $n$, $N$, $k$, $p^{-}_{\Omega}$, $p^{+}_\Omega$, $c_{\mathrm{log}}(p)$, and $d_0^*$ and $\lrcorner_0$ correspond to the standard Euclidean metric.
	
	
    By the Sobolev or Sobolev-Morrey inequality (for $p_1/s\leq n$ and $p_1/s > n$, respectively) and \eqref{est0} we have ($0<c_1\leq g\leq c_2<\infty$, thus the volume is comparable with the standard Euclidean volume)
    \begin{equation}\label{poincaresobolevdutov}
		\biggl(~ 	\fint\limits_{B_{R}} \left\lvert v \right\rvert^{p_{2}}\, dV\biggr)^{\frac{1}{p_{2}}} \leq c R \biggl(~ 	\fint\limits_{B_{R}} \left\lvert \nabla v\right\rvert^{\frac{p_{1}}{s}}\, dV\biggr)^{\frac{s}{p_{1}}} + c\biggl(~ 	\fint\limits_{B_{R}} \left\lvert v\right\rvert^{\frac{p_{1}}{s}}\, dV\biggr)^{\frac{s}{p_{1}}}  
        \leq  cR \biggl(~ 	\fint\limits_{B_{R}} \left\lvert \omega\right\rvert^{\frac{p_{1}}{s}}\, dV\biggr)^{\frac{s}{p_{1}}}. \notag
	\end{equation}
    
	Now we choose a cut-off function $\eta \in C_{c}^{\infty}\left(B_{R}\right)$ with 
	\begin{align*}
		0 \leq \eta \leq 1 \text{ in  } B_{R},\quad \eta \equiv 1 \text{ in } B_{R/2} \text{ and } \left\lvert d\eta \right\rvert \leq \frac{C}{R}.   
	\end{align*}
	Note that $\phi = \eta^{p_{2}} v  \in W_{0}^{1, p\left(\cdot\right)}\left( B_{R}; \varLambda^{k}\right)$ and thus we can plug $\phi$ as the test function in the weak formulation \eqref{weak formulation in WTdpx} of \eqref{main system}. This yields 
	\begin{equation}\label{eq:p1}
		\int\limits_{B_{R}} \bigl\langle  a(x) \mathcal{A}(p(x),\omega), d\left( \eta^{p_{2}} v  \right) \bigr\rangle \, dV = \int\limits_{B_{R}} \left\langle  F, d\left( \eta^{p_{2}} v  \right) \right\rangle  \, dV
        = \int\limits_{B_{R}}\left\langle F - \xi, d\left(\eta^{p_{2}} v \right) \right\rangle \, dV. 
    \end{equation}
	for any Lipschitz co-closed form $\xi$. Here we have used the integration-by-parts formula: for any form $\psi \in W_T^{1,1}(B_R)$ and a Lipschitz form $\xi$ with $d^*\xi =0$ there holds 
    $$
    \int\limits_{B_R} \langle\xi, d\psi\rangle\, dV = 0.
    $$
     
    Since $dv = \omega$ in $B_{R}$, we have   
	\begin{align*}
		d\left( \eta^{p_{2}} v  \right) = p_{2}\eta^{p_{2}-1}d\eta \wedge v + \eta^{p_{2}}dv = p_{2}\eta^{p_{2}-1}d\eta \wedge v + \eta^{p_{2}}\omega.
	\end{align*}
	Plugging this into \eqref{eq:p1}, we obtain 
	\begin{align*}
		a^-_\Omega\int\limits_{B_{R}}\eta^{p_{2}}\langle \mathcal{A}(p(x),\omega),\omega \rangle \, dV &\leq \int\limits_{B_{R}} \eta^{p_{2}}\bigl\langle   a(x) \langle \mathcal{A}(p(x),\omega),\omega \rangle \, dV \\
        &\begin{multlined}
          =\int\limits_{B_{R}} p_{2}\eta^{p_{2}-1}\langle F - \xi, d\eta \wedge v \rangle \, dV + \int\limits_{B_{R}}\eta^{p_{2}}\langle F - \xi, \omega \rangle\, dV \\
        - \int\limits_{B_{R}}p_{2}\eta^{p_{2}-1}\big\langle a(x)  \mathcal{A}(p(x),\omega),  d\eta \wedge v \big\rangle \, dV 
        \end{multlined} \\
		&:= I_{1} + I_{2} +I_{3}.
	\end{align*}
By convexity (see \eqref{convex}), 
$$
\int\limits_{B_R} \eta^{p_2}\mathcal{E}(p(x),\omega)\, dV \leq c (I_1+I_2+I_3) + c |B_R|(1+\mu^{p^{+}_\Omega}) . 
$$
    
	Now, using Young's inequality with $\varepsilon>0$, we have 
	\begin{align*}
		\left\lvert I_{1} \right\rvert &\leq c\int\limits_{B_{R}} \eta^{\frac{p(x)\left(p_{2}-1\right)}{p(x)-1}} \left\lvert F - \xi \right\rvert^{\frac{p(x)}{p(x)-1}} \, dV+ c \int\limits_{B_{R}}\left\lvert d \eta \right\rvert^{p_{2}}\left\lvert v \right\rvert^{p_{2}} \, dV+ c \left\lvert B_{R} \right\rvert \\
		&\leq c\int\limits_{B_{R}} \left\lvert F - \xi \right\rvert^{\frac{p\left(x\right)}{p(x)-1}}\, dV + \frac{c}{R^{p_{2}}} \int\limits_{B_{R}}\left\lvert v \right\rvert^{p_{2}}\,dV + c \left\lvert B_{R} \right\rvert ,\\
		\left\lvert I_{2} \right\rvert &\leq \varepsilon \int\limits_{B_{R}}\eta^{p_{2}}\left\lvert \omega \right\rvert^{p(x)}\, dV + C_{\varepsilon} \int\limits_{B_{R}} \eta^{p_{2}} \left\lvert F - \xi \right\rvert^{\frac{p(x)}{p(x)-1}}\, dV\\
		&\leq \varepsilon \int\limits_{B_{R}}\eta^{p_{2}}\left\lvert \omega \right\rvert^{p(x)}\,dV + C_{\varepsilon} \int\limits_{B_{R}} \left\lvert F - \xi \right\rvert^{\frac{p(x)}{p(x)-1}}\,dV
        \end{align*}
and 
        \begin{align*}
		\left\lvert I_{3} \right\rvert 		 &\leq a^+_\Omega p^{+}_\Omega\int\limits_{B_{R}}\eta^{p_{2}-1} (\mu^2 + |\omega|^2)^\frac{p(x)-1}{2} \left\lvert  d\eta \right\rvert \left\lvert v \right\rvert \,dV	\\
        &\leq \varepsilon \int\limits_{B_{R}}\eta^{\frac{p\left(x\right)\left(p_{2}-1\right)}{p(x)-1}}(\mu^2+\left\lvert \omega \right\rvert^2)^\frac{p(x)}{2} \, dV+ C_{\varepsilon}\int\limits_{B_{R}}\left\lvert d \eta \right\rvert^{p_{2}}\left\lvert v \right\rvert^{p_{2}}\,dV   + c \left\lvert B_{R} \right\rvert \\
		&\leq \varepsilon \int\limits_{B_{R}}\eta^{p_{2}}(\mu^2+\left\lvert \omega \right\rvert^2)^\frac{p(x)}{2}\,dV + \frac{C_{\varepsilon}}{R^{p_{2}}}\int\limits_{B_{R}}\left\lvert v \right\rvert^{p_{2}} \, dV + c \left\lvert B_{R} \right\rvert, 
	\end{align*}
	where in the last line, we have used the fact that $\frac{p(x)(p_{2}-1)}{p(x)-1} > p_{2}$ and $0 \leq \eta \leq 1$, implying 
    $$
    \eta^{\frac{p(x)(p_{2}-1)}{p(x)-1}} \leq \eta^{p_{2}}.
    $$
    Combining the estimates and choosing $\varepsilon>0$ small enough, we get 
	\begin{multline*}
		\int\limits_{B_{R}}\eta^{p_{2}}(\mu^2+\left\lvert \omega \right\rvert^2)^\frac{p(x)}{2} \, dV  \leq c\int\limits_{B_{R}} \left\lvert F - \xi \right\rvert^{\frac{p(x)}{p(x)-1}} \, dV+ \frac{c}{R^{p_{2}}} \int\limits_{B_{R}}\left\lvert v \right\rvert^{p_{2}} \, dV + c(1+\mu^{p^+_\Omega}) \left\lvert B_{R} \right\rvert.
	\end{multline*}
	Dividing by $\left\lvert B_{R}\right\rvert $ and using \eqref{poincaresobolevdutov}, we arrive at 
	\begin{align*}
		\fint\limits_{B_{R}}\eta^{p_{2}}\left\lvert \omega \right\rvert^{p(x)} \, dV &\leq c \fint\limits_{B_{R}} \left\lvert F - \xi \right\rvert^{\frac{p(x)}{p(x)-1}}\, dV + \frac{c}{R^{p_{2}}} \fint_{B_{R}}\left\lvert v \right\rvert^{p_{2}} \, dV+ c \\
		&\leq  \fint\limits_{B_{R}} \left\lvert F - \xi \right\rvert^{\frac{p(x)}{p(x)-1}}\, dV + c \biggl(~ 	\fint\limits_{B_{R}} \left\lvert \omega\right\rvert^{\frac{p_{1}}{s}}\, dV\biggr)^{\frac{sp_{2}}{p_{1}}} + c.
	\end{align*}
	Since $\eta \equiv 1$ on $B_{R/2},$ this implies 
	\begin{align}\label{prelim rev holder}
		\fint\limits_{B_{R/2}}\left\lvert \omega \right\rvert^{p(x)} \, dV\leq c \biggl(~\fint\limits_{B_{R}} \left\lvert \omega\right\rvert^{\frac{p_{1}}{s}}\, dV\biggr)^{\frac{sp_{2}}{p_{1}}} + c \biggl( 1 +  \fint\limits_{B_{R}} \left\lvert F - \xi \right\rvert^{\frac{p(x)}{p\left(x\right)-1}}\,dV \biggr)
	\end{align}
	By the Young inequality and the relation $ p(x) \geq p_{1}$ in $B_R$, we get
	\begin{align}
		\fint\limits_{B_{R}} \left\lvert \omega\right\rvert^{\frac{p_{1}}{s}}\, dV \leq \fint\limits_{B_{R}} \left\lvert \omega\right\rvert^{\frac{p(x)}{s}}\, dV + \fint\limits_{B_{R}} 1 \, dV= \fint\limits_{B_{R}} \left\lvert \omega\right\rvert^{\frac{p(x)}{s}}\, dV +  1. \label{estimate of du by p}
	\end{align}
	Using \eqref{estimate of du by p} and as $sp_{2}/p_{1} > 1,$ we have   
	\begin{align}\label{estimate of p1bys}
		\biggl(~ 	\fint\limits_{B_{R}} \left\lvert \omega\right\rvert^{\frac{p_{1}}{s}} \, dV \biggr)^{\frac{sp_{2}}{p_{1}}} &\leq \biggl(~ \fint\limits_{B_{R}} \left\lvert \omega\right\rvert^{\frac{p(x)}{s}}\,dV +  1\biggr)^{\frac{sp_{2}}{p_{1}}} \leq c \biggl[ \biggl(~ \fint\limits_{B_{R}} \left\lvert \omega\right\rvert^{\frac{p(x)}{s}}\, dV \biggr)^{\frac{sp_{2}}{p_{1}}} + 1\biggr] \notag\\&=c\biggl(~ \fint\limits_{B_{R}} \left\lvert \omega\right\rvert^{\frac{p(x)}{s}}\,dV \biggr)^{s}\biggl(~ \fint\limits_{B_{R}} \left\lvert \omega\right\rvert^{\frac{p(x)}{s}}\,dV \biggr)^{\frac{s\left( p_{2}-p_{1}\right)}{p_{1}}} + c. 
	\end{align}
	Now we estimate the term 
	$$
    I := \biggl(~ 	\fint\limits_{B_{R}} \left\lvert \omega\right\rvert^{\frac{p(x)}{s}}\, dV\biggr)^{\frac{s\left( p_{2}-p_{1}\right)}{p_{1}}}. 
    $$
	Since $s>1$, by H\"{o}lder inequality, we deduce 
	\begin{align*} 
		I &\leq \biggl(~\fint\limits_{B_{R}} \left\lvert \omega\right\rvert^{p(x)} \,dV\biggr)^{\frac{ p_{2}-p_{1}}{p_{1}}} = cR^{-\frac{n\left( p_{2}-p_{1}\right)}{p_{1}}}\biggl(~\int\limits_{B_{R}} \left\lvert \omega\right\rvert^{p(x)}\,dV \biggr)^{\frac{ p_{2}-p_{1}}{p_{1}}} \\&\leq cR^{-\frac{n\left( p_{2}-p_{1}\right)}{p_{1}}}\biggl(~\int\limits_{B_{R}} \left\lvert \omega\right\rvert^{p(x)}\,dV +1 \biggr)^{\frac{ p_{2}-p_{1}}{p_{1}}} \\&\leq cR^{-2n\Theta_p(R)}\biggl(~\int\limits_{B_{R}} \left\lvert \omega\right\rvert^{p(x)}\,dV +1 \biggr)^{\frac{2\Theta_p(R)}{p_{1}}} \leq c \biggl(~\int\limits_{B_{R}} \left\lvert \omega\right\rvert^{p(x)}\,dV +1 \biggr)^{\frac{2\Theta_p(R_{0})}{p^{-}_{\Omega}}}.
        \end{align*}
	We further evaluate this as 
    \begin{align*}
    I \leq c K_0 ^{2\Theta_p(R_0)} \leq c \exp \left(2c_{\mathrm{log}}(p)\frac{\log K_0}{\log (e+\frac{1}{R_0})} \right)
    \leq  c \exp \left(2c_{\mathrm{log}}(p)\frac{\log K_0}{\log (e+K_0)} \right) \leq c
    \end{align*}
using that $R_0 \leq 1/K_0$. Combining this with \eqref{prelim rev holder}, \eqref{estimate of p1bys} and \eqref{energy bound}, we arrive at 
	\begin{align*} 
		\fint\limits_{B_{R/2}}\left\lvert \omega \right\rvert^{p(x)}\,dV \leq c\biggl(~ \fint\limits_{B_{R}} \left\lvert \omega\right\rvert^{\frac{p(x)}{s}}\,dV \biggr)^{s}  + c \biggl( 1 +  \fint\limits_{B_{R}} \left\lvert F - \xi \right\rvert^{p'(x)}\,dV \biggr). 
	\end{align*}
	Now a Gehring type lemma of Giaquinta and Modica (see Lemma~\ref{lemma:GM} in Appendix, with $q=s$, $g=|\omega|^{p(x)/s}$, $f= (|F-\xi|+1)^{p'(x)/s}$) implies the existence of a constant $\sigma_{0}$ such that for any $0 < \sigma \leq \sigma_{0}$, we have 
	\begin{multline*}
		\biggl(~ \fint\limits_{Q_{R/2}}\left\lvert \omega \right\rvert^{p(x)\left(1 + \sigma\right)}\,dV\biggr)^{\frac{1}{1+\sigma}}
		 \leq c \fint\limits_{Q_{R}}\left\lvert \omega \right\rvert^{p(x)}\, dV  + c \biggl( 1+  \fint\limits_{Q_{R}}\left\lvert F - \xi \right\rvert^{(1+\sigma)p'(x)}\, dV \biggr)^{\frac{1}{1+\sigma}}
	\end{multline*} 
    provided that $Q_R \subset \Omega$ and $0<R \leq R_0$. This completes the proof. 
\end{proof}	

\begin{corollary}
 Let $M_1$ be a contractible $n$-dimensional $C^{1,1}$ submanifold of $M$ (with boundary). Under the hypotheses of Theorem~\ref{th:HIM} there exists $u\in W^{1,(1+\sigma)p(\cdot)}(M_1; \Lambda^k)$ such that $\omega = du$ and $d^* u=0$. 
\end{corollary}
\begin{proof}
We can construct $u$ as a solution to the boundary value problem 
\begin{align*}
    \left\lbrace \begin{aligned}
				du = \omega  \quad &\text{and} \quad  d^{\ast} u = 0 &&\text{ in } M_{1}, \\
				\nu\lrcorner u &= 0 &&\text{  on } bM_{1}.
			\end{aligned} \right. 
\end{align*}
See \cite{Balci_Sil_Surnachev_arxiv2025} and Theorem~\ref{T:dcb} in Appendix.
\end{proof}

\subsection{Higher integrability near the boundary}\label{ssec:hib}

We treat the following two types of boundary value condition: the Dirichlet condition $t\omega=0$ (or $\nu \wedge \omega=0$) and the Neumann condition $n\mathcal{A}(p,\omega)=0$ (or $\nu \lrcorner \mathcal{A}(p,\omega)=0$). We state the following result.

\begin{theorem}[Global higher integrability]\label{th:HI_B}
 Let $ F \in L^{q p'(\cdot) }\left(M; \varLambda^{k+1}\right)$, $q>1$. Let $\omega \in L^{p(\cdot)} (M; \varLambda^{k+1})$ be a  weak solution to the system \eqref{main system} with the Dirichlet condition $t\omega=0$ or the Neumann condition $n ((\mu^2+|\omega|^2)^{(p(x)-2)/2}\omega)=0$. Then there exists $\sigma >0$ such that  $\omega\in  L^{(1+\sigma)p(\cdot)}(\mathrm{M};\Lambda^{k+1})$. The exponent $\sigma$ depends only on $N$, $k$, $p^{-}$, $p^{+}$, $c_{\mathrm{log}}(p)$, $M$, and $q$.
\end{theorem}

The proof is done in an admissible boundary coordinate system. Let $\Omega\subset \mathbb{R}^n\cap\{x^n>0\}$ be a Lipschitz domain, $\Gamma=\mathrm{int}\,(\partial\Omega\cap \{x^n=0\})$ be nonempty, the coordinate system be admissible (that is, $g_{in}=\delta_{in}$ on $\Gamma$), with $g_{ij}$ at least Lipschitz. The notation $K_0[\Omega]$ is inherited from the last section (see \eqref{energy bound}), and we assume that $\mu \in [0,\mu^+]$.

\begin{theorem}[Higher integrability near the boundary]\label{higher integrability boundary}
Let $ F \in L^{q p'(\cdot) } \left(\Omega; \varLambda^{k+1}\right)$, $q>1$. Let $\omega \in L^{p(\cdot)} (\Omega; \varLambda^{k+1})$ be a weak solution to the system \eqref{main system} in $\Omega$ satisfying the Dirichlet or Neumann condition $t\omega=0$ or $n ((\mu^2+|\omega|^2)^{(p(x)-2)/2}\omega)=0$, respectively, on $\Gamma$. Then there exist constants $c = c (\textit{data},\mu^+)>0$ and $\sigma_{0} = \sigma_{0}\bigl(\textit{data}, K_{0}[\Omega], q-1 \bigr) \in (0,1)$ and a radius $R_{0} = R_{0}\bigl( n,K_0[\Omega], c_{\mathrm{log}}(p) \bigr) >0$ such that $\omega\in L^{(1+\sigma)p(\cdot)}_{\mathrm{loc}}(\Omega\cup\Gamma;\Lambda^{k+1})$ and for any cube $Q_{R}$ such that $Q_R \cap \{x^n>0\}\subset \Omega$ with $0 < R \leq R_{0}$, any Lipschitz co-closed form $\xi$ and any $\sigma \leq \sigma_{0}$, we have 
	\begin{multline}\label{eq:hiestB}
		\biggl(~\fint\limits_{Q_{R/4}\cap\Omega}\left\lvert \omega\right\rvert^{p(x)(1 + \sigma)}\ \mathrm{d}V\biggr)^{\frac{1}{1+\sigma}}
		\\ \leq c \fint\limits_{Q_{R}\cap\Omega}\left\lvert \omega \right\rvert^{p(x)}\ \mathrm{d}V \  + c \biggl( 1+  \fint\limits_{Q_{R}\cap\Omega}\left\lvert F - \xi \right\rvert^{(1+\sigma)p'(x) }\ \mathrm{d}V \biggr)^{\frac{1}{1+\sigma}}. 
	\end{multline} 
\end{theorem}

\begin{proof}
I. (\textit{Neumann case}) To pass from local estimates to estimates up the boundary for the Neumann boundary condition $n \mathcal{A}(x,\omega) =0$  we use the same reflection principle as was used by \cite{hamburgerregularity}. Let the reflection operator be defined by $S(x',x^n) =(x',-x^n) $. Then in an admissible boundary coordinate system consider the forms $\tilde \omega$ and $\tilde F$ which coincide with $\omega$ and $F$, respectively, for $x^n>0$ and with $S^* \omega$ and $S^* F$, respectively, for $x^n<0$. We use $\tilde a$ and $\tilde p$ to denote the even extension of $a$ and $p$, $\tilde a(x',-x^n)=a(x',x^n)$, $\tilde p(x',-x^n)=p(x',x^n)$. Extend the metric $g_{ij}$ to $x^n<0$ by setting $g_{ij} = (S^* g)_{ij}$ for $x^n<0$, that is $g_{ij}(x',x^n)=g_{ij}(x',-x^n)$ for $i,j\leq  n-1$ or $i=j=n$, and $g_{in}(x',-x^n)=-g_{in}(x',x^n)$ for $i<n$. The volume form is extended correspondingly. The extension $\tilde \omega$ satisfies $d\tilde \omega =0$ in the ball $B_R$. Indeed, for a smooth form $\xi \in C_0^\infty(B_R)$, recalling that the Jacobian of the mapping $S$ is $-1$, and $(S^*)^2=\mathrm{id}$, we have 
\begin{align*}
\int\limits_{B_R} \tilde\omega \wedge d \xi = \int\limits_{B_R^+} \omega \wedge d \xi + \int\limits_{B_R^{-}}S^* \omega \wedge d \xi = \int\limits_{B_R^+} \omega \wedge d(\xi-S^* \xi ) = 0
\end{align*}
since the tangential part of  $\xi- S^*\xi$ on $x^n=0$ equals zero.

The extension $\tilde \omega$ is a local weak solution to the equation 
$$
d^* ( \tilde a (\mu^2 + |\tilde \omega|^2)^{\frac{\tilde p(x)-2}{2}} \tilde \omega) = d^* \tilde F
$$
in $B_R$. Indeed,
\begin{align*}
\int\limits_{B_R} \langle\tilde a (\mu^2 + |\tilde \omega|^2)^{\frac{\tilde p(x)-2}{2}} \tilde \omega - \tilde F, d\xi \rangle\, dV
=\int\limits_{B_R^+} \langle a (\mu^2 + |\omega|^2)^{\frac{p(x)-2}{2}} \omega -  F, d(\xi + S^* \xi)\rangle\, dV=0
\end{align*}
if the Neumann condition on the current is assumed. Thus in this case in the neighborhood of the boundary we can use the local estimates obtained above. It only remains to note that for any cube $Q_{R/4}$ such that $Q_{R/4}\cap \Omega$ is nonempty, the set  $Q_{R/2}$ belongs to the union of $Q_R \cap (\Omega\cap \Gamma)$ and its reflection to half-space $x^n\leq 0$.

Below we shall use the same argument for the Morrey bounds and H\"older estimates.

II. (\textit{Dirichlet case}) We extend $\omega$ and $F$ by zero to the half-space $\{x^n<0\}$. To prove the reverse H\"older inequality, we consider the following three cases. If the ball $B_{R}$ belongs to $\Omega$ there is nothing to change. If the ball $B_{R/2}$ lies in the half space $x^n<0$ the reverse H\"older inequality is obviously valid with $0$ on the left-hand side. Thus we are left with the case when $B_{R}$ intersects $\{x^n<0\}$ but $B_{R/2}$ intersects $\{x^n>0\}$. The ball $B_R=B_R(x',x^n)$ belongs then to the ball $\widetilde B_{2R}=B_{2R}(x',0)$, which belongs in its turn to the ball $B_{3R} = B_{3R}(x',x^n)$. Using Lemma~\ref{L:Gauge_Tan}, we construct in the half-ball  $\widetilde B_{2R}^+=\widetilde{B}_{2R}\cap \{x^n \geq 0\}$ a form $v$ satisfying 
$$
dv=\omega\quad \text{in}\quad \widetilde B_{2R}^+, \quad tv =0\quad \text{on}\quad x^n=0, 
$$
such that 
$$
\|\nabla v\|_{p(\cdot),B_R} + R^{-1 }\|v\|_{p(\cdot),B_R} \leq C \|\omega\|_{p(\cdot),B_{3R}}.
$$
After that, all the estimates work in the same way as above with the only difference that the reverse H\"older inequality is now obtained in the form 
\begin{align*} 
		\fint\limits_{B_{R/2}}\left\lvert \omega \right\rvert^{p(x)}\,dV \leq c\biggl(~ \fint\limits_{B_{3R}} \left\lvert \omega\right\rvert^{\frac{p(x)}{s}}\,dV \biggr)^{s}  + c \biggl( 1 +  \fint\limits_{B_{3R}} \left\lvert F - \xi \right\rvert^{p'(x)}\,dV \biggr), 
	\end{align*}
with $\omega, F, \xi=0$ in the half-space ${x^n<0}$. The same Giaquinta-Modica lemma of Gehring's type completes the proof.
\end{proof}
\begin{corollary}
 Let $M$ be contractible. Under the hypotheses of Theorem~\ref{th:HI_B} there exists $u\in W^{1,(1+\sigma)p(\cdot)}(M; \Lambda^k)$ such that $\omega = du$ and $d^* u=0$, in the Dirichlet case $tu=0$ and in the Neumann case $nu=0$. 
\end{corollary}

\subsection{Morrey bounds for \texorpdfstring{$\omega$}{omega}.}
In this section we prove the following result.

\begin{theorem}\label{th:MorreyBounds}
Let the variable exponent $p(\cdot)$ satisfy \eqref{definition log holder}, \eqref{log holder vanishing}, and the weight function $a$ satisfy \eqref{acond} and be continuous. Let $F$ be in $BMO (M; \Lambda^{k+1})$ and $\omega \in L^{p(\cdot)}_{\mathrm{loc}}(M;\Lambda^{k+1})$ be a local weak solution of \eqref{main system} in $M$. Then $\omega \in \mathrm{L}^{p^{-}_{M},\lambda}_{\mathrm{loc}}(M;\Lambda^{k+1})$ for any $\lambda\in (0,n)$. If $M_1$ is a $C^{1,1}$ contractible submanifold of $M$ with boundary, $\mathrm{dist}\, (M_1,bM)>0$, there exists $\widetilde{u}\in  W^{1,p(\cdot)}(M_1,\Lambda^k)$ such that $d\widetilde u=\omega$ in $M_1$, $|\nabla \widetilde u|\in \mathrm{L}^{p^-_M,\lambda}(M_1)$ for any $\lambda\in (0,n)$, and $\widetilde u\in C^{0,\gamma}(M_1;\Lambda^k)$ for any $\gamma\in (0,1)$.
\end{theorem}

By $BMO (M;\Lambda^k)$ we mean the set of $k$-forms with bounded mean oscillation, see Section~\ref{ssec:MC}. We are not claiming $F \in BMO$ is sharp for Theorem~\ref{th:MorreyBounds}, but this is a transparent sufficient condition. A more detailed version of this estimate is provided below.


Since our estimates are local in nature, we again work in a coordinate patch $(U,\varphi)$, so we assume that $\omega$ is an $\mathbb{R}^N$-valued form on a domain $\Omega=\varphi(U) \subset \mathbb{R}^n$,  probably with a nonconstant metric with Lipschitz coefficients. Fix $x_{0} \in \Omega$ and choose $\bar{R}>0$ such that $Q_{8\bar{R}}\left(x_{0}\right) \subset \subset \Omega$. Clearly, $\omega$ is a weak solution to \eqref{main system} in $Q_{8\bar{R}}\left(x_{0}\right)$. Replacing $\Omega$ by $Q_{8\bar{R}}\left(x_{0}\right)$ in Theorem \ref{higher integrability}, we determine the constants $R_{0}$, $c$, and $\sigma_{0}$ such that \eqref{higher integrability estimate} holds for $R<R_0$ and any $Q_R \subset Q_{8 \bar{R}}(x_0)$. \smallskip 

        Now we choose $\sigma \leq \min \left\lbrace p^{-}_\Omega -1, \sigma_{0}/2\right\rbrace $ and determine $R_{1}>0$ by the conditions 
		\begin{align*}
			\Theta_p \left(8R_{1}\right) < \sigma /4 \qquad \text{ and } \qquad R_{1} < \min \left\lbrace R_{0}/16, 1/16 \right\rbrace. 
		\end{align*} 
		
        Set 
		\begin{align*}
			p_{\text{max}}:= \max\limits_{x\in \overline{B}_{R_{1}}\left(x_{0}\right) }p\left(x\right).
		\end{align*}
        By our choice of parameters  for $x\in \overline{B}_{R_1}(x_0)$ we have 
        \begin{equation}\label{pest1}
        \begin{aligned}
			p_{\text{max}}\left( 1+ \frac{\sigma}{4} \right) &\leq \left( p(x) + \Theta_p\left(2R_1\right) \right)\left( 1 + \frac{\sigma}{4}\right) \\
            &\leq p(x) \left(1+\frac{\sigma}{4}\right)^2 \leq p(x) (1+\sigma). 
		\end{aligned}
        \end{equation}
Thus, by Theorem \ref{higher integrability}, we have $\omega \in L^{p_{\text{max}}}\left(B_{R_{1}}\left(x_{0}\right) ; \varLambda^{k+1}\right)$ and we can find $\bar{u} \in W^{1, p_{\text{max}}}\left( B_{R_{1}/2}\left(x_{0}\right); \varLambda^{k}\right)$ such that (see Section~\ref{Div_curl_systems})
		\begin{align}\label{gauge fixing 2}
			\left\lbrace \begin{aligned}
				d\bar{u} &=\omega &&\text{ in } B_{R_{1}/2}(x_{0}), \\
				d^{\ast}_0\bar{u} &= 0 &&\text{ in } B_{R_{1}/2}(x_{0}), \\
				\nu \lrcorner_0 \bar{u} &= 0 &&\text{ on } \partial B_{R_{1}/2}(x_{0}).
			\end{aligned}\right.
		\end{align}
Here $d^*_0$ and $\lrcorner_0$ are taken with respect to the standard Eucludian metric.

		Consider a ball $B_{4R}\left(x_{c}\right)\equiv B_{4R} \subset \subset B_{R_{1}/4}\left(x_{0}\right)$ with radius $0 < R < R_{1}/16$ and set 
		\begin{align*}
			p_{1}:= \min_{x\in \overline{B}_{4R}}p\left(x\right) \quad \text{ and } \quad p_{2}:= \max\limits_{x\in \overline{B}_{4R}}p\left(x\right). 
		\end{align*}
		From \eqref{pest1} it follows that 
		\begin{equation}\label{pest2}
			p_{2} \left( 1 + \frac{\sigma}{4}\right)\leq p(x) (1+\sigma) 
		\end{equation}
		for all $x \in \overline{B}_{4R}$.

		
        
        We set 
		\begin{align*}
			K:= 1+\int\limits_{B_{R_{1}/4}\left(x_{0}\right)} \left\lvert \omega \right\rvert^{p_{\text{max}}}\,dV, 
		\end{align*}
        which is finite by Theorem \ref{higher integrability}. We also denote 
		\begin{equation}\label{p2def}
			p_{2}\left(r\right):= \max\limits_{x\in \overline{B}_{4r}}p\left(x\right),
		\end{equation}
		where $B_{4r}\left(x_{c}\right)\equiv B_{4r} \subset \subset B_{R_{1}/32}\left(x_{0}\right)$.

        For a weight function $a(\cdot)$ by $\Theta_a$ we denote its modulus of continuity. Introduce the set of parameters
        $$
        data_1 =\{n, N, k, p^{-}_\Omega, p^{+}_\Omega, a^-_\Omega, a^+_\Omega, c_{\mathrm{log}}(p), \Theta_{a}, K_{0}[\Omega], K\}.
        $$
		In this step we prove 
		
		\begin{theorem}\label{MorreyBound}
        Let the variable exponent $p(\cdot)$ satisfy \eqref{definition log holder}, \eqref{log holder vanishing}, and the weight function $a$ satisfy\eqref{acond} and be continuous in $\Omega$. Let $F$ be in $BMO (B_{R_1}(x_0), \Lambda^{k+1})$ and $\omega \in L^{p(\cdot)}(\Omega;\Lambda^{k+1})$ be a weak solution of \eqref{main system} in $\Omega$. For any $0 < \tau <n$ there exists a choice of a positive radius $\widetilde{R}\in (0, R_1/16)$, depending only on $data_1$, and a positive constant $c$, depending on $\tau$, $data_1$, $\mu^+$, and the norm of $F$ in $BMO (B_{R_1}(x_0), \Lambda^{k+1})$, such that 
			\begin{align*}
				\int\limits_{B_{r}\left(x_{c}\right)} \left\lvert \omega \right\rvert^{p_{2}\left(r\right)}\,dV \leq c r^{n-\tau},
			\end{align*}
			whenever $0 < r <  \widetilde{R}/\left(64\right)^{2}$ and  $x_{c} \in \overline{B}_{\widetilde{R}}\left(x_{0}\right)$ .
		\end{theorem} 
\begin{proof}

       By our choice of $0<\widetilde R < R_1/16$, for any $0<R< \widetilde{R}/32 $ and any $x_{c} \in \overline{B}_{\widetilde{R}}\left(x_{0}\right)$ we have $B_{4R}\left(x_{c}\right)\equiv B_{4R} \subset \subset B_{R_{1}/4}\left(x_{0}\right)$. 

Recall the notation \eqref{Aedef}. By \eqref{weak formulation in WTdpx}, we have 
		\begin{align}\label{weak formulation bar u in WT1p in BR}
			\int\limits_{\Omega} \bigl\langle a(x)\mathcal{A}(p,\omega), d\phi \bigr\rangle\,dV = \int\limits_{\Omega}\left\langle F, d\phi \right\rangle \,dV 
				\end{align}
 for every  $\phi \in W_{T}^{d,p\left(\cdot\right)}\left( \Omega; \varLambda^{k}\right)$ and so for every $\phi \in W_{T}^{d,p\left(\cdot\right)}\left( B_{R}; \varLambda^{k}\right)$. 

		Let $x_{\text{max}} \in \overline{B}_{4R}$ be such that $ p\left(x_{\text{max}}\right) = p_{2}:=p_2(R)$. Since $\bar{u} \in W^{1,p_{\text{max}}} \left( B_{R}; \varLambda^{k}\right)\subset  W^{1,p_{2}}\left( B_{R}; \varLambda^{k}\right)$, by Proposition \ref{existence of minimizers}, we find $v \in W^{1, p_{2}}\left( B_{R}; \varLambda^{k}\right)$ such that
		\begin{align}\label{frozen system for comparison}
			\left\lbrace \begin{aligned}
				d^{\ast}\bigl( a(x_{\text{max}}) \mathcal{A}(p_2, dv) \bigr)  &=0 &&\text{ in }  B_{R},\\ 
				d^{\ast}v &= 0 &&\text{ in }  B_{R}, \\
				\nu\wedge v &= \nu \wedge \bar{u} &&\text{ on } \partial B_{R}. 
			\end{aligned}\right. 
		\end{align}
		It is quite possible that $x_{\text{max}} \notin B_{R}$, but this would not be a problem. Note that the form $\bar{u}-v$ is in $ W_{d^{\ast}, T}^{1,p_{2}}\left( B_{R}; \varLambda^{k}\right)$ and consequently, also in $W_{d^{\ast}, T}^{1,p\left(\cdot\right)}\left( B_{R}; \varLambda^{k}\right)$. Since $v$ is a unique minimizer of the functional 
		\begin{align*}
			v \mapsto \int\limits_{B_{R}} a( x_{\text{max}} ) \mathcal{E}(p_2,dv) \, dV \quad \text{ in } \bar{u} + W_{d^{\ast}, T}^{1,p_{2}}\left( B_{R}; \varLambda^{k}\right),  
		\end{align*}
		by minimality we have 
        \begin{equation}\label{minimality estimate}
		\begin{aligned}
			\int\limits_{B_{R}} \mathcal{E}(p_2,dv) \,dV \leq \int\limits_{B_{R}} \mathcal{E}(p_2,\omega)\,dV. 
            \end{aligned}
            \end{equation}
In particular, 
\begin{equation}\label{minimality1}
\int\limits_{B_R} |\omega - dv|^{p_2}\, dV \leq c(p^{+}_\Omega)  \int\limits_{B_R} \mathcal{E}(p_2,\omega)\, dV.
\end{equation}
            
		By virtue of \eqref{minimality estimate} and inequality~\eqref{sup_est_hom} in Theorem \ref{Uhlenbeck estimate} (see \eqref{sup_est_hom2} in the present notation) applied to the frozen system \eqref{frozen system for comparison}, for $\rho\leq R/2$ we have the following estimate
		\begin{equation*}
			\int\limits_{B_{\rho}} \mathcal{E}(p_2,dv)\, dV \leq c \left(\frac{\rho}{R} \right)^{n} \fint\limits_{B_{R}}  \mathcal{E}(p_2,dv)\, dV \leq c \left(\frac{\rho}{R} \right)^{n} \fint\limits_{B_{R}}  \mathcal{E}(p_2,\omega)\, dV.
       \end{equation*}
By Lemma~\ref{lemma_algebra} we have 
$$
\mathcal{E}(p_2,\omega) \leq c(p^{-}_\Omega, p^{+}_\Omega) \bigl( \mathcal{E}(p_2,dv) + \langle \mathcal{A}(p_2,\omega) - \mathcal{A}(p_2,dv), \omega-dv \rangle\bigr).
$$
Using this we deduce 
		\begin{align}
			\int\limits_{B_{\rho}} \mathcal{E}(p_2,\omega) \, dV 
                &\leq  c \int\limits_{B_{\rho}} \mathcal{E}(p_2,dv) \, dV 
             + c\int\limits_{B_{\rho}}  \langle \mathcal{A}(p_2,\omega) - \mathcal{A}(p_2,dv), \omega-dv \rangle\, dV \notag\\
                &\leq c \left(\frac{\rho}{R}\right)^{n}\int\limits_{B_{R}} \mathcal{E}(p_2,\omega)\, dV 
            + c\int\limits_{B_R}\langle \mathcal{A}(p_2,\omega) - \mathcal{A}(p_2,dv), \omega-dv \rangle\, dV. \label{decay for Morrey bound}
		\end{align}
        
		Now, we need to estimate the last term on the right of \eqref{decay for Morrey bound}. 
        Plugging $\bar{u} - v$ as a test function in \eqref{weak formulation bar u in WT1p in BR} and integrating by parts, we obtain 
		\begin{align}
			\int\limits_{B_{R}} \bigl\langle a(x) \mathcal{A}(p(x),\omega), \omega - dv \bigr\rangle\, dV &= \int\limits_{B_{R}}\bigl\langle F, \omega - dv\bigr\rangle\, dV \notag\\&= \int\limits_{B_{R}}\bigl\langle F - {\ast^{-1}(\ast F)_{R}}, \omega - dv\bigr\rangle\, dV. \label{testing u -w for F}
		\end{align}
Here  we use that $\omega -dv = d (\bar{u}-v)$, $\nu \wedge (\bar{u}-v)=0$ on $\partial B_R$, and the Hodge dual of a closed form is co-closed, thus
$$
\int\limits_{B_R} \bigl\langle \ast^{-1}(\ast F)_{R}, \omega - dv\bigr\rangle\, dV =0.
$$

		On the other hand, by plugging $v -\bar{u} $ as a test function in the weak formulation for the ``frozen'' system \eqref{frozen system for comparison}, we deduce 
		\begin{align*}
			0 &= \int\limits_{B_{R}}\bigl\langle a(x_{\text{max}}) \mathcal{A}(p_2,dv), dv - \omega \bigr\rangle\, dV \\
			&\begin{multlined}[t]
				= \int\limits_{B_{R}}\left\langle   a(x_{\text{max}})[ \mathcal{A}(p_2,\omega)- \mathcal{A}(p_2,dv) ], \omega - dv \right\rangle\, dV \\ - \int\limits_{B_{R}}\left\langle   a(x_{\text{max}})\mathcal{A}(p_2,\omega), \omega - dv \right\rangle\, dV.
			\end{multlined}
		\end{align*}
		This implies
        \begin{align}\label{estimate 1}
			a^-_\Omega \int\limits_{B_{R}}  \bigl\langle\mathcal{A}(p_2,\omega) &-\mathcal{A}(p_2,dv), \omega-dv \bigr\rangle\, dV \\ 
            &\leq \int\limits_{B_{R}}\bigl\langle   a(x_{\text{max}})\bigl[  \mathcal{A}(p_2,\omega) -\mathcal{A}(p_2,dv) \bigr], \omega - dv \bigr\rangle\, dV \notag 
			\\&= \int\limits_{B_{R}}\bigl\langle   a(x_{\text{max}})\mathcal{A}(p_2,\omega), \omega - dv \bigr\rangle\, dV =I_{1} + I_{2},\notag
		\end{align}
		where 
		\begin{align*}
			I_{1}&:= -\int\limits_{B_{R}}\bigl\langle   a(x)\mathcal{A}(p(x),\omega) - a(x_{\text{max}}) \mathcal{A}(p_2,\omega), \omega - dv \bigr\rangle \, dV\intertext{ and }
			I_{2} &:= \int\limits_{B_{R}}\bigl\langle   a(x) \mathcal{A}(p(x),\omega), \omega - dv \bigr\rangle\, dV. 
		\end{align*}
        
		Using \eqref{testing u -w for F}, we arrive at 
		\begin{align}\label{I2eq}
			I_{2} =  \int\limits_{B_{R}}\bigl\langle   a(x) \mathcal{A}(p(x),\omega), \omega- dv \bigr\rangle\, dV = -\int\limits_{B_{R}}\bigl\langle F - \ast^{-1}(\ast F)_{R}, \omega - dv\bigr\rangle\, dV.
		\end{align}
		For $I_{1}$, we write $I_{1}= I_{11} + I_{12},$ where 
		$$
			I_{11}:= -\int\limits_{B_{R}}\bigl\langle   a(x)\bigl[ \mathcal{A}(p(x),\omega) - \mathcal{A}(p_2,\omega)\bigr], \omega - dv \bigr\rangle\, dV
            $$
            and
            $$
			I_{12} := -\int\limits_{B_{R}}\bigl\langle   \left[ a(x) - a(x_{\text{max}})\right] \mathcal{A}(p_2,\omega), \omega - dv \bigr\rangle\, dV. 
		$$
		Now we estimate each of the terms. Using, \eqref{I2eq}, the H\"older inequality and \eqref{minimality1} we have  
		\begin{align}\label{estimate of F term}
			\left\lvert I_{2}\right\rvert &\leq \int\limits_{B_{R}}\left\lvert \ast F - \left(\ast F\right)_{R}\right\rvert\cdot \left\lvert \omega - dv\right\rvert\, dV \notag\\
			&\leq \biggl(~ \int\limits_{B_{R}} \left\lvert \ast F - \left(\ast F\right)_{R} \right\rvert^{\frac{p_{2}}{p_{2}-1}}\, dV\biggr)^{\frac{p_{2}-1}{p_{2}}}\biggl(~ \int\limits_{B_{R}} \left\lvert \omega - dv \right\rvert^{p_{2}}\, dV\biggr)^{\frac{1}{p_{2}}} \notag\\
			&\leq c \biggl(~ \int\limits_{B_{R}} \left\lvert \ast F - \left(\ast F\right)_{R} \right\rvert^{\frac{p_{2}}{p_{2}-1}}\, dV\biggr)^{\frac{p_{2}-1}{p_{2}}}\biggl(~ \int\limits_{B_{R}}  \mathcal{E}(p_2,\omega) \, dV\biggr)^{\frac{1}{p_{2}}} \notag\\
			&\leq  \varepsilon \int\limits_{B_{R}}  \mathcal{E}(p_2,\omega) \, dV+ C_{\varepsilon} \int\limits_{B_{R}} \left\lvert \ast F - \left(\ast F\right)_{R} \right\rvert^{\frac{p_{2}}{p_{2}-1}}\, dV,
		\end{align}
		for every $\varepsilon>0$, by the Young's inequality. Now we estimate  
		\begin{align}\label{estimate of a term}
			\left\lvert I_{12}\right\rvert &\leq \int\limits_{B_{R}}   \left\lvert  a(x) - a(x_{\text{max}})\right\rvert\cdot (\mu^2+\left\lvert \omega \right\rvert^2)^\frac{p_{2}-1}{2}\cdot\left\lvert \omega - dv \right\rvert\, dV \notag\\
			&\leq \Theta_{a}\left(4R\right)\int\limits_{B_{R}}(\mu^2+\left\lvert \omega \right\rvert^2 )^\frac{p_{2}-1}{2}\cdot\left\lvert \omega - dv\right\rvert\, dV \notag\\
			&\leq \Theta_{a}\left(4R\right) \biggl(~\int\limits_{B_{R}}(\mu^2+\left\lvert \omega \right\rvert^2)^\frac{p_{2}}{2}\, dV\biggr)^{\frac{p_{2}-1}{p_{2}}}\biggl(~\int\limits_{B_{R}} \left\lvert \omega -dv \right\rvert^{p_{2}}\, dV \biggr)^{\frac{1}{p_{2}}} \notag\\
			&\leq c\Theta_{a}(4R)  \int\limits_{B_R}\mathcal{E}(p_2,\omega)\, dV,
		\end{align}
        where we again use \eqref{minimality1}.

        
		Finally, we estimate the most difficult term $I_{11}$. By the Newton-Leibniz formula, 
        \begin{align*}
        \mathcal{A}(q_1,\eta)- \mathcal{A}(q_2,\eta) &= \int\limits_{q_1}^{q_2} (\mu^2 + |\eta|^2)^\frac{s-2}{2}\eta  \log (\mu^2 +|\eta|^2)^{1/2} \, ds \\
        &\leq |q_1-q_2|\cdot \max ((\mu^2 +|\eta|^2)^\frac{q_1-1}{2},(\mu^2 +|\eta|^2)^\frac{q_2-1}{2}) \cdot \log (\mu^2 +|\eta|^2)^{1/2}.
        \end{align*}
         Denote
            $$
            g(x,t) = \max \bigl((\mu^2+t^2)^\frac{p_{2}-1}{2}, (\mu^2+t^2)^\frac{p(x)-1}{2} \bigr) \cdot\left\lvert \log   (\mu^2+t^2)^{1/2} \right\rvert.
            $$
        We have 
		\begin{align*}
			\left\lvert I_{11}\right\rvert &\leq a^+_\Omega\int\limits_{B_{R}}\left\lvert  \mathcal{A}(p(x),\omega) -  \mathcal{A}(p_2,\omega) \right\rvert \cdot \left\lvert \omega - dv\right\rvert\, dV \notag \\
			&\leq a^+_\Omega\Theta_p(4R)\int\limits_{B_{R}} g(x,|\omega|)  \left\lvert \omega - dv\right\rvert\, dV \notag \\
            &\begin{aligned}
                = a^+_\Omega\Theta_p(4R)  \int\limits_{B_R \cap \{ |\omega|\geq 1+\mu\}} & g(x,|\omega|)  \left\lvert \omega - dv\right\rvert \, dV\\
            &+ a^+_\Omega\Theta_p(4R)\int\limits_{B_R \cap \{ |\omega|\leq 1+\mu\}}g(x,|\omega|)  \left\lvert \omega - dv\right\rvert\, dV.
            \end{aligned}
            \end{align*}
Since for $t\in (0,2(1+\mu))$ there holds
$$
\max (t^{p(x)-1}, t^{p_2-1}) |\log t| \leq \max (t^{p^{-}_\Omega-1},t^{p^{+}_\Omega-1}) |\log t| \leq C (p^{-}_\Omega,p^{+}_\Omega,\mu), 
$$
we have
\begin{align*}
a^+_\Omega\Theta_p(4R) &\int\limits_{B_R \cap \{ |\omega|\leq 1+\mu\}} g(x,|\omega|)  \left\lvert \omega - dv\right\rvert\, dV\\ 
&\leq c a^+_\Omega\Theta_p(4R)\int\limits_{B_R} \left\lvert \omega - dv\right\rvert\, dV  \leq ca^+_\Omega\Theta_p(4R) \biggl(~\int\limits_{B_R} \left\lvert \omega - dv\right\rvert^{p_2}\,dV + R^n \biggr)\\
&\leq ca^+_\Omega\Theta_p(4R) \biggl(~\int\limits_{B_R} \mathcal{E}(p_2,\omega)\,dV + R^n \biggr)
\end{align*}
by \eqref{minimality1}.


For the integral over the set $B_R \cap \{|\omega|\geq 1+\mu\}$, by the Young inequality we further have 
            \begin{align}
            \mathfrak{w}&:=a^+_\Omega\Theta_p(4R) \int\limits_{B_R \cap \{ |\omega|\geq 1+\mu\}} g(x,|\omega|)  \left\lvert \omega - dv\right\rvert\, dV \notag\\ 
			&\begin{multlined}[t]
			   \leq a^+_\Omega \Theta_p(4R)\log\left(\frac{1}{R}\right)\int\limits_{B_{R}} |\omega - dv|^{p_2}\, dV\\
            + a^+_\Omega \Theta_p(4R)\log\left(\frac{1}{R}\right)\int\limits_{B_{R}\cap \{ |\omega|\geq 1+\mu\}}(\mu^2+|\omega|^2)^\frac{p_2}{2} \left(\frac{\log (\mu^2+|\omega|^2)^{1/2}}{\log (1/R)}\right)^\frac{p_2}{p_2-1}\,dV . 
			\end{multlined} 
			   \label{frak_w}
            \end{align}
For $1+\mu \leq |\omega| \leq R^{-n/p_1}$ there holds 
\begin{equation}\label{te0}
\left(\frac{\log (\mu^2+|\omega|^2)^{1/2}}{\log (1/R)}\right)^\frac{p_2}{p_2-1}\leq c(n,p^{-}_\Omega,p^{+}_\Omega),
\end{equation}
For $t\geq R^{-n/p_1}$ the function 
$$
f(t) = t^{p_2 - (1+\sigma)p(x)} \biggl(\frac{\log t}{\log(1/R)} \biggr)^\frac{p_2}{p_2-1}
$$
is decreasing if 
$$
\frac{n}{p_1 (p_2-1)} \leq  \frac{p(x)(1+\sigma)-p_2}{p_2} \log\frac{1}{R}.
$$
Since the left-hand side of this relation does not exceed $n/[p^{-}_\Omega(p^{-}_\Omega-1)]$, while the right-hand side is greater than $(\sigma/4)\log (1/R)$ due to \eqref{pest2}, we get the estimate
\begin{equation}\label{tech_est}
t^{p_2} \biggl(\frac{\log t}{\log(1/R)} \biggr)^\frac{p_2}{p_2-1} \leq C(n,p^{-}_\Omega,\sigma) R^{n(p_1(1+\sigma)-p_2)/p_1}t^{(1+\sigma)p(x)}
\end{equation}
for $t\geq R^{-n/p_1}$ and $x\in B_R$ if $R \leq R (n,p^{-}_\Omega,\sigma)$. At this point we further bound $\widetilde{R}$ in the statement of the Theorem so that \eqref{tech_est} holds. Now, 
$$
n\frac{p_1(1+\sigma)-p_2}{p_1}\geq n \frac{p_1(1+\sigma)-p_1 - \Theta_p(4R)}{p_1} \geq n\sigma -n\Theta_p(4R),
$$
and so
$$
R^{n(p_1(1+\sigma)-p_2)/p_1} \leq R^{n\sigma - n \Theta_p(4R)} \leq C(n, c_{\mathrm{log}}(p)) R^{n\sigma}
$$
by the log-H\"older condition. Thus from \eqref{tech_est} we get 
\begin{equation}\label{tech_est1}
t^{p_2} \biggl(\frac{\log t}{\log(1/R)} \biggr)^\frac{p_2}{p_2-1} \leq  C(n,p^{-}_\Omega,\sigma, c_{\mathrm{log}}(p)) R^{n\sigma}t^{(1+\sigma)p(x)}
\end{equation}
for $t\geq R^{-n/p_1}$ and $x\in B_R$.


Using \eqref{te0} (where $1+\mu\leq |\omega|< R^{-n/p_1}$) and \eqref{tech_est1} (where $|\omega|\geq R^{-n/p_1}$) for the second term on the right-hand side of \eqref{frak_w} and again \eqref{minimality1} for the first term on the right-hand side of \eqref{frak_w} we then continue the chain of estimates above as:
            \begin{align*}
            \mathfrak{w}
            & \leq  c\left(\Theta_p(4R)\log\frac{1}{R}\right) \biggl(~\int\limits_{B_{R}} \mathcal{E}(p_2,\omega)\, dV+ R^{n\sigma  }\int\limits_{B_{R}} |\omega|^{(1+\sigma)p(x)} \, dV \biggr).
		\end{align*}

By the higher integrability result of Theorem~\ref{higher integrability}, the relation \eqref{pest2} and the initial choice of $\sigma$, there holds
\begin{align*}
R^{n\sigma}&\int\limits_{B_R} |\omega|^{p(x)(1+\sigma)}\, dV \\
&\leq cR^{n(1+\sigma)} \biggl(~ \biggl(~\fint\limits_{Q_{2R}} |\omega|^{p(x)}\,dV \biggr)^{1+\sigma} + 1 + \fint\limits_{Q_{2R}} |{\ast}F-{(\ast F)_{2R}}|^{(1+\sigma)p'(x)}\,dV \biggr)\\
&\leq c \biggl(~ \biggl(~\int\limits_{Q_{2R}} |\omega|^{p(x)}\, dV \biggr)^{1+\sigma} + R^{n(1+\sigma)} + R^{n\sigma}\int\limits_{Q_{2R}} |{\ast}F-{(\ast F)_{2R}}|^{(1+\sigma)p'(x)}\, dV \biggr).
\end{align*}
Combining the above estimates we get 
\begin{equation}\label{I11f}
|I_{11}|
\begin{aligned}[t]
 \leq &c\biggl( \Theta_p(4R) \log \frac{1}{R} \biggr)\\
&\qquad\times \biggl(~K^\sigma\int\limits_{B_{R}} \mathcal{E}(p_2,\omega)\, dV + R^{n} +R^{n\sigma}\int\limits_{Q_{2R}} |{\ast}F-{(\ast F)_{2R}}|^{(1+\sigma)p'(x)}\, dV \biggr).
\end{aligned}
\end{equation}
Finally gathering the estimates \eqref{decay for Morrey bound} -- \eqref{I11f}, and using the BMO property of $F$, which implies the BMO property for $\ast F$, we arrive at the decay relation
\begin{align}\label{fest0}
\int\limits_{B_\rho}&\mathcal{E}(p_2(R),\omega)\, dV \\
&\leq A \left[ \left(\frac{\rho}{R} \right)^n + \varepsilon + \Theta_a(4R) + \Theta_p(4R) \log \frac{1}{4R} \right] \int\limits_{B_{4R}} \mathcal{E}(p_2(R),\omega)\, dV +B R^n. \notag
\end{align}

		Now we set 
		\begin{align*}
			\Phi (\rho):= \int\limits_{B_{\rho}}(\mathcal{E}(p_2(\rho),\omega)+1)\, dV. 
		\end{align*}
By the Young inequality,
$$
\mathcal{E}(p_2(\rho),\omega)\leq \mathcal{E}(p_2(R),\omega)+1 \quad \text{ and } \quad  
\Phi(\rho) \leq \int\limits_{B_\rho} (\mathcal{E}(p_2(R),\omega)+2)\, dV \leq 2\Phi(R)
$$
for $\rho \leq R$. Using this notation in \eqref{fest0} we arrive at
		
		\begin{align*}
			\Phi \left( \rho\right)\leq A \left[ \left(\frac{\rho}{R}\right)^{n} + \varepsilon + \Theta_{a}(4R)+ \Theta_p(4R)\log \left(\frac{1}{4R}\right) \right] 	\Phi \left( R\right) + B_1 R^{n}.
		\end{align*} %
		Choosing $\varepsilon >0$ and $R_{1}$ sufficiently small, the Campanato--Giaquinta--Giusti iteration lemma (see Lemma~\ref{GG} in Appendix, with $\alpha=n$ and $0<\beta<n$) implies $\Phi(\rho) \leq c (\rho/R)^\beta$ for all $\beta \in (0,n)$, whence the claim easily follows.
        \end{proof}

Now we return to the proof of the result claimed in the beginning of this section.

\begin{proof}[Proof of Theorem~\ref{th:MorreyBounds}]
    (i) We again can argue in a local coordinate system. Theorem~\ref{MorreyBound} implies that we have the estimate 
	\begin{align*}
			\int\limits_{B_{r}\left(x_{c}\right)} \left\lvert \omega\right\rvert^{p_{2}(r)}\, dV \leq c r^{n-\tau},
	\end{align*}
	where $x_{c}$ and $r$ are as before and $p_2(r)$ is defined by \eqref{p2def}. Consequently, we have 
	\begin{align*}
		\int\limits_{B_{r}\left(x_{c}\right)} \left\lvert \omega \right\rvert^{p^{-}_\Omega} \, dV \leq cr^{n}\fint\limits_{B_{r}\left(x_{c}\right)} \left\lvert \omega \right\rvert^{p^{-}_\Omega}\, dV &\leq cr^{n}\biggl(~ \fint\limits_{B_{r}\left(x_{c}\right)} \left\lvert \omega \right\rvert^{p_{2}(r)}\, dV \biggr)^{\frac{p^{-}_\Omega}{p_{2}(r)}} \\
		&\leq cr^{n}\left( cr^{-\tau} \right)^{\frac{p^{-}_\Omega}{p_{2}\left(r\right)}} \leq c r^{ n - \tau}.
	\end{align*}
	This proves that for any $0 < \tau < n, $ $\omega$ is locally in the Morrey space $\mathrm{L}^{p^{-}_\Omega, n-\tau}$.
    
    (ii) Now, for any contractible $C^{1,1}$ submanifold $M_{1}$ separated from $bM$, by Theorem~\ref{divcurl system} in Appendix and Theorem 35 in  \cite{Sil_Sengupta_MorreyLorentz}, we find $\widetilde{u} \in W^{1, p(\cdot)}( M_{1}; \Lambda^{k})$ with $\nabla \widetilde{u} \in \mathrm{L}^{p^{-}_M, n-\tau}(M_{1}; \mathbb{R}^{n}\otimes\Lambda^{k} )$ for any $\tau\in (0,n)$, such that  
	\begin{align*}
		\left\lbrace \begin{aligned}
			d\widetilde{u} = \omega \quad &\text{and} \quad  d^{\ast} \widetilde{u} = 0 &&\text{ in } M_{1}, \\
			\nu\lrcorner \widetilde{u} &= 0 &&\text{  on } bM_{1}.
		\end{aligned}\right. 
	\end{align*}
By Morrey's ``Dirichlet growth'' theorem \cite[Theorem 3.5.2]{Morrey1966}, we deduce that $\widetilde{u} \in C^{0, \gamma}\left( M_{1};  \varLambda^{k}\right)$ for any $0 < \gamma < 1$. This completes the proof.  
\end{proof} 

\begin{remark} 
 In the proof of Theorem 35 in \cite{Sil_Sengupta_MorreyLorentz}, the metric is Euclidean and the domain is assumed to be $C^{2,1}$ but as flattening the boundary changes the metric, it is easy to adapt the proof to our current setting. The crux of the matter is that the derivative of the Neumann potential is a CZ operator and is bounded in all Morrey spaces $L^{q, \lambda}$ as long as $1 < q < \infty$ and $0 \le \lambda < n$.
\end{remark}

\subsection{Global Morrey estimates}

In this section we discuss Morrey bounds up to the boundary for the Dirichlet or Neumann boundary conditions.

\begin{theorem}\label{th:MorreyBoundsDN}
Let the variable exponent $p(\cdot)$ satisfy \eqref{definition log holder}, \eqref{log holder vanishing}, and the weight function $a$ satisfy\eqref{acond} and be continuous. Let $F$ be in $BMO (M; \Lambda^{k+1})$ and $\omega \in L^{p(\cdot)}(M;\Lambda^{k+1})$ be a local weak solution of \eqref{main system} in $M$ satisfying the Dirichlet condition $t\omega=0$ or the Neumann condition $n ((\mu^2+|\omega|^2)^{(p(x)-2)/2}\omega)=0$. Then $\omega \in \mathrm{L}^{p^{-}_{M},\lambda}(M;\Lambda^{k+1})$ for any $\lambda\in (0,n)$. If $M$ is contractible, there exists $\widetilde{u}\in  W^{1,p(\cdot)}(M;\Lambda^k)$ such that $d\widetilde u=\omega$ in $M$, $|\nabla \widetilde u|\in \mathrm{L}^{p^-_M,\lambda}(M)$ for any $\lambda\in (0,n)$, and $\widetilde u\in C^{0,\gamma}(M;\Lambda^k)$ for any $\gamma\in (0,1)$, satisfying $tu=0$ in the Dirichlet case and $nu=0$ in the Neumann case, respectively.
\end{theorem}

 Since local estimates are proved above in Theorem~\ref{MorreyBound}, it is sufficient to consider an admissible boundary coordinate system. Let $\Omega\subset \mathbb{R}^n\cap\{x^n>0\}$ be a Lipschitz domain, $\Gamma=\mathrm{int}\,(\partial\Omega\cap \{x^n=0\})$ be nonempty, the coordinate system be admissible (that is, $g_{in}=\delta_{in}$ on $\Gamma$), with $g_{ij}$ at least Lipschitz.

		\begin{theorem}\label{MorreyBoundDN}
        Let the variable exponent $p(\cdot)$ satisfy \eqref{definition log holder}, \eqref{log holder vanishing}, and the weight function $a$ satisfy\eqref{acond} and be continuous in $\Omega$. Let $F$ be in $BMO (B_{R_1}(x_0), \Lambda^{k+1})$ and $\omega \in L^{p(\cdot)}(\Omega;\Lambda^{k+1})$ be a weak solution of \eqref{main system} in $\Omega$ satisfying the Dirichlet or Neumann condition $t\omega =0$ or $n( (\mu^2 + |\omega|^2)^{(p(x)-2)/2}\omega)=0$ on $\Gamma$. For any $0 < \tau <n$ there exists a choice of a positive radius $\widetilde{R}\in (0, R_1/16)$, depending only on $data_1$, and a positive constant $c$, depending on $\tau$, $data_1$, $\mu^+$, and the norm of $F$ in $BMO (B_{R_1}(x_0), \Lambda^{k+1})$, such that 
			\begin{align*}
				\int\limits_{B_{r}(x_{c})\cap \Omega} \left\lvert \omega \right\rvert^{p_{2}\left(r\right)}\,dV \leq c r^{n-\tau},
			\end{align*}
			whenever $0 < r <  \widetilde{R}/\left(64\right)^{2}$ and  $x_{c} \in \overline{B}_{\widetilde{R}}\left(x_{0}\right)$ .
		\end{theorem} 
\begin{proof}

I. (\textit{Neumann boundary condition}). For the Neumann boundary condition the boundary neighbourhood is treated by the same extension device as was used in Section~\ref{ssec:hib} in the proof of Theorem~\ref{higher integrability boundary}. After extending the metric, $\omega$ and $F$ to $\{x^n<0\}$, the boundary estimates are reduced to the local estimates obtained in Theorem~\ref{MorreyBound}.

II. (\textit{Dirichlet boundary condition}). As above, we work in an admissible boundary coordinate system. We work either in balls separated from the boundary, where estimates were obtained in Theorem~\ref{MorreyBound}, or in half-balls centered on the boundary. In the latter case case, for half-balls $B_{R_1}$ centered on the boundary, using Lemma~\ref{L:Gauge_Tan} we first find $\bar{u}$ in $B_{R_1/2}^+$ such that $d\bar{u} = \omega$ and $t\bar{u}=0$ on $B_{R_1/2}\cap\{x^n=0\}$. Then  we find $v \in W^{1, p_{2}}\left( B_{R}^+; \varLambda^{k}\right)$ such that
		\begin{align}\label{frozen system for comparison boundary}
			\left\lbrace \begin{aligned}
				d^{\ast}\bigl( a(x_{\text{max}}) \mathcal{A}(p_2, dv) \bigr)  &=0 &&\text{ in }  B_{R}^+,\\ 
				d^{\ast}v &= 0 &&\text{ in }  B_{R}^+, \\
				\nu\wedge v &= \nu \wedge \bar{u} &&\text{ on } \partial B_{R}^+. 
			\end{aligned}\right. 
		\end{align}
To this end, let us extend $\bar{u}$ to $B_R^{-}$ by $-S^* \bar{u}$, that is the components $\bar{u}_I$ with $n\notin I$ will be
odd and the components $\bar{u}_{In}$ will be even functions of $x^n$. Let us extend the metric tensor as above (that is, $g_{ij}$ is even in $x^n$ if $i,j<n$ or $i,j=n$ and odd otherwise). There exists a unique solution of \eqref{frozen system for comparison} in $B_R$. Let us show that the form $\hat{v}=-S^* v$ obtained by the even reflection of $v_{In}$ and odd reflection of $v_I$, $n\notin I$, will be a solution to the same boundary value problem. First, by definition one can easily check that $\hat{v}$ satisfies the same boundary condition. Second, it delivers the same energy of the corresponding functional. Third, 
$$
\int\limits_{B_R} \langle \hat{v}, d\xi \rangle \, dV = -\int\limits_{B_R} \langle v, d S^* \xi \rangle\, dV =0
$$
for any smooth compactly supported test form $\xi$. By uniqueness, $v = \hat{v}$, thus $tv=0$ on $B_R\cap\{x^n=0\}$.

Using then the results of Hamburger for half-balls (see Theorem~\ref{Hamburger estimate}) we see that the rest of the arguments goes through for balls replaced with half-balls. Combining the internal and boundary estimate we see that for the ball $B_r(x',t)$ and $\omega$ obtained by the even reflection of $\omega_{In}$ and odd reflection of $\omega_I$, $n\notin I$, we have the required estimate for all balls $B_r$.
\end{proof}


\subsection{H\"{o}lder continuity of \texorpdfstring{$\omega$}{omega}}

In this section we prove the main result of this paper.

\begin{theorem}\label{main theorem}
	 Let the variable exponent $p(\cdot)$ satisfy \eqref{definition log holder} and \eqref{holder exponent} with some $0<\alpha_1<1$. 
    Let the weight function $a \in C^{0, \alpha_{2}}_{\mathrm{loc}}(M)$ for some $0 < \alpha_{2} < 1$ and satisfy \eqref{acond}. Let  $F \in C^{0, \alpha_{3}}_{\mathrm{loc}}(M \varLambda^{k+1})$ for some $0 < \alpha_{3} < 1$. Let $\omega \in L_{\mathrm{loc}}^{p(\cdot)} (M; \varLambda^{k+1})$ be a local weak solution to the system \eqref{main system} in $M$. Then the form $\omega$ is locally H\"older continuous in $M$ with the H\"older exponent depending only on $n$, $N$, $p^{-}_M$, $p^{+}_M$, $\alpha_1$, $\alpha_2$, $\alpha_3$, $M$. 
\end{theorem}
	\begin{proof}
We again argue in a local coordinate system $(U,\varphi)$, with $\Omega = \varphi(U)$ a Lipschitz domain in $\mathbb{R}^n$. We show that $\omega$ is H\"{o}lder continuous in a ball of small enough radius, but otherwise arbitrary. We set 
        \begin{align*}
			\alpha := \min \left\lbrace \alpha_{1}, \alpha_{2}, \alpha_{3}\right\rbrace .
		\end{align*}
		
		Now we prove the H\"{o}lder continuity of the $\omega$ in $B_{\widetilde{R}}\left(x_{0}\right)$, where $\widetilde{R}>0$ is the radius given by Theorem~\ref{MorreyBound}. We choose $0 < \rho_{0} < 1$ such that $\left(2\rho_{0}\right)^{\frac{1}{1+\theta}} < \widetilde{R}/\left(64\right)^{2}.$ 
		Set 
		$$
        r= r ( \rho ) := \left(2\rho\right)^{\frac{1}{1+\theta}}. 
        $$
        Clearly, for any $0 < \rho < \rho_{0}$, we have $r( \rho) < \widetilde{R}/\left( 64\right)^{2}$. Thus, by Theorem~\ref{MorreyBound}, for any $0<\tau < n$ we have 
		\begin{align}\label{Morrey decay}
			\int\limits_{B_{r ( \rho )}\left(x_{c}\right)}\left\lvert \omega \right\rvert^{p_{2}\left(r( \rho )\right)} \, dV \leq c \left(r ( \rho )\right)^{ n - \tau }.
		\end{align}
		As before, define $\bar{u}$ as in \eqref{gauge fixing 2} and find $v \in W^{1, p_{2}\left(r \left( \rho \right)\right)}\left( B_{r( \rho )}; \varLambda^{k}\right)$ such that 
		\begin{align}\label{frozen system for comparison Holder}
			\left\lbrace \begin{aligned}
				d^{\ast}( a(\bar{x}_{\text{max}}) \mathcal{A}(p(\bar{x}_{\text{max}}),dv)) ) &=0 &&\text{ in }  B_{r(\rho)}, \\
				d^{\ast}v &= 0 &&\text{ in }  B_{r (\rho)}, \\
				\nu\wedge v &= \nu \wedge \bar{u} &&\text{ on } \partial B_{r (\rho)}, 
			\end{aligned}\right. 
		\end{align} 
		where $\bar{x}_{\text{max}} \in \overline{B}_{4r ( \rho )}$ is a point such that $p\left(\bar{x}_{\text{max}}\right)= p_{2}(r ( \rho ))$. 

        Hence, as before in \eqref{estimate 1}, we have 
		\begin{align}\label{estimate 1 for holder}
			a^-_\Omega \int\limits_{B_{r}}   \bigl\langle \mathcal{A}(p_2,\omega) - \mathcal{A}(p_2,dv), \omega-dv \bigr\rangle\, dV \leq I_1 + I_2
		\end{align}
		where 
		\begin{align*}
			I_{1}&:= \int\limits_{B_{r}}\left\langle   a(x)\mathcal{A}(p(x),\omega)- a(x_{\text{max}})\mathcal{A}(p_2(r),\omega), \omega - dv \right\rangle\, dV, \\
			I_{2} &:= - \int\limits_{B_{r}}\left\langle   a(x)\mathcal{A}(p(x),\omega), \omega - dv \right\rangle\, dV. 
		\end{align*}
		As in \eqref{estimate of F term}, but using now the Morrey bound~\eqref{Morrey decay}, we arrive at 
		\begin{align}\label{estimate of F term holder}
			\left\lvert I_{2}\right\rvert &\leq \int\limits_{B_{r}}\left\lvert \ast F - (\ast F)_{r}\right\rvert\cdot \left\lvert \omega - dv\right\rvert \, dV \notag \\
            &\leq c(n)\left[\ast F\right]_{C^{0, \alpha}} r^{n+\alpha} \fint\limits_{B_{r}} \left\lvert \omega - dv\right\rvert \, dV \notag \\
            & \leq c \left[{\ast} F\right]_{C^{0, \alpha}} r^{n+\alpha} \biggl(~\fint\limits_{B_{r}} \mathcal{E}(p_2(r),\omega)\, dV \biggr)^\frac{1}{p_2(r)} \leq c \left[ {\ast}F\right]_{C^{0, \alpha}} r^{n+\alpha-\tau}.
		\end{align}
        Here we used that $F\in C^{0,\alpha}$ implies $\ast F\in C^{0,\alpha}$. Writing $I_{1}=I_{11}+ I_{12}$ as before, we have, as in \eqref{estimate of a term},  
\begin{align}\label{estimate of I12 holder}
	\left\lvert I_{12}\right\rvert &\leq c\Theta_{a}(4r)  \int\limits_{B_{r}} \mathcal{E}(p_2(r),\omega)\, dV \stackrel{\eqref{Morrey decay}}{\leq} cr^{\alpha + n -\tau }. 
\end{align}
On the other hand, from \eqref{I11f} we obtain 
\begin{align}\label{estimate of I11 holder}
		|I_{11}| \leq c r^{\alpha + n - \tau} \log \frac{1}{r}.
\end{align}	
Thus, combining \eqref{estimate of F term holder}, \eqref{estimate of I12 holder} and \eqref{estimate of I11 holder}, we obtain 
\begin{equation}
\label{estimate of I holder}
\int\limits_{B_{r}}  \bigl\langle \mathcal{A}(p_2,\omega) -\mathcal{A}(p_2,dv), \omega-dv \bigr\rangle\, dV
\leq cr^{\alpha + n -\tau }\log\left(\frac{1}{4r}\right)
\leq cr^{\frac{\alpha}{2} + n -\tau }. 
\end{equation}
Now if $p_{2}(r) \geq 2$, using \eqref{mu1} we have 
\begin{align*}
	\int\limits_{B_{r}} \left\lvert \omega - dv\right\rvert^{p_{2}(r)}\, dV &\leq c(p^{+}_\Omega) \int\limits_{B_{r}} \bigl\langle \mathcal{A}(p_2,\omega) -\mathcal{A}(p_2,dv), \omega-dv \bigr\rangle\, dV
    \leq cr^{\frac{\alpha}{4} + n -\tau }. 
\end{align*}
On the other hand, if $1 < p_{2}\left(r\right)< 2$, from \eqref{mu2} by the H\"older inequality we have 
\begin{align*}
 		\int\limits_{B_{r}} \left\lvert \omega - dv\right\rvert^{p_{2}(r)}\, dV 
		&\begin{multlined}[t]
		    \leq \biggl(~\int\limits_{B_{r}}\bigl\langle \mathcal{A}(p_2,\omega) -\mathcal{A}(p_2,dv), \omega-dv \bigr\rangle\, dV\biggr)^{\frac{p_{2}}{2}}\\ \times\biggl(~\int\limits_{B_{r}} \left(2\mu^2+ \lvert \omega \rvert^{2} + \lvert dv \rvert^{2}\right)^{\frac{p_{2}}{2}} \, dV\biggr)^{\frac{2-p_{2}}{2}}
		\end{multlined} \\
		&\leq cr^{\frac{p_{2}}{2}\left(\frac{\alpha}{2} + n -\tau \right)}\biggl(~\int\limits_{B_{r}} \mathcal{E}(p_2,\omega)\, dV\biggr)^{\frac{2-p_{2}}{2}} \\
		&\leq c r^{\frac{p_{2}}{2}\left(\frac{\alpha}{2} + n -\tau \right)}\cdot r^{\left(n-\tau\right)\frac{2-p_{2}}{2}} = c r^{n-\tau + \frac{\alpha p_{2}}{4} } \leq c r^{\left(\frac{\alpha}{4} + n -\tau \right)}. 
\end{align*}
Thus, in both cases, we arrive at 
\begin{align}\label{comparison estimate holder}
	\int\limits_{B_{r}} \left\lvert \omega - dv\right\rvert^{p_{2}(r)}\, dV \leq c r^{\frac{\alpha}{4} + n -\tau }. 
\end{align}
Further, using the Uhlenbeck--Hamburger estimate \eqref{hamburger_osc_use}, we have 
\begin{align*}
	\int\limits_{B_{\rho}} \bigl\lvert \omega - \left(\omega\right)_{B_{\rho}} \bigr\rvert^{p_{2}(r)}\, dV &\leq 	c\int\limits_{B_{\rho}} \bigl\lvert \omega - \left(dv\right)_{B_{\rho}} \bigr\rvert^{p_{2}(r)}\, dV \\
	&\leq c\int\limits_{B_{\rho}} \bigl\lvert dv - \left(dv\right)_{B_{\rho}} \bigr\rvert^{p_{2}(r)}\, dV + c\int\limits_{B_{\rho}} \bigl\lvert \omega - dv \bigr\rvert^{p_{2}(r)}\, dV \\
	&\stackrel{\eqref{hamburger_osc_use}}{\leq} c\left( \frac{\rho}{r}\right)^{\beta p_{2}(r)}\rho^{n} \biggl(~ \fint\limits_{B_{r}} \mathcal{E}(p_2,\omega)\, dV\biggr) + c\int\limits_{B_{\rho}} \left\lvert \omega- dv \right\rvert^{p_{2}(r)}\, dV \\
	&\stackrel{\eqref{comparison estimate holder}}{\leq}c\left( \frac{\rho}{r}\right)^{\beta p_{2}\left(r\right)}\rho^{n} \biggl(~ \fint\limits_{B_{r}} \mathcal{E}(p_2,\omega)\, dV\biggr) + c r^{\frac{\alpha}{4} + n -\tau } \\
	&\stackrel{\eqref{Morrey decay}}{\leq}c\left( \frac{\rho}{r}\right)^{\beta p_{2}\left(r\right)}\rho^{n} r^{-\tau}+ c r^{\frac{\alpha}{4} + n -\tau } \leq c\left( \frac{\rho}{r}\right)^{\beta}\rho^{n} r^{-\tau}+ c r^{\frac{\alpha}{4} + n -\tau }.  
\end{align*}
Now we choose $\tau = \alpha\beta /8\left(n+\beta\right)$ and choose $\theta = \alpha/4\left(n+\beta\right).$ Then the powers of $\rho$ are the same in both terms and is equal to $n + \kappa$ where 
\begin{align*}
	\kappa := \frac{\alpha\beta}{2\left[ \alpha + 4\left( n +\beta\right)\right]} >0.
\end{align*} 
This implies 
\begin{align*}
	\int\limits_{B_{\rho}} \bigl\lvert \omega - \left(\omega\right)_{B_{\rho}} \bigr\rvert^{p^{-}_\Omega}\, dV 
    &\leq c\rho^{n}\biggl(~ \fint\limits_{B_{\rho}} \bigl\lvert \omega - \left(dv\right)_{B_{\rho}} \bigr\rvert^{p_{2}(r)}\, dV \biggr)^{\frac{p^{-}_\Omega}{p_{2}(r)}} \\
	&\leq c\rho^{n}\left( c\rho^{\kappa} \right)^{\frac{p^{-}_\Omega}{p_{2}\left(r\right)}} \leq c\rho^{n+p^{-}_\Omega \tilde \alpha},\quad \tilde \alpha = \frac{\kappa}{p^{+}_\Omega}
\end{align*}
as $p_{2}\left(r\right) < p^{+}_\Omega$ and $\rho_{0} < 1$. Since this holds for any $x_{c} \in \overline{B}_{\widetilde{R}}\left(x_{0}\right)$ and any $0 < \rho < \rho_{0}$, by Campanato's characterization (see the original paper \cite{Campanato} or \cite[Chapter III, \S 1, Theorem 1.2]{Giaquinta}), this implies the H\"{o}lder continuity with exponent $\tilde \alpha$ of $\omega$. This completes the proof.  
\end{proof}

\begin{corollary} Under conditions of Theorem ~\ref{main theorem}, for any contractible $C^{2,\alpha_4}$ submanifold $M_1$ of $M$ there exists $\widetilde{u} \in C^{1,\tilde{\alpha}}\left( \Omega_{1}; \varLambda^{k}\right)$, depending on $\omega$ and $\Omega_{1},$ where $0 < \tilde{\alpha} <1$ depends only on $n$, $N$, $k$, $p^{-}_M$, $p^{+}_M$, $a^-_M$, $a^+_M$, $\alpha_{1}$, $\alpha_{2}$, $\alpha_{3}$, $\alpha_4$, such that $d\widetilde{u} = \omega$ and $d^* \widetilde{u}=0$ in $M_1$. 
\end{corollary}

\begin{proof}

Using  \cite[Theorem 7.7.8]{Morrey1966} we can find $\widetilde{u} \in C^{1, \tilde{\alpha}}\left(M_1; \varLambda^{k}\right)$, where the exponent is given by $\tilde{\alpha}=\min\{\alpha_4, \text{ the H\"older exponent for }\omega\}$, such that  
\begin{align*}
	\left\lbrace \begin{aligned}
			d\widetilde{u} = \omega  \quad &\text{and} \quad  d^{\ast} \widetilde{u} = 0 &&\text{ in } M_{1}, \\
		\nu\lrcorner \widetilde{u} &= 0 &&\text{  on } bM_{1}.
	\end{aligned}\right. 
\end{align*}
In fact, $\widetilde u$ is constructed as the codifferential of the Neumann potential of $\omega$. Moreover, $\|\widetilde{u}\|_{C^{1,\tilde{\alpha}} (M_1;\Lambda^k)} \leq c \|\omega\|_{C^{\tilde{\alpha}} (M_1;\Lambda^{k+1})}$.
\end{proof}


\subsection{Global H\"older continuity}

\begin{theorem}\label{main theorem global}
	 Let the variable exponent $p(\cdot)$ satisfy \eqref{definition log holder} and \eqref{holder exponent} with some $0<\alpha_1<1$.  
    Let the weight function $a \in C^{0, \alpha_{2}}(M)$ for some $0 < \alpha_{2} < 1$ and satisfy \eqref{acond}. Let  $F \in C^{0, \alpha_{3}}(M \varLambda^{k+1})$ for some $0 < \alpha_{3} < 1$. Let $\omega \in L^{p(\cdot)} (M; \varLambda^{k+1})$ be a weak solution to the system \eqref{main system} in $M$ satisfying the Dirichlet or the Neumann condition $t\omega=0$ or $n ( (\mu^2+|\omega|^2)^{(p(x)-2)/2})=0$, respectively. Then $\omega$ is H\"older continuous in $M$ with the H\"older exponent depending only on $n$, $N$, $p^{-}_M$, $p^{+}_M$, $\alpha_1$, $\alpha_2$, $\alpha_3$, $M$. 
\end{theorem}

\begin{proof} The interior estimates were obtained in Theorem~\ref{main theorem}. So we only have to study the regions close to the boundary. As before, we argue in an admissible boundary coordinate system.

I. (\textit{Neumann case}). Again, by the extension device described in Section~\ref{ssec:hib} in the proof of Theorem~\ref{higher integrability boundary}, in this case the proof is reduced to the interior estimates covered by Theorem~\ref{main theorem}.

II. (\textit{Dirichlet case}). In this case we use the same choice of the potential $\bar{u}$ and the comparison function $v$ as in the proof of Theorem~\ref{MorreyBoundDN}. That is, $\bar u$ is provided by Lemma~\ref{L:Gauge_Tan} instead of \eqref{gauge fixing 2} and $v$ is provided by \eqref{frozen system for comparison boundary} instead of \eqref{frozen system for comparison}). Thus the arguments in the proof of Theorem~\ref{main theorem} are repeated in half-balls centered at the boundary.
\end{proof}

\begin{corollary} Under conditions of Theorem ~\ref{main theorem global}, if $M$ is contractible and $C^{2,\alpha_4}$, there exists $\widetilde{u} \in C^{1,\tilde{\alpha}}\left(M; \varLambda^{k}\right)$ where $0 < \tilde{\alpha} <1$ depends only on $n$, $N$, $k$, $p^{-}_M$, $p^{+}_M$, $a^-_M$, $a^+_M$, $\alpha_{1}$, $\alpha_{2}$, $\alpha_{3}$, $\alpha_4$, such that $d\widetilde{u} = \omega$ and $d^* \widetilde{u}=0$ in $M$, and $tu=0$ (respectively $nu=0$) for the Dirichlet/Neumann boundary condition. 
\end{corollary}

\section{Results for the second-order system}\label{sec:2nd}

In this section we collect results for the second-order system \eqref{main2}. As above, $n \geq 2$, $N \geq 1$ and $0 \leq k \leq n-1$ are integers and $M$ is a compact orientable $C^{1,1}$ manifold with boundary, $\mathrm{dim}\, M =n$, the variable exponent $p(\cdot)$ satisfies \eqref{definition log holder} and the weight $a(\cdot)$ satisfies \eqref{acond}.

\subsection{Higher integrability of the gradient}


\begin{theorem}\label{T1u}
Let $u\in L^{p(\cdot)}_{\mathrm{loc}}(M; \Lambda^k)$ be a solution to \eqref{eq2} with $F\in L^{qp(\cdot)}_{\mathrm{loc}}(M; \Lambda^{k+1})$, $q>1$. Then there exists a constant 
$$
\sigma_{0} = \sigma_{0}\bigl(\textit{data}, K_{0}[\Omega], q-1 \bigr) \in (0,1)
$$
such that for any $\sigma <\sigma_0$ and any $M_1 \Subset \mathring{M}$ there exists a closed form $\eta$ such that $u-\eta\in W^{1,(1+\sigma)p(\cdot)} (M_1; \Lambda^k)$, and moreover $t(u-\eta)=0$ for the Dirichlet case. If $d^* u \in L_{\mathrm{loc}}^{(1+\sigma)p(\cdot)}(M;\Lambda^{k-1})$ then $u\in W^{1,(1+\sigma)p(\cdot)}_{\mathrm{loc}}(M;\Lambda^k)$.  
\end{theorem}

\begin{proof}

The first statement follows from Theorem~\ref{th:HIM} for $\omega=du$ and the results of  \cite{Balci_Sil_Surnachev_arxiv2025} for div-curl systems in variable exponent spaces (Theorem~\ref{divcurl system} in Appendix). In any $C^{1,1}$ submanifold $M_1 \subset \mathring{M}$  there exists $\widetilde u \in W^{1,(1+\sigma)p(\cdot)}(M_1; \Lambda^k)$ such that $d\widetilde u = du$ and  $d^* \widetilde u = 0$ in $M_1$, $\nu \lrcorner \widetilde u =0$ on $b M_1$. Then $\eta = u-\widetilde u$ is a closed form.

Now let $d^*u \in L_{\mathrm{loc}}^{(1+\sigma)p(\cdot)}(M;\Lambda^{k-1})$. Let $\xi$ be a Lipschitz function with compact support in $\mathring{M}$. Then 
\begin{align*}
    d(\xi u) &= \xi du + d\xi \wedge u\in L^{p(\cdot)}(M;\Lambda^{k+1}), \\ 
d^* (\xi u) &= \xi d^* u - d\xi \lrcorner u \in  L^{p(\cdot)}(M;\Lambda^{k+1}).
\end{align*}
 
Thus $\xi u \in W^{1,p(\cdot)}(M; \Lambda^k)$ by the Gaffney inequality of \cite{Balci_Sil_Surnachev_arxiv2025}. Since $\xi$ is arbitrary by the Sobolev inequality  this implies $u\in L^{np(\cdot)/(n-1)}_{\mathrm{loc}}(M;\Lambda^k)$, and thus (we can assume without loss that $\sigma<n/(n-1)$) 
\begin{align*}
    d(\xi u)\in L^{(1+\sigma)p(\cdot)}(M;\Lambda^{k+1}) \quad \text{ and } \quad d^* (\xi u) \in  L^{(1+\sigma)p(\cdot)}(M;\Lambda^{k+1}).
\end{align*}

Again by the Gaffney inequality in variable exponent spaces, $\xi u\in W^{1, (1+\sigma)p(\cdot)}(M; \Lambda^k)$. It remains to recall again that $\xi$ is arbitrary.
\end{proof}

The global version is given by
\begin{theorem}\label{T1uB}
Let $u\in L^{p(\cdot)}(M; \Lambda^k)$ be a solution to \eqref{eq2} with $F\in L^{qp(\cdot)}(M; \Lambda^{k+1})$, $q>1$, satisfying the Dirichlet boundary condition $tu=0$ or the Neumann boundary condition $n (\mathcal{A}(p,du))=0$. Then there exists a constant $\sigma_{0} = \sigma_{0}\bigl(\textit{data}, K_{0}[\Omega], q-1 \bigr) \in (0,1)$ such that for any $\sigma <\sigma_0$ there exists a closed form $\eta$ such that $u-\eta\in W^{1,(1+\sigma)p(\cdot)} (M; \Lambda^k)$, and moreover $t(u-\eta)=0$  for the Dirichlet case. If $d^* u \in L^{(1+\sigma)p(\cdot)}(M;\Lambda^{k-1})$ then $u\in W^{1,(1+\sigma)p(\cdot)}(M;\Lambda^k)$.  
\end{theorem}

\subsection{H\"{o}lder continuity of solutions}

\begin{theorem}\label{T2u}
Let the variable exponent $p(\cdot)$ satisfy \eqref{definition log holder} and \eqref{log holder vanishing}. 
	  Let the weight function $a(\cdot)$ be continuous and satisfy \eqref{acond}. Let $F \in BMO(M; \varLambda^{k+1})$. Let $u\in W_{\mathrm{loc}}^{d,p(\cdot)} \left(\mathring{M}; \varLambda^{k}\right)$ be a local weak solution to the system \eqref{main2} in $M$. Then, for any $C^{1,1}$ submanifold $M_{1} \Subset \mathring{M}$, there exists $\widetilde{u} \in W^{1,p(\cdot)}( M_{1}; \Lambda^{k})$, depending on $\omega$ and $M_{1}$, such that $d\widetilde{u} = du$ and $d^*\widetilde{u}=0$ in $M_1$, $|\nabla \widetilde{u}| \in \mathrm{L}^{p^{-}_M, \lambda}(M_1)$ for any $\lambda\in (0,n)$ and $\widetilde{u} \in C^{0,\kappa}\left( M_{1}; \varLambda^{k}\right)$ for any $0<\kappa<1$. If $d^* u=0$ then $u\in C^{0,\kappa}_{\mathrm{loc}}(M;\Lambda^k)$ and $|\nabla u| \in \mathrm{L}_{\mathrm{loc}}^{p^{-}_M, \lambda}(M)$ for any $\lambda\in (0,n)$.
\end{theorem}

\begin{proof}
The proof repeats the proof of Theorem~\ref{th:MorreyBounds}. It only remains to take $\omega =du$ in this proof. Note that $d^* \widetilde u=0$ by construction. Since $\omega = du$ is already exact, the compatibility conditions are automatically satisfied and no assumption on the homology of $M_1$ is required. 

For the second part, if one already knows that $d^* u =0$, then the difference $u-\widetilde{u}$ is a harmonic field ($d(u-\widetilde u)=0$, $d^* (u-\widetilde u)=0$). Since on a $C^{1,1}$ manifold any harmonic field belongs (at least locally) to $W^{1,q}$ for any $q<\infty$, its gradients belongs to the Morrey class $\mathrm{L}^{p^{-}_M,\lambda}_{\mathrm{loc}}(M;\Lambda^k)$ for any $\lambda\in (0,n)$. By the Sobolev embedding they are also H\"older continuous with any exponent less than $1$. 
\end{proof}

\begin{theorem}\label{T2uB}
Let the variable exponent $p(\cdot)$ satisfy \eqref{definition log holder} and \eqref{log holder vanishing}. 
	  Let the weight function $a(\cdot)$ be continuous and satisfy \eqref{acond}. Let $F \in BMO(M; \varLambda^{k+1})$. Let $u\in W^{d,p(\cdot)} \left(M; \varLambda^{k}\right)$ be a weak solution to the system \eqref{main2} in $M$ satisfying the Dirichlet boundary condition $tu=0$ or the Neumann boundary condition $n (\mathcal{A}(p,du))=0$. Then there exists $\widetilde{u} \in W^{1,p(\cdot)}( M; \Lambda^{k})$, such that $d\widetilde{u} = du$ and $d^*\widetilde{u}=0$, $tu=0$ or $nu=0$, respectively, $|\nabla \widetilde{u}| \in \mathrm{L}^{p^{-}_M, \lambda}(M)$ for any $\lambda\in (0,n)$ and $\widetilde{u} \in C^{0,\kappa}\left( M; \varLambda^{k}\right)$ for any $0<\kappa<1$. If $d^* u=0$ then $u\in C^{0,\kappa}(M;\Lambda^k)$ and $|\nabla u| \in \mathrm{L}^{p^{-}_M, \lambda}(M)$ for any $\lambda\in (0,n)$.
\end{theorem}

\subsection{H\"{o}lder continuity for the gradient}

By the same arguments as above we get the following result.

\begin{theorem}\label{T3u}
 Let the variable exponent $p(\cdot)$ satisfy \eqref{definition log holder} and \eqref{holder exponent} with some $0<\alpha_1<1$. Let the weight function $a \in C^{0, \alpha_{2}}_{\mathrm{loc}}(\mathring{M})$ for some $0 < \alpha_{2} < 1$ and satisfy \eqref{acond}. Let  $F \in C^{0, \alpha_{3}}_{\mathrm{loc}}\left(\mathring{M}; \varLambda^{k+1}\right)$ for some $0 < \alpha_{3} < 1$. Let $u \in W_{\mathrm{loc}}^{d,p(\cdot)} \left(\mathring{M}; \varLambda^{k+1}\right)$ be a local weak solution to the system \eqref{main2} in $M$. There exists a number $\alpha \in (0,1)$, which depends only on $n$, $N$, $p^{-}_M$, $p^+_M$, $\alpha_1$, $\alpha_2$, $\alpha_3$, and $M$, such that for any $C^{2,\alpha_4}$, $\alpha_4\in (0,1)$, submanifold $M_1\Subset \mathring{M}$ there exists $\widetilde u \in C^{1,\min\{\alpha,\alpha_4\}}(M_1;\Lambda_k)$ satisfying $d\widetilde u = du$ and $d^*\widetilde u=0$ in $M_1$. If $M\in C^{2,\alpha_4}$, $\alpha_4\in (0,1)$, $d^*u =0$ in $M$, then $u\in C_{\mathrm{loc}}^{1,\min\{\alpha,\alpha_4\}}(\mathring{M};\Lambda^k)$. 
\end{theorem}

\begin{theorem}\label{T3u_global}
 Let the variable exponent $p(\cdot)$ satisfy \eqref{definition log holder} and \eqref{holder exponent} with some $0<\alpha_1<1$. Let the weight function $a \in C^{0, \alpha_{2}}(M)$ for some $0 < \alpha_{2} < 1$ and satisfy \eqref{acond}. Let  $F \in C^{0, \alpha_{3}}\left(M; \varLambda^{k+1}\right)$ for some $0 < \alpha_{3} < 1$. Let $M$ be additionally of the class $C^{2,\alpha_4}$, $\alpha_4\in (0,1)$. Let $u \in W^{d,p(\cdot)} \left(M; \varLambda^{k+1}\right)$ be a weak solution to the system \eqref{main2} in $M$ satisfying the Dirichlet or Neumann boundary condition $tu=0$ or $n \mathcal{A}(p,du)=0$, respectively. There exists a number $\alpha \in (0,1)$, which depends only on $n$, $N$, $p^{-}_M$, $p^+_M$, $\alpha_1$, $\alpha_2$, $\alpha_3$, $\alpha_4$, and $M$, such that  there exists $\widetilde u \in C^{1,\alpha}(M;\Lambda_k)$ satisfying $d\widetilde u = du$ and $d^*\widetilde u=0$ in $M$. If $d^*u =0$ in $M$, then $u\in C^{1,\alpha}(M;\Lambda^k)$. 
\end{theorem}


\section{Appendix}\label{sec:appendix}

Here we provide details on notation and proofs of some facts used above.

\subsection{Exterior algebra and exterior bundle notation}\label{EAS} Let $W$ be an $n$-dimensional real linear space.  We write $\Lambda^{k}W$ to denote the vector space of all alternating $k-$linear maps $f:\underbrace{W\times\cdots\times W}_{k-\text{times}
		}\rightarrow\mathbb{R}$. For $k=0$, we set $\Lambda^{0} W  =\mathbb{R}$. For $k\leq n$, $\operatorname{dim}\left(\Lambda^{k}  W  \right)  ={\binom{{n}}{{k}}}$, and $\Lambda^{k} \mathbb{R}%
		^{n}  =\{0\}$ for $k>n$. 

If $\left\{  e_{1},\cdots,e_{n}\right\}  $ is a basis of $W$, then, the dual basis $\left\{  e^{1},\cdots,e^{n}\right\}$ (with $e^i(e_j)=\delta^i_j$) is a basis for $\Lambda^{1}W$. If 
        $$
        \mathcal{T}^{k}=\left\{ \left(  i_{1}\,,\cdots,i_{k}\right)\ :\ 1\leq i_{1}<\cdots<i_{k}\leq n\right\},
        $$
        then
		$\left\lbrace e^{I}:= e^{i_{1}}\wedge\cdots\wedge e^{i_{k}},\ I \in \mathcal{T}^k \right\rbrace$,
		is a basis of $\Lambda^{k}W$. If $v_1,\ldots,v_N$ is a basis of an $N$-dimensional real vector space $E$ (identified with $\mathbb{R}^N$), an element $\xi\in\varLambda^{k}(W,E)$ can be written as%
		\begin{equation}\label{xidef}
			\xi=\sum\limits_{j=1}^{N}\sum_{I\in\mathcal{T}^{k}}\xi
			_{I}^{\hphantom{I}|j}\,e^{I}\otimes v_{j} = \sum_{I\in\mathcal{T}^{k}} \xi_I\, e^I,\quad\text{where}\quad \xi_I \in E.
		\end{equation}
If both spaces $W$ and $E$ are equipped with scalar product, $\langle \cdot,\cdot\rangle_W$ and $\langle \cdot,\cdot\rangle_E$, correspondingly, then the scalar product of $\xi$ and $\eta$ in $\Lambda^k$ is then defined as 
        $$
        \langle \xi,\eta\rangle_{\Lambda^k}=  \sum_{I,J \in \mathcal{T}^k} \langle\xi_I, \eta_J\rangle_{E} \langle e^I, e^J \rangle _{\Lambda^k W}.
        $$
        Here $\langle e^I, e^J \rangle _{\Lambda^k W} = G^{IJ}_W$ where $G^{IJ}_W$ is the minor at the intersection of rows $I$ and columns $J$ of the matrix $g^{ij}_W$ (the inverse matrix of the metric corresponding to the space $W$).  Further we shall denote the scalar product in $\Lambda^k$ by $\left\langle \ ,\ \right\rangle $.

       \
 $\wedge$, $\lrcorner$,  and, respectively, $\ast$ denote the exterior product, the interior product, and, respectively, the Hodge star operator. For a pair of forms, the $\wedge$ and $\lrcorner$ operators are defined by standard formulas if one of the forms  is an element of $\Lambda^k W$. In this paper we shall use only $v \wedge \xi$ and $v\lrcorner \xi$, where $v= \sum_{j=1}^n v_j e^j\in \Lambda^1W$, $\xi \in \Lambda^k$, and in this case as usual
\begin{align*}
 (v \lrcorner \xi)_I &= \sum_{j=1}^{n} v^j \xi_{jI},\ I \in \mathcal{T}_{k-1}, \qquad v^j = g^{jk}_W v_k,\\
(v\wedge \xi)_I &= \sum_{j=1}^{k+1} (-1)^{j-1} v_j \xi_{I'_j},\ I \in \mathcal{T}_{k+1}, \\
\end{align*}
where for $I=\{i_1,\ldots,i_{k+1}\}$, $I'_j = \{i_1,\ldots,i_{j-1},i_{j+1},\ldots i_{k+1}\}$. The Hodge star operator $\ast \xi$ defined as the interior product of $\xi$ and the volume form $\sqrt{g}e^1\wedge\ldots\wedge e^n$, also easily extends to vector valued forms. For $\xi$ in \eqref{xidef}, 
\begin{align*}
(\ast \xi)_{i_1\ldots i_{n-k}} &= \sqrt{g}\sum g^{j_1 l_1} \ldots g^{j_k l_k} \xi_{l_1\ldots l_k} \varepsilon_{j_1\ldots j_k i_1\ldots i_{n-k}},\\
\text{or}\quad (\ast \xi)_I  &= \sqrt{g} G^{JK}\xi_J \mathrm{sign}\,(K,I).
\end{align*}

Now, we define 
$$
\Lambda^k = \Lambda^k(TM,E) = \bigcup_{p\in M} \Lambda^k(T_pM,E),
$$
then a differential $k$-form $\omega$ is a section of $\Lambda^k$, that is a mapping from $M$ to $\Lambda^k$ such that $\omega(p)\in \Lambda^k(T_pM,E)$ for all $p\in M$. In the paper we identify the target space $E$ with $\mathbb{R}^N$. In local coordinates $x^1,\ldots,x^n$, an $\mathbb{R}^n$-valued $k$-form $\omega$ is represented as 
$$
\sum_{I \in \mathcal{T}^k}\omega_I dx^I, \quad \omega_I \in \mathbb{R}^N,
$$
and under the coordinate change $x=x(y)$ these coefficients change as 
$$
\omega_{i_1\ldots i_k}'(y) = \frac{\partial (x^{j_1},\ldots,x^{j_k})}{\partial (y^{i_1},\ldots,y^{i_k})} \omega_{j_1\ldots j_k}(x).
$$

\subsection{Morrey and Campanato spaces on Euclidean domains}\label{ssec:MC}

Let $\Omega$ be a Lipschitz domain in $\mathbb{R}^n$ and $dV = dx^1\ldots dx^n$. For $1\leqslant p <  \infty$ and $\lambda \geq 0,$  $\mathrm{L}^{p,\lambda}\left(\Omega;\varLambda^{k}\right) $ stands for the Morrey space of all $\omega \in L^{p}\left(\Omega;\varLambda^{k}\right)$ such that 
		$$ \lVert \omega \rVert_{\mathrm{L}^{p,\lambda}\left(\Omega;\varLambda^{k}\right)}^{p} := \sup_{\substack{ x_{0} \in \overline{\Omega},\\ \rho >0 }} 
		\rho^{-\lambda} \int\limits_{B_{\rho}(x_{0}) \cap \Omega} \lvert \omega \rvert^{p}\, dV < \infty, 
        $$
        endowed with the norm 
		$ \lVert \omega \rVert_{\mathrm{L}^{p,\lambda}\left(\Omega;\varLambda^{k}\right)}$ and $\mathcal{L}^{p,\lambda}\left(\Omega;\varLambda^{k}\right) $ denotes the Campanato space of all $\omega \in L^{p}\left(\Omega;\varLambda^{k}\right)$ such that 
		$$ [\omega ]_{\mathcal{L}^{p,\lambda}\left(\Omega;\varLambda^{k}\right)}^{p} := \sup_{\substack{ x_{0} \in \overline{\Omega},\\  \rho >0 }} 
		\rho^{-\lambda} \int\limits_{B_{\rho}(x_{0}) \cap \Omega} \lvert \omega  - (\omega)_{ \rho , x_{0}}\rvert^{p} \, dV< \infty, 
        $$ 
        where the element $(\omega)_{ \rho , x_{0}}\in \Lambda^k$ is defined by
		$$
        (\displaystyle (\omega)_{ \rho , x_{0}})_I = \frac{1}{\left\lvert B_{\rho}(x_{0}) \cap \Omega \right\rvert}\int\limits_{B_{\rho}(x_{0}) \cap \Omega} \omega_I\, dV ,
        $$  
		endowed with the norm 
		$$ \lVert \omega \rVert_{\mathcal{L}^{p,\lambda}\left(\Omega;\varLambda^{k}\right)} := \lVert  \omega \rVert_{L^{p}(\Omega, \varLambda^{k})} +  
		[\omega ]_{\mathcal{L}^{p,\lambda}\left(\Omega;\varLambda^{k}\right)}$$. 
        
        By definition, $BMO(\Omega;\Lambda^k) = \mathcal{L}^{1,n}(\Omega;\Lambda^k)$. The following facts are standard (see \cite{Giaquinta,giaquinta-martinazzi-regularity}): 

        \begin{itemize}
            \item $\mathrm{L}^{p,0} = L^p$ and   $\mathrm{L}^{p,\lambda} =\{0\}$ for $\lambda >n$,
            \item $\mathrm{L}^{p,\lambda} \simeq\mathcal{L}^{p,\lambda}$ for $\lambda \in [0,n)$, 
            \item $\mathrm{L}^{p,n} \simeq L^\infty$ and $\mathcal{L}^{p,n}=BMO$, 
            \item $\mathcal{L}^{p,\lambda} \simeq C^{(\lambda-n)/p}$ for $n<\lambda\leq n+p$, and $\mathcal{L}^{p,\lambda} $ is the space of constant functions for $\lambda>n+p$.
        \end{itemize}

We also use the space $BMO(M;\Lambda^k)$, which is the set of forms such that in any coordinate system $(U,\varphi)$ they belong
to $BMO (\varphi(U);\Lambda^k)$.  

\subsection{Div-curl systems}\label{Div_curl_systems}

We use the following technical result for div-curl systems. In \cite{Balci_Sil_Surnachev_arxiv2025} this result is obtained for $N=1$, the generalization for vector-valued forms is straightforward. For the constant exponent case (or in $C^{m,\alpha}$ spaces) these facts are more or less standard, see \cite{Kre70,Kress,bolikphd,Bolik97,Bolik01,Bolik04,
Bolik07}.
    
    \begin{theorem}\label{divcurl system}
		Let $M$ be a compact $n$-dimensional $C^{1,1}$ Riemannian manifold with boundary and let \begin{align*}
			p \in \mathcal{P}^{\log} \left( M \right) \text{ and } 1 < p^{-}_{M} \leq p^{+}_{M} < \infty. 
		\end{align*} 
		Let $u_{0} \in W^{1, p(\cdot)}\left(M; \varLambda^{k}\right)$, 
		$f \in L^{p(\cdot)}\left( M; \varLambda^{k+1}\right)$, and $w \in L^{p(\cdot)}\left( M; \varLambda^{k-1}\right)$.
		Then the following hold true.\smallskip
		
		\noindent \textbf{(i)} Suppose $f$ and $w$  satisfy $df = 0$, $d^{\ast} w = 0$ in $M$ and 
		$ \nu\wedge (d u_{0}-f) =0$ on $bM$, and for every $\chi \in \mathcal{H}_T(M;\varLambda^{k+1})$ and $\psi \in \mathcal{H}_T(M;\varLambda^{k-1})$,
		\begin{equation*}
			\int\limits_{M} \langle f ; \chi \rangle\, dV - \int\limits_{bM} \langle \nu \wedge u_0 ; \chi \rangle\, d\sigma = 0 
			\qquad \text{ and } \qquad \int\limits_{M} \langle w ; \psi \rangle\, dV = 0. 
		\end{equation*}
		Then there exists a unique solution $u \in W^{1, p(\cdot)}\left( M; \varLambda^{k}\right)\cap \left( \mathcal{H}_{T}\left(  M;\varLambda^{k}\right)\right)^{\perp}$ to the boundary value problem 
		\begin{equation} \label{problemddeltalinear}
			\left\lbrace \begin{aligned}
				du = f  \quad &\text{and} \quad  d^{\ast} u = w &&\text{ in } M, \\
				\nu\wedge u &= \nu\wedge u_0 &&\text{  on } bM,
			\end{aligned} 
			\right. \tag{$\mathcal{P}_{T}$}
		\end{equation}
		satisfying the estimates 
		\begin{equation}\label{solest}
			\left\lVert u \right\rVert_{W^{1, p(\cdot)}\left( M; \varLambda^{k}\right)} \leq c \left(  \left\lVert f\right\rVert_{L^{p(\cdot)}(M;\Lambda^{k+1})}  + \left\lVert w\right\rVert_{L^{p(\cdot)}(M;\Lambda^{k-1})} + \left\lVert  u_{0}\right\rVert_{W^{1, p(\cdot)}(M;\Lambda^k)}\right), 
		\end{equation} \smallskip
        where the constant $c$ depends only on $n$, $N$, $p_M^+$, $p_M^{-}$, $M$, and on the constant $c_1$ in the log-H\"older condition in Definition~\ref{def:1}.
        
		\noindent\textbf{(ii)} Suppose $f$ and $w$  satisfy $df = 0$, $d^{\ast} w = 0$ in $M$ and $ \nu\lrcorner (d^\ast u_0-w)  = 0$ on $bM$, and for every $\chi \in \mathcal{H}_N(M;\varLambda^{k-1})$ and $\psi \in \mathcal{H}_N(M;\varLambda^{k+1})$,
		\begin{equation*}
			\int\limits_{M} \langle w ; \chi \rangle\, dV - \int\limits_{bM} \langle \nu \lrcorner u_0 ; \chi \rangle\, d\sigma = 0 
			\qquad \text{ and } \qquad \int\limits_{M} \langle f ; \psi \rangle\, dV = 0. 
		\end{equation*}
		Then there exists a unique solution $u \in W^{1, p(\cdot)}\left(M; \varLambda^{k}\right)\cap \left( 	\mathcal{H}_{N}\left( M;\varLambda^{k}\right)\right)^{\perp}$ to the boundary value problem
		\begin{equation} \label{problemddeltalinearnormal}
			\left\lbrace \begin{aligned}
				du = f  \quad &\text{and} \quad  d^{\ast} u = w &&\text{ in } M, \\
				\nu\lrcorner u &= \nu\lrcorner u_0 &&\text{  on } bM,
			\end{aligned} 
			\right. \tag{$\mathcal{P}_{N}$}
		\end{equation}
		satisfying the estimate \eqref{solest} with $c=c(n,N,p_\Omega^+, p_\Omega^-,M, c_1)$.
	\end{theorem}

Here the notation $W^{1, p(\cdot)}\left( M; \varLambda^{k}\right)\cap \left( \mathcal{H}_{T}\left( M;\varLambda^{k}\right)\right)^{\perp}$ denotes the subspace of $W^{1, p(\cdot)}\left( M; \varLambda^{k}\right)$ defined by the condition
$$
\int\limits_M \langle u,\psi \rangle\, dV  =0 \quad \text{for all}\quad \psi \in \mathcal{H}_{T}\left(  M;\varLambda^{k}\right).
$$
This is well defined since elements of $\mathcal{H}_{T}\left( M;\varLambda^{k}\right)$ belong at least to $W^{1,q}(\Omega; \varLambda^{k})$ for all $q<\infty$, and thus are (H\"older) continuous on $M$. The notation $\left( \mathcal{H}_{N}\left( M;\varLambda^{k}\right)\right)^{\perp}$ is understood in the similar fashion. 
    Recall that if $M$ is contractible, the space $\mathcal{H}_T(M;\Lambda^k)=\{0\}$  for all $k=0,\ldots,n-1$ and $\mathcal{H}_T(M;\Lambda^n)=\{c dV, \ c\in\mathbb{R}^N\}$, while the space  $\mathcal{H}_N(M;\Lambda^k)=\{0\}$  for all $k=1,\ldots,n$ and $\mathcal{H}_N(M;\Lambda^0)=\{c, \ c\in\mathbb{R}^N\}$. 
    
    We shall use Theorem~\ref{divcurl system} for balls only (but with a nonconstant metric). In this case by using a simple scaling argument we have  the following result where the scalar product, volume form, $d_0^*$, and $\lrcorner_0$ correspond to the standard Euclidean metric.


\begin{theorem}\label{T:dcb}
Let $p\in \mathcal{P}^{log}(B_R)$, $1<p_{B_R}^{-}<p_{B_R}^+<\infty$, $u_0\in W^{1,p(\cdot)}(B_R;\Lambda^k)$, $f\in L^{p(\cdot)}(B_R;\Lambda^{k+1})$, $w\in L^{p(\cdot)}(B_R;\Lambda^{k-1})$ satisfy $df=0$ and $d_0^* w=0$ in $B_R$ and $\nu \wedge (du_0-f) =0$ (resp. $\nu \lrcorner (d_0^*u_0 -w)=0$) on $\partial B_R$. If $k=n-1$ (resp. $k=1$) let additionally 
$$
\int\limits_{B_R} f = \int\limits_{\partial B_R} u_0 \quad \biggl(\text{resp.}\quad \int\limits_{B_R} w\, dV = \int\limits_{\partial B_R} \langle \nu,u_0\rangle\, d\sigma \biggr).
$$    
Then there exists a unique $u\in W^{1,p(\cdot)}(B_R;\Lambda^k)$ satisfying $du=f$ and $d_0^* u=w$ in $B_R$, $\nu\wedge u =\nu \wedge u_0$ (resp. $\nu \lrcorner_0 u = \nu\lrcorner_0 u_0$) on $\partial B_R$, such that 
$$
\text{for $k=n$ (resp. $k=0$)} \quad \int\limits_{B_R} u=0, \quad \biggl(\text{resp.}\quad \int\limits_{B_R} u\, dV=0 \biggr).
$$  
Moreover there holds 
\begin{align*}
\|\nabla u\|_{L^{p(\cdot)}(B_R;\Lambda^k)} + R^{-1}& \|u\|_{L^{p(\cdot)}(B_R;\Lambda^k)} \\&
\leq C (\|f\|_{L^{p(\cdot)}(B_R;\Lambda^{k+1})} + \|w\|_{L^{p(\cdot)}(B_R;\Lambda^{k-1})})\\&\qquad \qquad 
+ C(\|\nabla u_0\|_{L^{p(\cdot)}(B_R;\Lambda^k)} +R^{-1}\|u_0\|_{L^{p(\cdot)}(B_R;\Lambda^k)}  )
\end{align*}
with a constant $C$ depending only on $\sup\limits_{0<r<1/2} \Theta_p(rR) \log r^{-1}$, where $\Theta_p$ stands for the modulus of continuity of $p$, and on the numbers $n$, $p_{B_R}^{-}$, $p_{B_R}^+$.   
\end{theorem}

\begin{proof}
Let $T$ be the resolvent operator for the boundary value problem \eqref{problemddeltalinear} (resp. \eqref{problemddeltalinearnormal}) for $R=1$. For $u=T(f,w,u_0)$ we have the estimate 
\begin{align*}
\|\nabla u\|_{L^{p(\cdot)}(B_1;\Lambda^k)} + &\|u\|_{L^{p(\cdot)}(B_1;\Lambda^k)} \\&
\leq C (\|f\|_{L^{p(\cdot)}(B_1;\Lambda^{k+1})} + \|w\|_{L^{p(\cdot)}(B_1;\Lambda^{k-1})})\\&\qquad\qquad  
+ C(\|\nabla u_0\|_{L^{p(\cdot)}(B_1;\Lambda^k)} +\|u_0\|_{L^{p(\cdot)}(B_1;\Lambda^k)}  )
\end{align*}
with the constant which depends only on $n$, $p_{B_1}^+$, $p_{B_1}^{-}$, and $\sup_{0<r<1/2} \Theta_p(r) \log r^{-1}$. 
Then the resolvent operator in the ball $B_R$ can be constructed as $ ((\varphi_R)^*)^{-1} T (\varphi_R)^* $ where $\varphi_R (x) = Rx$.  It remains to note that $c_{\mathrm{log}}(p(R\cdot ))\leq c_{\mathrm{log}} (p)$ for $R\leq 1$.
\end{proof}

We shall use a simple variant of this statement in half-balls.
\begin{lemma}\label{L:Gauge_Tan}
Let $p\in \mathcal{P}^{log}(B_R)$, $1<p_{B_R}^{-}<p_{B_R}^+<\infty$. Let $f\in L^{p(\cdot)}(B_R^+)$ satisfy $df=0$ and $tf=0$ on $x^n=0$. Then there exists $u\in W^{1,p(\cdot)}(B_R^+)$ such that $du=f$ in $B_R^+$ and $tu=0$ on $x^n=0$, and 
$$
\|\nabla u\|_{L^{p(\cdot)}(B_R^+;\Lambda^k)} + R^{-1} \|u\|_{L^{p(\cdot)}(B_R^+;\Lambda^k)} 
\leq \|f\|_{L^{p(\cdot)}(B_R;\Lambda^{k+1})}.
$$
\end{lemma}
\begin{proof}
Let $S(x',x^n) =(x',-x^n)$ be the reflection operator across $x^n=0$. Extend $f$ to $B_R^-$ by $-S^*f$. That is, $f_I(x',x^n) = -f_I(x',-x^n)$ if $n\notin I$ and $f_I(x',x^n) = f_I(x',-x^n)$ if $n\in I$. Then $df=0$ in $B_R$. Indeed, for any $\xi\in C_0^\infty(B_R; \Lambda^{n-k-2})$ we have
$$
\int\limits_{B_R} f\wedge d\xi = \int\limits_{B_R^+} f \wedge  d(\xi+S^*\xi) =0.   
$$
Let $u$ be a unique solution of the boundary value problem 
$$
d u = f\quad \text{in}\quad B_R, \quad \nu \lrcorner u =0 \quad \text{on}\quad \partial B_R
$$
provided by the previous theorem. It is easy to see that the form $\widetilde u$ obtained from the relations $\widetilde u_I (x',x^n) = \pm\widetilde u_I (x',-x^n)$ with $+$ if $n\in I$ and $-$ if $n\notin I$, solves the same boundary value problem. Therefore, $u=\widetilde u$. This immediately implies that $tu =0$ on $x^n=0$. The estimate follows then from the previous theorem
\end{proof}

This can be also extended by using homotopy operator of Poincar\'{e} or Bogovskii type \cite{DieRuz2003},  \cite{CostabelMcIntosh}, \cite{MitreaMonniaux2008}. 

\subsection{Algebraic inequalities}

We recall some well-known algebraic inequalities for vectors. Let $\xi,\eta \in \mathbb{R}^m$, $m\in \mathbb{N}$, and let $\langle \xi,\eta \rangle$, be the scalar product of $\xi$ and $\eta$, that is $\langle \xi,\eta \rangle = g_{ij}\xi^i\eta^j$ with a symmetric positive definite metric tensor $g_{ij}$. Our aim in this section is to demonstrate the stability of standard algebraic inequalities with respect to $p$. For $p\geq 2$ there holds
\begin{align}\label{alg1}
\begin{aligned}
&\langle |\xi|^{p-2}\xi - |\eta|^{p-2}\eta,\xi-\eta\rangle\\
&\qquad \quad= \frac{|\xi|^{p-2}+|\eta|^{p-2}}{2}|\xi-\eta|^2 + \frac{|\xi|^{p-2}-|\eta|^{p-2}}{2}(|\xi|^2-|\eta|^2)\\
&\qquad \quad\geq 2^{-1} (|\xi|^{p-2} + |\eta|^{p-2}) |\xi-\eta|^2\\
&\qquad \quad\geq 2^{-p/2}(|\xi|^2 + |\eta|^2)^\frac{p-2}{2} |\xi-\eta|^2 \geq 2^{1-p} |\xi-\eta|^p.
\end{aligned}
\end{align}
For $1<p\leq 2$, using the relation 
$$
\langle D(|\xi|^{p-2}\xi)[\eta] ,\eta\rangle= \langle|\xi|^{p-2}\eta+ (p-2)|\xi|^{p-4}(\xi,\eta)\xi,\eta\rangle \geq (p-1)|\xi|^{p-2}|\eta|^2
$$
and the Newton-Leibnitz formula, we easily deduce
\begin{align}\label{alg2}
\begin{aligned}
\langle|\xi|^{p-2}\xi - |\eta|^{p-2}\eta,\xi-\eta\rangle
&\geq (p-1) |\xi-\eta|^2\int\limits_0^1 |\eta+t(\xi-\eta)|^{p-2} \, dt\\
&\geq (p-1)(|\xi|^2 + |\eta|^2)^\frac{p-2}{2} |\xi-\eta|^2.
\end{aligned}
\end{align}

Now, for $p\geq 2$ there holds
\begin{align}\label{mu1}
\begin{aligned}
&\bigl\langle (\mu^2+|\xi|^2)^\frac{p-2}{2}\xi - (\mu^2+|\eta|^2)^\frac{p-2}{2}\eta,\xi-\eta\bigr\rangle \\
&\qquad \qquad \geq 2^{-p/2} (2\mu^2 + |\xi|^2+|\eta|^2)^\frac{p-2}{2}|\xi-\eta|^2\\
&\qquad \qquad \geq 2^{-p} |\xi-\eta|^p + \frac{1}{4}\mu^{p-2}|\xi-\eta|^2.
\end{aligned}
\end{align}
To see this, consider the vectors $\widetilde\xi=(\xi,\mu)$, $\widetilde\eta = (\eta,\mu)$ in $\mathbb{R}^{m+1}$. Then \eqref{mu1} reduces to \eqref{alg1} for $\widetilde{\xi}$ and $\widetilde{\eta}$. Similarly, if $1<p<2$ then 
\begin{multline}\label{mu2}
\langle (\mu^2+|\xi|^2)^\frac{p-2}{2}\xi - (\mu^2+|\eta|^2)^\frac{p-2}{2}\eta,\xi-\eta\rangle \\
\geq (p-1)(2\mu^2+|\xi|^2 + |\eta|^2)^\frac{p-2}{2} |\xi-\eta|^2.
\end{multline}
As a corollary, for $2\leq p \leq p^{+}$ we have
$$
|\xi-\eta|^\frac{p+2}{2} \leq 2^{p^{+}/2} \bigl\langle (\mu^2+|\xi|^2)^\frac{p-2}{4}\xi - (\mu^2+|\eta|^2)^\frac{p-2}{4}\eta, \xi-\eta \bigr\rangle,
$$
and so 
\begin{equation}\label{mu3}
|\xi -\eta|^p \leq 2^{p^{+}} \bigl|(\mu^2+|\xi|^2)^\frac{p-2}{4}\xi - (\mu^2+|\eta|^2)^\frac{p-2}{4}\eta\bigr|^2.
\end{equation}
And for $1<p^{-} \leq p <2$ using \eqref{mu2} we have 
\begin{align*}
&|\xi-\eta|^2 \leq \frac{2}{p^{-}} \bigl\langle (\mu^2+|\xi|^2)^\frac{p-2}{4}\xi - (\mu^2+|\eta|^2)^\frac{p-2}{4}\eta, \xi-\eta \bigr\rangle (\mu^2 + |\xi|^2 + |\eta|^2 )^\frac{2-p}{4},    
\end{align*}
which implies that 
\begin{align}\label{mu4}
\begin{aligned}
 &|\xi-\eta|^p \leq c(p^{-})\bigl|(\mu^2+|\xi|^2)^\frac{p-2}{4}\xi - (\mu^2+|\eta|^2)^\frac{p-2}{4}\eta\bigr|^p (\mu^p + |\xi|^p + |\eta|^p)^\frac{2-p}{2}.   
\end{aligned}
 \end{align}

We shall use the following two simple algebraic lemmas.

\begin{lemma}\label{algebraic_lemma}
		Let $m \in \mathbb{N}$. For any $1 < p < \infty$, there exists a constant $c=c\left(p\right) >0$ such that for any $\xi, \eta \in \mathbb{R}^{m}$, we have 
		\begin{align}\label{comparison algebraic ineq}
			\lvert \xi \rvert^{p} \leq c \lvert \eta \rvert^{p} + c 	\left( \left\lvert \xi \right\rvert^{2} + 	\left\lvert \eta \right\rvert^{2} \right)^{\frac{p-2}{2}}\left\lvert \xi - \eta \right\rvert^{2}.
		\end{align}
        Moreover, $c(p)\leq 2^{p+2}$.
	\end{lemma}
\begin{proof}
Follows easily by considering the two cases: $|\xi|\leq 2 |\eta|$ and $|\xi|>2|\eta|$. In the first case $|\xi|^p \leq 2^p |\eta|^p$, and in the second case 
$$
(|\xi|^p+|\eta|^2)^{(p-2)/2} |\xi-\eta|^2 \geq \frac{|\xi|^2}{4(|\xi|^2+|\eta|^2)} (|\xi|^2+|\eta|^2)^{p/2} \geq \frac{1}{5} |\xi|^p. 
$$
Thus the statement of the lemma is valid for any $c(p)\geq \max (2^p,5)$.
\end{proof}

In particular, for any $p\in [p^{-}_\Omega,p^{+}_\Omega]$ the statement of Lemma~\ref{algebraic_lemma} is valid with $c(p) = c(p^{+}_\Omega)$. 

\begin{lemma}\label{algebraic_lemma_1}
		Let $m \in \mathbb{N}$ and $\mu \in \mathbb{R}$. For any $1 < p < \infty$, there exists a constant $c=c\left(p\right) >0$ such that for any $\xi, \eta \in \mathbb{R}^{m}$, we have 
		\begin{align}\label{comparison algebraic ineq_1}
			\left\lvert \xi \right\rvert^{p} \leq c\mu^p+c \left\lvert \eta \right\rvert^{p} + c 	\left(\mu^2+ \left\lvert \xi \right\rvert^{2} + 	\left\lvert \eta \right\rvert^{2} \right)^{\frac{p-2}{2}}\left\lvert \xi - \eta \right\rvert^{2}.
		\end{align}
        Moreover, $c(p)\leq 2^{2p+2}$.
	\end{lemma}
\begin{proof}
It suffices to consider the two vectors $(\xi,2^{-1/2}\mu)$ and $(\eta,2^{-1/2}\mu)$ in $\mathbb{R}^{m+1}$ and to apply the previous estimate.
\end{proof}

We shall use Lemma~\ref{algebraic_lemma} for $p\geq 2$ and Lemma~\ref{algebraic_lemma_1} for $p<2$, and in the latter case we can use the universal constant $c(p)=c(2)=20$.

From inequalities \eqref{alg1} and \eqref{comparison algebraic ineq} for $p\geq 2$ and \eqref{alg2} and \eqref{comparison algebraic ineq_1} for $1<p<2$ we infer the following
\begin{lemma}\label{lemma_algebra}
For any $1<p^{-}\leq p^{+}< \infty$ there exists a constant $C=C(p^{-},p^{+})$ such that  for any $p \in [p^{-},p^{+}]$, any $\mu \in \mathbb{R}$ and any  $\xi,\eta \in \mathbb{R}^m$, $m\in \mathbb{N}$, there holds
$$
|\xi|^p \leq C \bigl(\mu^p+ |\eta|^p +  \bigl\langle (\mu^2+|\xi|^2)^\frac{p-2}{2}\xi - (\mu^2+|\eta|^2)^\frac{p-2}{2}\eta,\xi-\eta\bigr\rangle \bigr)
$$
\end{lemma}

\subsection{Uniform convexity}

We start with the following estimate.
\begin{lemma}\label{uniform_convexity0}
Let $\phi(t)=(\mu^2+t^2)^{p/2}$, $\mu \in \mathbb{R}$, $1<p^{-}<p<p^{+}<\infty$. For any $\varepsilon>0$ there exists $\delta=\delta(\varepsilon,p^{-},p^{+})$ such that for any $u,v >0$ either $|u-v|\leq \varepsilon \max (u,v)$ or $\phi((u+v)/2) \leq (1-\delta) (\phi(u)+\phi(v))/2$. 
\end{lemma}
\begin{proof}
Let $0<u<v$, $0<\varepsilon<1$, and $|u-v|> \varepsilon \max (u,v)$, so $0<u< (1-\varepsilon)v$. For $\delta>0$ consider the function 
$$
\Phi(u,v)=\phi((u+v)/2) - \frac{1-\delta}{2} (\phi(u)+\phi(v)).
$$
From the convexity of the function $\phi(t)$ we have
\begin{align*}
\Phi'_u(u,v) = \frac{p}{2} \left[\frac{u+v}{2} (\mu^2 +((u+v)/2)^2)^\frac{p-2}{2}- (1-\delta) u(\mu^2+ u^2)^\frac{p-2}{2} \right]> 0 
\end{align*}
and similarly $\Phi'_v(u,v)>0$. So it is sufficient to set $u=(1-\varepsilon)v$ and find $\delta$ such that 
$$
\lim_{v\to \infty} \Phi((1-\varepsilon)v,v) \leq 0.
$$ 
But this is clearly so if 
$$
(1-\varepsilon/2)^p \leq (1-\delta)\frac{1 + (1-\varepsilon)^p}{2}.
$$
The existence of such $\delta=\delta (\varepsilon,p^{-},p^{+})>0$ is established by direct computation (see for instance \cite[Chapter 2, Section 2.4, Remark 2.4.6]{Diening_et_al_variable_exponent}) or using the Clarkson inequalities. 

Assume without loss that $p^{-}\leq 2$ and $p^{+}\geq 2$. For $p\geq 2$ and $0\leq x \leq 1$ there holds (the first Clarkson inequality)
$$
\left(\frac{1+x}{2} \right)^p + \left(\frac{1-x}{2} \right)^p \leq \frac{1+x^p}{2},
$$
so setting $x=1-\varepsilon$ we obtain the required estimate with 
$$
\delta(\varepsilon) = \left(\frac{\varepsilon}{2}\right)^p \frac{2}{1+(1-\varepsilon)^p} \geq \left(\frac{\varepsilon}{2}\right)^{p^{+}}.
$$
For $1<p\leq 2$ we have (the second Clarkson inequality)
$$
\left[\left(\frac{1+x}{2} \right)^\frac{p}{p-1} + \left(\frac{1-x}{2} \right)^\frac{p}{p-1}\right]^{p-1} \leq \frac{1+x^p}{2}.
$$
Therefore, setting again $x=1-\varepsilon$ we obtain
$$
\left(1-\frac{\varepsilon}{2} \right)^p \leq (1-\delta(\varepsilon)) \frac{1+(1-\varepsilon)^p}{2}
$$
with 
\begin{align*}
\delta(\varepsilon) &= 1- \left[1+ \left(\frac{\varepsilon}{2-\varepsilon} \right)^\frac{p}{p-1} \right]^{1-p} \\&\geq (1-2^{1-p}) \left(\frac{\varepsilon}{2-\varepsilon} \right)^\frac{p}{p-1} 
\geq (1-2^{1-p^{-}}) \left(\frac{\varepsilon}{2} \right)^\frac{p^{-}}{p^{-}-1},
\end{align*}
The proof of Lemma~\ref{uniform_convexity0} is complete.
\end{proof}

As an immediate corollary we obtain the following
\begin{lemma}\label{uniform_convexity1}
Let $p:\Omega \to [p^{-}_\Omega,p^+_\Omega]$, $1<p^{-}_\Omega\leq p^{+}_\Omega<\infty$, be measurable. The function $\varphi:\Omega\times[0,\infty]\to[0,\infty]$ defined by
$$
\varphi(x,t) = (\mu(x)^2 + t^2)^{p(x)/2}-(\mu(x))^{p(x)}.
$$
is a uniformly convex generalized $N$-function. 
\end{lemma}
The reader can find the necessary definitions and further properties of uniformly convex $N$-functions and corresponding  semimodulars in \cite[Section 2.4]{Diening_et_al_variable_exponent}.

\subsection{Uhlenbeck estimates}\label{append_uhl}

Here we show how to derive estimates \eqref{sup_est_hom} and \eqref{hamburger_osc_use} from the results of \cite{hamburgerregularity}, where, following \cite{uhlenbecknonlinearelliptic}, a system of the general form
$$
d^* (\rho(|\omega|^2) \omega)=0, \quad d\omega =0,
$$
is studied.

We only have to check the conditions of \cite{uhlenbecknonlinearelliptic} and \cite{hamburgerregularity}. Consider the function $\rho(Q) = (\mu^2 +Q)^\frac{p-2}{2}$, where $Q = |\omega|^2$. Clearly, $\rho(Q) \omega = D f(\omega)$, $f(\omega) = p^{-1} (\mu^2+|\omega|^2)^{p/2}$. This constitutes Hypothesis H1 of \cite{hamburgerregularity} (with $\lambda=\Lambda =1$, $a=0$). Now, the function $\rho(Q)$ satisfies 
$$
\rho(Q) + 2Q \rho'(Q) = \frac{\mu^2 + (p-1)Q}{\mu^2 + Q} (\mu^2 + Q)^\frac{p-2}{2}
$$
and so 
$$
(\mu^2+Q)^\frac{p-2}{2} \leq \rho(Q) + 2Q \rho'(Q) \leq (p-1) (\mu^2+Q)^\frac{p-2}{2}.
$$
This gives Hypothesis H2($\rho$) of \cite{hamburgerregularity} (with $c = \max(p-1,1)$).
We also have 
\begin{align*}
|\rho'(Q_1)Q_1 - \rho'(Q_2)Q_2| &= \frac{|p-2|}{2}\left|Q_1 (\mu^2+Q_1)^\frac{p-4}{2} - Q_2(\mu^2+Q_2)^\frac{p-4}{2} \right|\\
&=\frac{|p-2|}{2} \biggl|\int\limits_{Q_1}^{Q_2} (\mu^2 + Q)^\frac{p-6}{2} (\mu^2 + Q (p-2)/2 )\, dQ\biggr| \\
&\leq \frac{p|p-2|}{4} (\mu^2 + Q_1 + Q_2)^\frac{p-4}{2} |Q_1-Q_2|.
\end{align*}
This is Hypothesis H3($\rho$) of \cite{hamburgerregularity} (with $\alpha=2$, $c=p|p-2|/4$). 
There also holds ((1.29) in \cite{hamburgerregularity} with $c=p-1$).
$$
|Q\rho'(Q)| = \frac{|p-2|}{2}Q (\mu^2 + Q )^\frac{p-4}{2} \leq \frac{|p-2|}{2} (\mu^2+Q)^\frac{p-2}{2}. 
$$
Now  let 
$$
H(\omega) = (\mu^2+|\omega|^2)^\frac{p}{2}, \quad \mathcal{V}(\omega) =(\mu^2+|\omega|^2)^\frac{p-2}{4}\omega.
$$
By \cite[Theorem 4.1]{hamburgerregularity}, if $B_R(x_0) \subset \Omega$ we have 
$$
\sup_{B_{R/2}(x_0)} H(\omega) \leq c(n,N,p^{-},p^{+})\fint\limits_{B_{R}(x_0)} H(\omega)\, dV.
$$
This constitutes \eqref{sup_est_hom}.

The second estimate, stated in the same \cite[Theorem 4.1]{hamburgerregularity} has the following form. Let 
$$
\Phi(x_0,r) = \fint\limits _{B_r(x_0)}  |\mathcal{V}(\omega) - (\mathcal{V}(\omega))_{x_0,r}|^2\, dV
$$
Then 
\begin{equation}\label{hamburger_osc}
\Phi(x_0,\rho) \leq c (\rho/R)^{2\alpha} \Phi(x_0,R).
\end{equation}
where the positive constants $c$ and $\alpha$ depend again only on $n$, $N$, $k$, $p^{-}$, and $p^{+}$. 


Let us show that the estimate \eqref{hamburger_osc_use} follows easily from \eqref{hamburger_osc}. Indeed, if $2\leq p \leq p^{+}$, then from \eqref{mu3} it follows that 
\begin{align*}
     \fint\limits_{B_{\rho}} \big\lvert  \omega  - \left( \omega \right)_{B_{\rho}}\big\rvert^{p}\, dV &\leq  c(p^{+})\fint\limits_{B_{\rho}} \big\lvert \mathcal{V}( \omega ) - \big( \mathcal{V}( \omega )\big)_{\rho}\big\rvert^{2}\, dV \\
     &\leq c  \left(\frac{\rho}{R}\right)^{2\alpha} \fint\limits_{B_{R}} \left\lvert \mathcal{V}( \omega ) - \bigl( \mathcal{V}( \omega )\bigr)_{R}\right\rvert^{2} \,dV\\
     &\leq c  \left(\frac{\rho}{R}\right)^{2\alpha} \fint\limits_{B_{R}} \lvert \mathcal{V}( \omega ) \rvert^{2}\, dV \leq c  \left(\frac{\rho}{R}\right)^{2\alpha}\fint\limits_{B_{R}}   (\mu^p+\lvert\omega\rvert^{p})\, dV. 
    \end{align*}
    For $1<p^{-}\leq  p < 2$, using \eqref{mu4}, the obvious fact that $|\mathcal{V}^{-1}(\xi)|\leq |\xi|^{2/p}$, and the H\"older inequality we have 
    \begin{align*}
       \fint\limits_{B_{\rho}} &\big\lvert  \omega  - \left(  \omega \right)_{\rho}\big\rvert^{p} \, dV \\
       &\leq 2\fint\limits_{B_{\rho}} \big\lvert  \omega  - \mathcal{V}^{-1}\bigl( \bigl(  \mathcal{V}(\omega)\bigr)_{\rho}\bigr)\big\rvert^{p} \, dV \\
       &\leq  c(p^{-}) \fint\limits_{B_{\rho}} \big\lvert \mathcal{V}( \omega) - \bigl( \mathcal{V}(\omega )\bigr)_{\rho}\big\rvert^{p} \bigl(\mu^p+ |\omega|^p + |\mathcal{V}^{-1} ( (\mathcal{V}(\omega))_\rho )|^p \bigr)^\frac{2-p}{2}\, dV \\
       &\leq \biggl(~ \fint\limits_{B_{\rho}} \big\lvert \mathcal{V}(\omega ) - \bigl( \mathcal{V}(\omega)\bigr)_{\rho}\big\rvert^{2}\, dV\biggr)^{\frac{p}{2}} \biggl(~ \fint\limits_{B_{\rho}}( \mu^p + |\omega|^p + | (\mathcal{V}(\omega))_\rho|^2)\, dV\biggr)^{\frac{2-p}{2}} \\
       &\leq c\left(\frac{\rho}{R}\right)^{p\alpha}\biggl(~\fint\limits_{B_{R}}\big\lvert \mathcal{V}(\omega ) - \bigl( \mathcal{V}(\omega)\bigr)_{\rho}\big\rvert^{2}\, dV\biggr)^{\frac{p}{2}} \biggl(~ \fint\limits_{B_{\rho}} (\mu^p + |\omega|^p)\, dV\biggr)^{\frac{2-p}{2}}\\
       & \leq c \left(\frac{\rho}{R}\right)^{p\alpha} \biggl(~\fint\limits_{B_{R}}\big\lvert (\mu^p + |\omega|^p)\, dV\biggr)^{\frac{p}{2}} \biggl(~ \fint\limits_{B_{\rho}} (\mu^p + |\omega|^p)\, dV\biggr)^{\frac{2-p}{2}}\\
       & \leq c \left(\frac{\rho}{R}\right)^{p\alpha} \fint\limits_{B_R} (\mu^p + |\omega|^p)\, dV,
    \end{align*}
where we have used that by the sup estimate \eqref{sup_est_hom} for $\rho < R/2$ we have 
\begin{align*}
   \fint\limits_{B_{\rho}} (\mu^p +|\omega|^p)\, dV \leq c \fint\limits_{B_{R}} (\mu^p +|\omega|^p)\, dV.  
\end{align*}
Thus, \eqref{hamburger_osc_use} is established with setting $\beta = \alpha \min \{1, 2/p\}$.

\subsection{Gehring-type lemma}\label{append:Gehring}

First, we recall the generalized Gehring-type lemma of Giaquinta and Modica in the form, presented in  \cite[Chapter V, Theorem 1.2]{Giaquinta}. Let $f,g$ be two nonnegative functions on $Q_1(0)$ and $g\in L^q(Q_1)(0)$, $q>1$, $f\in L^r(Q_1(0))$, $r>q$. Let $d(x)=\mathrm{dist}(x,\partial Q_1(0))$ and for a nonnegative function $h\in L^1_{\mathrm{loc}}(\mathbb{R}^n)$ denote
$$
M_{R_0} (h) (x ) = \sup_{R<R_0} \fint\limits_{B_R(x)} h\, dV.
$$

\begin{lemma}[Giaquinta, Modica]
Suppose almost everywhere on $Q_1(0)$
$$
M_{\frac{d(x)}{m}} (g^q) (x) \leq b M^q(g)(x) + M(f^q)(x) +\theta M(g^q)(x),
$$
where $m\in \mathbb{N}$, $b>1$ and $0\leq \theta <1$. Then $g\in L^p_{\mathrm{loc}}(Q_1(0))$ for $p\in [q,q+\varepsilon)$ and 
$$
\biggl(~\fint\limits_{Q_{1/2}(0)} g^p\, dV\biggr)^\frac{1}{p} \leq c \biggl\{ \biggl(~\fint\limits_{Q_{1}(0)} g^q\, dV\biggr)^\frac{1}{q} + \biggl(~\fint\limits_{Q_{1}(0)} f^p\, dV\biggr)^\frac{1}{p} \biggr\}
$$
where $\varepsilon = \varepsilon (b,\theta,q,n,r,m)$ and $c=c(b,\theta,q,n,m)$ are positive constants.
\end{lemma}
The dependence of $\varepsilon$ on  the parameters has the following form $\varepsilon = \min (\varepsilon_*,p-q)$, where $\varepsilon_*$ varies continuously with respect to the parameters $b$, $\theta$, $q$, $n$. The dependence on $c$ on its parameters is also continuous (see \cite[Theorem 4]{Acerbi_Mingione_Crelle_2005}, \cite[Section 4]{Bojarski_Iwaniec_1983}, \cite[Proposition 6.1]{Iwaniec_1995}). In \cite{Giaquinta} this lemma is stated with $m=2$, the general case is obvious. 

As a corollary, the following statement holds, which we give in the form of \cite[Chapter V, Proposition 1.1]{Giaquinta} with cubes  replaced by balls in the assumption. 
\begin{lemma}[Giaquinta, Modica]\label{lemma:GM}
Let $Q$ be an $n$-cube. Suppose
$$
\fint\limits_{B_R(x_0)} g^q\, dV \leq b \biggl(~\fint\limits_{B_{2R}(x_0)} g\, dV \biggr)^{q} + \fint\limits_{B_{2R}(x_0)} f^q\, dx +\theta \fint\limits_{B_{2R}(x_0)} g^q\, dV
$$
for each $x_0 \in Q$ and each $R< \frac{1}{2} \min (\mathrm{dist}(x_0,\partial Q), R_0)$, where $R_0$, $b$, $q$ are constants with $b>1$, $R_0>0$, $0\leq \theta<1$.Then $g\in L^p_{\mathrm{loc}}(Q)$ for $p\in [q,q+\varepsilon)$ and 
$$
\biggl(~\fint\limits_{Q_{R}(x_0)} g^p\, dV\biggr)^\frac{1}{p} \leq c \biggl\{ \biggl(~\fint\limits_{Q_{2R}(x_0)} g^q\, dV\biggr)^\frac{1}{q} + \biggl(~\fint\limits_{Q_{2R}(x_0)} f^p\, dV\biggr)^\frac{1}{p} \biggr\}
$$
for $Q_{2R}\subset Q$, $R<R_0$, where $c$ and $\varepsilon$ are positive constants depending only on $b$, $\theta$, $q$, $n$ (and $r$).
\end{lemma}

\subsection{Giaquinta--Giusti iteration lemma}\label{append:GG}

We recall the following well-known modification by Acerbi, Mingione \cite{Acerbi_Mingione_ARMA_2001} of the iteration lemma by Giaquinta and Giusti (see \cite[Chapter III, \S 2, Lemma 2.1]{Giaquinta}), where the monotonicity assumption is weakened. 

\begin{lemma} \label{GG} 
Let $\phi(t)$ be a non-negative function satisfying $\phi(s) \leq  M\phi(t)$ for $s\leq t$. Suppose that 
$$ 
\phi(\rho) \leq A \left [ \left(\frac{\rho}{R} \right)^\alpha + \varepsilon \right] \phi(R) + B R^\beta
$$
for all $0<\rho <R\leq R_0$, with $A,\alpha,\beta>0$, $B,\varepsilon\geq 0$ and $\beta <\alpha$. Then there exists a constant $\varepsilon_0 = \varepsilon_0 (A,\alpha,\beta)$ such that if $\varepsilon<\varepsilon_0$ we have 
$$
\phi(\rho) \leq cM \left( \frac{\rho}{R}\right)^\beta [\phi(R) + BR^\beta] 
$$
for all $\rho < R \leq R_0$, where the constant $c$ depends only on $\alpha$, $\beta$, $A$.
\end{lemma}
The choice of constants guaranteed in this lemma is stable: for any compact range of admissible parameters one can choose universal constants $c$ and $\varepsilon_0$.

\bmhead{Acknowledgements}
The research of Anna Balci was supported by  Deutsche Forschungsgemeinschaft (DFG, German Research Foundation) - SFB 1283/2 2021 - 317210226 in Bielefeld university and  by Charles University  PRIMUS/24/SCI/020 and Research Centre program No. UNCE/24/SCI/005. Swarnendu Sil's research  was supported by ANRF-SERB MATRICS grant MTR/2023/000885 and ANRF ARG grant ANRF/ARG/2025/000348/MS.  The research of Mikhail Surnachev was supported by the Russian Science Foundation (project No. 25-71-30001). 
\bibliography{ref} 
\end{document}